\documentclass[10pt]{amsart}

\usepackage{fullpage}
\usepackage{amssymb}
\usepackage{amsmath}
\usepackage{amsxtra}
\usepackage{amscd}
\usepackage{graphicx}
\usepackage{amsfonts}
\usepackage{pb-diagram}
\usepackage{float}
\usepackage{enumerate}
\usepackage{dsfont}

\numberwithin{equation}{section}
\newtheorem{thm}{Theorem}[section]
\newtheorem{lem}[thm]{Lemma}

\newtheorem{prop}[thm]{Proposition}

\newtheorem{defn}[thm]{Definition}

\newtheorem{exmp}[thm]{Example}

\newtheorem{rems}[thm]{Remarks}
\newtheorem{rem}[thm]{Remark}

\def\C{\mathbb{C}}
\def\Z{\mathbb{Z}}

\def\og{\mathsf{g}}

\def\cR{\mathcal{R}}
\def\cF{\mathcal{F}}

\newcommand{\blambda}{\boldsymbol{\lambda}}

\newcommand{\bmu}{\boldsymbol{\mu}}
\newcommand{\btau}{\boldsymbol{\tau}}

\newcommand{\btheta}{\boldsymbol{\theta}}

\newcommand{\cc}{\mathrm{c}}
\newcommand{\ccc}{\mathrm{cc}}
\newcommand{\pos}{\mathrm{p}}
\newcommand{\posd}{\mathrm{p}^{(d)}}
\newcommand{\posm}{\mathrm{p}^{(m)}}
\newcommand{\gad}{\Gamma_d}
\def\cT{\mathcal{T}}
\def\bv{\mathbf{v}}
\def\tX{\widetilde{X}}
\def\tr{\mathrm{tr}}

\def\og{\overline{g}}
\def\oX{\overline{X}}
\def\oq{\overline{q}}

\begin{document}

\title[Markov traces on affine and cyclotomic Yokonuma--Hecke algebras]
  {Markov traces on affine and cyclotomic\\ Yokonuma--Hecke algebras}
\author{Maria Chlouveraki}
\address{Laboratoire de Math\'ematiques UVSQ, B\^atiment Fermat, 45 avenue des \'Etats-Unis,  78035 Versailles cedex, France.}
\email{maria.chlouveraki@uvsq.fr}

\author{Lo\"ic Poulain d'Andecy}
\address{Korteweg-de Vries Institute for Mathematics, University of Amsterdam\\
P.O. Box 94248, 1090 GE Amsterdam, The Netherlands.}
\email{L.B.PoulainDAndecy@uva.nl}


\thanks{We are grateful to Sofia Lambropoulou and Jes\'us Juyumaya for their useful suggestions and corrections to this paper. We would also like to thank Vincent S\'echerre for our  fruitful conversations and his interesting remarks concerning the affine case.
The second author was supported by ERC-advanced grant no.~268105. For the first author, this research project is implemented within the framework of the Action ``Supporting Postdoctoral Researchers'' of the Operational Program ``Education and Lifelong Learning'' (Action's Beneficiary: General Secretariat for Research and Technology), and is co-financed by the European Social Fund (ESF) and the Greek State. }


\begin{abstract}
In this article, we define and study the affine and cyclotomic Yokonuma--Hecke algebras. These algebras generalise at the same time the Ariki--Koike and affine Hecke algebras and the Yokonuma--Hecke algebras. We study the representation theory of these algebras and construct several bases for them. We then show how we can define Markov traces on them, which we in turn use to construct invariants for framed and classical knots in the solid torus. Finally, we 
study the Markov trace with zero parameters on the cyclotomic Yokonuma--Hecke algebras and determine the Schur elements with respect to that trace. 
\end{abstract}

\maketitle

\tableofcontents

\section{Introduction}\label{sec-intro}

Ariki--Koike algebras were introduced by Ariki and Koike \cite{ArKo} as generalisations of the Iwahori--Hecke algebras of types $A$ and $B$. Following the definition of Hecke algebras associated with complex reflection groups as quotients of their braid group algebras by Brou\'e, Malle and Rouquier \cite{BMR}, Ariki--Koike algebras ${\rm H}(m,n)$ can be viewed as the Hecke algebras associated with the complex reflection groups of type $G(m,1,n)$. Hence, they are quotients of an affine braid group algebra (of type A) 
and deformations of the group algebra of $G(m,1,n)$. The complex reflection group $G(m,1,n)$ is isomorphic to the wreath product $(\Z/m\Z) \wr \mathfrak{S}_n$, so for $m=1$ and $m=2$ Ariki--Koike algebras are the Iwahori--Hecke algebras of types $A$ and $B$ respectively. Moreover, the Iwahori--Hecke algebra ${\rm H}(n)$ of type $A$ is an obvious subalgebra of the Ariki--Koike algebra ${\rm H}(m,n)$ for any $m$. 

The representation theory of Ariki--Koike algebras has been studied in the original paper by Ariki and Koike, where a basis for these algebras is also given. Other, inductive bases for Ariki--Koike algebras have been given by Lambropoulou in \cite{La2} and Ogievetsky and the second author in \cite{OgPo1}. 

Yokonuma--Hecke algebras were introduced by Yokonuma \cite{yo} in the context of Chevalley groups, also as generalisations of Iwahori--Hecke algebras. More precisely, the Iwahori--Hecke algebra associated to a finite Chevalley group $G$ is the centraliser algebra associated to the permutation representation of $G$ with respect to a Borel subgroup of $G$. The Yokonuma--Hecke algebra is the centraliser algebra associated to the permutation representation of $G$ with respect to a maximal unipotent subgroup of $G$. Thus,  Yokonuma--Hecke algebras can be also regarded as particular cases of unipotent Hecke algebras.

The Yokonuma--Hecke algebra ${\rm Y}(d,n)$ of type $A$ ($G={\rm GL}_n(\mathbb{F}_{q^2})$) is a quotient of the group algebra of the modular framed braid group $(\Z/d\Z) \wr B_n$, where $B_n$ is the classical braid group on $n$ strands (of type $A$). In recent years, the presentation of the algebra ${\rm Y}(d,n)$  has been transformed  by Juyumaya  \cite{ju, juka, ju2} to the one used in this paper. From this presentation, it becomes obvious that the algebra ${\rm Y}(d,n)$ is a deformation of the group algebra of $G(d,1,n)$ which respects the wreath product structure, unlike the Ariki--Koike algebra. For $d=1$, the algebra  ${\rm Y}(1,n)$ coincides with the Iwahori--Hecke algebra $\mathrm{H}(n)$ of type $A$. However, for $d>1$, $\mathrm{H}(n)$ is not an obvious subalgebra of ${\rm Y}(d,n)$.

A basis for the Yokonuma--Hecke algebra  ${\rm Y}(d,n)$ has been constructed by Juyumaya in \cite{ju2}. 
Some information on its representation theory in the general context of unipotent Hecke algebras has been obtained by Thiem in \cite{thi, thi2,thi3}. 
In our previous paper \cite{ChPo}, we developed an inductive, and highly combinatorial, approach to the representation theory of the Yokonuma--Hecke algebra of type $A$. We gave explicit formulas for the representations of  ${\rm Y}(d,n)$, and in order to do this, we introduced and studied, what turned out to be, the affine Yokonuma--Hecke algebra ${\rm Y}(d,\infty,2)$. 

In this paper, we define and study the algebra ${\rm Y}(d,m,n)$, which we call \emph{affine Yokonuma--Hecke algebra} when $m=\infty$ and \emph{cyclotomic Yokonuma--Hecke algebra} when $m \in \Z_{>0}$. This algebra generalises both ${\rm H}(m,n)$ (for $m=\infty$, we consider ${\rm H}(\infty,n)$ to be the affine Hecke algebra of ${\rm GL}_n$) 
and ${\rm Y}(d,n)$, which are quotients of ${\rm Y}(d,m,n)$. Further, for $d=1$, ${\rm Y}(1,m,n)$ coincides with ${\rm H}(m,n)$, while, for $m=1$, ${\rm Y}(d,1,n)$ coincides with 
${\rm Y}(d,n)$. The existence of these algebras has been first mentioned by Juyumaya and Lambropoulou in \cite{jula6}, where they refer to them as ``modular framisations'' of the 
generalised Hecke algebra of type $B$ (which is the affine Hecke algebra of ${\rm GL}_n$) 
and the cyclotomic Hecke algebra respectively. For a complete survey on the framisation of algebras with applications in knot theory, the reader may refer to \cite{jula5}.
Furthermore, the affine Yokonuma--Hecke algebra ${\rm Y}(d,\infty,n)$ can be seen as the centraliser algebra associated to the permutation representation of $p$-adic ${\rm GL}_n$
with respect to a pro-$p$-Iwahori subgroup, which is the pro-$p$-Iwahori Hecke ring  studied by Vign\'eras in \cite{Vi1}.

In the third section of this paper, we give an explicit description, in combinatorial terms, of the irreducible representations of the cyclotomic Yokonuma--Hecke algebra  ${\rm Y}(d,m,n)$ (case $m < \infty$).  The formulas for the action of the generators  generalise and unify two known situations. These two situations are two different generalisations of classical constructions for the Iwahori--Hecke algebra of type $A$. One one hand, for $d>1$, the formulas for ${\rm Y}(d,m,n)$ generalise the formulas obtained in \cite{ArKo} for the Ariki--Koike algebra ${\rm Y}(1,m,n)$ (for $d=1$ and $m=1,2$, the formulas already appear in the classical work of Hoefsmit \cite{Ho}). On the other hand, for $m>1$, we obtain a generalisation of the formulas in \cite{ChPo} for the Yokonuma--Hecke algebra ${\rm Y}(d,1,n)$. 

In the fourth section, we provide several generating sets for both affine and cyclotomic Yokonuma--Hecke algebras. Using  
the knowledge of the dimension of the irreducible representations for finite $m$, we are able to show that  
these spanning sets are bases of 
${\rm Y}(d,m,n)$ for every $m$ (in the affine situation, this is deduced from the results in the cyclotomic case).
One of the bases is the analogue of the Ariki--Koike basis for the Ariki--Koike algebra (and of the Bernstein basis for the affine Hecke algebra). The other ones are inductive bases, so they are well-adapted to the study of the whole chain of cyclotomic, or affine, Yokonuma--Hecke algebras.  They are the analogues of the inductive bases for the Ariki--Koike algebra given in \cite{La2, OgPo1}.
Finally, we use the results of this section to conclude that the representations constructed in Section $3$ form a complete set of pairwise non-isomorphic irreducible representations and to obtain a semisimplicity criterion for the cyclotomic Yokonuma--Hecke algebra.

Having constructed a basis, we proceed in Section $\ref{sec-markov}$ to the definition of a Markov trace on ${\rm Y}(d,m,n)$.    
This definition encompasses the definition of Markov traces both on Ariki--Koike algebras (by Ocneanu/Jones for $m=1$ \cite{Jo}, by Geck and Lambropoulou for $m=2$ 
\cite{La1, Gela},  and by Lambropoulou for $m \geq 3$ \cite{La2}) and on Yokonuma--Hecke algebras of type $A$ (by Juyumaya \cite{ju2}).  The Markov trace is 
obtained as the composition of certain ``relative traces'', whose study yields the existence and uniqueness of the Markov trace
for every choice of parameters. 
These relative traces in the case of  Ariki--Koike algebras and the affine Hecke algebra of ${\rm GL}$
are used to construct commutative Bethe subalgebras \cite{IsOg,IsKi}, which play a fundamental role in the theory of chain or Gaudin models; see, for example, \cite{MTV} for the symmetric group case. For ${\rm H}(m,n)$ (that is, for $d=1$), such relative traces have been explicitly constructed in \cite{OgPo1}. For the Yokonuma--Hecke algebras ($d>1$ and $m=1$), the relative traces provide an alternative approach to the Markov trace defined in \cite{ju2}.

In Section \ref{sec-inv}, we use the Markov trace that we constructed in order to define knot invariants. 
Jones was the first to use a Markov trace to define an invariant for classical knots and links (Jones polynomial, which in turn led to the HOMFLYPT polynomial with the use of the Ocneanu trace), using Alexander's theorem which states that every link can be represented by a braid.    
In \cite{La1} Lambropoulou showed that every link in the solid torus can be represented by an element of the affine braid group. Using the Markov traces constructed in  \cite{La1,Gela,La2}, she defined invariants  for knots and links in the solid torus. 
Finally, Juyumaya and Lambropoulou \cite{jula2} used Juyumaya's trace and the fact that the Yokonuma--Hecke algebra is a quotient of the framed braid group algebra to construct invariants for framed knots and links (which become invariants for classical knots and links if we forget the framings, see \cite{jula3}).
So it is only natural to use the Markov trace on ${\rm Y}(d,m,n)$ to define invariants for framed knots and links in the solid torus
(which become invariants for classical knots and links in the solid torus if we forget the framings).
For this, we prove an analogue of Alexander's theorem for framed knots and links in the solid torus and we impose a certain condition on the parameters of the Markov trace, the \emph{affine E-condition}, which is a generalisation of the E-condition of Juyumaya and Lambropoulou \cite{jula2}.
We note that, following recent results obtained in \cite{CJKL}, the invariants for classical  links from the Yokonuma--Hecke algebra distinguish topologically more links than the HOMFLYPT polynomial. This implies that the invariants we obtain here for classical links in the solid torus provide more topological information than the invariants defined in \cite{La1,Gela,La2}.

Finally, in Section \ref{sec-sym}, we restrict again ourselves to the case where $m < \infty$.  We study a very special Markov trace on ${\rm Y}(d,m,n)$, the one where all parameters are equal to $0$. This Markov trace generalises both the canonical 
symmetrising trace on the Ariki--Koike algebra ${\rm H}(m,n)$ constructed by Bremke and Malle \cite{BM, MaMa, GIM} and the canonical 
symmetrising trace on the Yokonuma--Hecke algebra ${\rm Y}(d,n)$ constructed in \cite{ChPo}. 
We compute the Schur elements for the cyclotomic Yokonuma--Hecke algebra  with respect to this trace by showing that they can be expressed as products of Schur elements 
for Ariki--Koike algebras, which are already known \cite{GIM, Ma, ChJa}.

\subsection*{Notation} We set $E_m:=\{0,\ldots,m-1\}$ for $m\in\Z_{>0}$, and $E_{\infty}:=\Z$.

Let $q$ and $v_a$, $a\in\Z_{>0}$, be indeterminates and set $\cR_m:=\C[q^{\pm1},v_1^{\pm1},\ldots,v_m^{\pm1}]$ for $m\in\Z_{>0}$, and $\cR_{\infty}:=\C[q^{\pm1}]$. We denote by $\cF_m$ the field of fractions of $\cR_m$.

\section{Affine and cyclotomic Yokonuma--Hecke algebras}\label{sec-def}

Let  $d \in\Z_{>0}$, $m\in\Z_{>0}\cup\{\infty\}$ and $n\in \Z_{>0}$.  We denote by ${\rm Y}(d,m,n)$ the associative algebra over $\cR_m$ generated by elements
$$t_1,\ldots,t_n,g_1,\ldots,g_{n-1},X_1^{\pm1}$$
subject to the following defining relations:
\begin{equation}\label{def-aff1}
\begin{array}{rclcl}
g_ig_j & = & g_jg_i && \mbox{for all $i,j=1,\ldots,n-1$ such that $\vert i-j\vert > 1$,}\\[0.1em]
g_ig_{i+1}g_i & = & g_{i+1}g_ig_{i+1} && \mbox{for  all $i=1,\ldots,n-2$,}\\[0.1em]
t_it_j & =  & t_jt_i &&  \mbox{for all $i,j=1,\ldots,n$,}\\[0.1em]
t_jg_i & = & g_it_{s_i(j)} && \mbox{for all $i=1,\ldots,n-1$ and $j=1,\ldots,n$,}\\[0.1em]
t_j^d   & =  &  1 && \mbox{for all $j=1,\ldots,n$,}\\[0.2em]
g_i^2  & = & 1 + (q-q^{-1}) \, e_{i} \, g_i && \mbox{for  all $i=1,\ldots,n-1$,}
\end{array}
\end{equation}
where   $s_i$ is the transposition $(i,i+1)$ and
$$e_i :=\frac{1}{d}\sum\limits_{s=0}^{d-1}t_i^s t_{i+1}^{-s}\,,$$
together with the following relations concerning  the 
generator $X_1$: 
\begin{equation}\label{def-aff2}
\begin{array}{rclcl}
X_1\,g_1X_1g_1 & =  & g_1X_1g_1\,X_1  &&\\[0.1em]
X_1g_i & = & g_iX_1 && \mbox{for all $i=2,\ldots,n-1$,}\\[0.1em]
X_1t_j & = & t_jX_1 && \mbox{for all $j=1,\ldots,n$,}\\[0.1em]
(X_1-v_1)\cdots(X_1-v_m) & = & 0 && \mbox{if $m<\infty$.}
\end{array}
\end{equation}
The algebra ${\rm Y}(d,\infty,n)$ was called in \cite{ChPo} the \emph{affine Yokonuma--Hecke algebra}. For $m<\infty$, we call the algebra ${\rm Y}(d,m,n)$ the \emph{cyclotomic Yokonuma--Hecke algebra}. 
These algebras are isomorphic to the modular framisations of,  respectively, the affine Hecke algebra ($m=\infty$) and the cyclotomic Hecke algebra ($m<\infty$); see definitions in \cite[Section 6]{jula5} and Remark 1 in \cite{ChPo}.   
For any $m$, we define $\cF_m{\rm Y}(d,m,n):=\cF_m\otimes_{\cR_m}{\rm Y}(d,m,n)$.

Note that the elements $e_i$ are idempotents in ${\rm Y}(d,m,n)$,
that we have $g_ie_i=e_ig_i$ for all $i=1,\dots,n-1$, 
and that the elements $g_i$ are invertible, with
\begin{equation}
g_i^{-1} = g_i - (q- q^{-1})\, e_i  \qquad \mbox{for all $i=1,\ldots,n-1$}.
\end{equation}
Moreover, if $m<\infty$, we can write the last relation in (\ref{def-aff2}) as 
\begin{equation}\label{X1m}
X_1^m+\gamma_{m-1}^{(m)}X_1^{m-1}+\dots+\gamma_1^{(m)}X_1+\gamma_0^{(m)}=0\,,
\end{equation}
where $\gamma_0^{(m)},\gamma_1^{(m)},\ldots,\gamma_{m-1}^{(m)}\in\cR_m$. Note that $\gamma_0^{(m)}=(-1)^mv_1\dots v_m$ is invertible in $\cR_m$. Thus, we have
\begin{equation}\label{inv-X1}
X_1^{-1}=-\frac{1}{\gamma_0^{(m)}}\Bigl( X_1^{m-1}+\gamma_{m-1}^{(m)}X_1^{m-2}+\dots + \gamma_2^{(m)}X_1+\gamma_1^{(m)} \Bigr)\in\cR_m[X_1]\,,
\end{equation}
and, in particular, $X_1^{-1}$ can be removed from the set of generators when $m<\infty$.

\vskip .2cm
We recall that the Yokonuma--Hecke algebra, defined by Yokonuma in  \cite{yo}, of type $A$ is the associative algebra over $\C[q^{\pm1}]$ generated by elements $t'_1,\ldots,t'_n,g'_1,\ldots,g'_{n-1}$ with the defining relations as in (\ref{def-aff1}) with $g_i$ replaced by $g'_i$ and $t_i$ replaced by $t'_i$ \cite{ju,juka,ju2}. We denote by ${\rm Y}_m(d,n)$ the Yokonuma--Hecke algebra with the ground ring extended to $\cR_m$.
There is a surjective homomorphism $\pi_{\rm Y}$ from the algebra ${\rm Y}(d,m,n)$ onto the algebra ${\rm Y}_m(d,n)$ given on generators by:
\begin{equation*}\label{surj-pi}
\pi_{\rm Y}(t_j)=t'_j\,,\ j=1,\ldots,n,\quad\pi_{\rm Y}(g_i)=g'_i\,,\ i=1,\ldots,n-1,\quad\text{and}\quad
\pi_{\rm Y}(X_1)=\left\{\begin{array}{ll}\!\!\!v_1& \text{if $m<\infty$}\\
\!\!\!1& \text{if $m=\infty$}\end{array}\right..
\end{equation*}
The fact that $\pi_{\rm Y}$ defines an algebra homomorphism follows from the fact that the relations (\ref{def-aff2}) are trivially satisfied if $X_1$ is replaced by $v_1$ if $m<\infty$ and by $1$ if $m=\infty$.

Now let $\iota_{\rm Y}$ be the morphism from the algebra ${\rm Y}_m(d,n)$ to the algebra ${\rm Y}(d,m,n)$, defined by:
 \begin{equation}\label{inj-i}
 \iota_{\rm Y}(g'_i)=g_i\ \ \ \text{for $i=1,\ldots,n-1$}\quad\ \ \text{and}\ \ \quad  \iota_{\rm Y}(t'_j)=t_j\ \ \ \text{for $j=1,\ldots,n$.}
\end{equation}
The composition $\pi_{\rm Y}\circ\iota_{\rm Y}$ is the identity morphism of ${\rm Y}_m(d,n)$ and this implies that $\iota_{\rm Y}$ is injective. Thus, the subalgebra of ${\rm Y}(d,m,n)$ generated by $t_1,\ldots,t_n,g_1,\ldots,g_{n-1}$ is isomorphic to ${\rm Y}_m(d,n)$. In the paper, we will refer to ${\rm Y}_m(d,n)$ as the Yokonuma--Hecke algebra and implicitly use the isomorphism with the subalgebra of ${\rm Y}(d,m,n)$ generated by $t_1,\ldots,t_n,g_1,\ldots,g_{n-1}$. For $m=1$, the algebra ${\rm Y}(d,1,n)$ coincides with ${\rm Y}_1(d,n)$. 

\begin{rems} \label{rem1}{\rm 
{\bf (a)} The first four relations in (\ref{def-aff1}) are defining relations for the classical framed braid group $\Z\wr B_n$, where $B_n$ is the classical braid group on $n$ strands, with the $t_j$'s being interpreted as the ``elementary framings" (framing 1 on the $j$th strand). The quotient of $\Z\wr B_n$ over the relations $t_j^d = 1$ is the modular framed braid group $(\Z/d\Z) \wr B_n$ (the framing of each braid strand is regarded modulo~$d$). Thus, the Yokonuma--Hecke algebra is the quotient of the modular framed braid group algebra over the last relation in (\ref{def-aff1}).\\
{\bf (b)} The first two relations in (\ref{def-aff1}) together with the first two  relations in (\ref{def-aff2}) are defining relations for the affine braid group $B_n^{\mathrm{aff}}$ (with generators $X_1, g_1,\ldots,g_{n-1}$). Adding the generators $t_1,\ldots,t_n$, we obtain the framed affine braid group $\Z\wr B_n^{\mathrm{aff}}$ by 
 considering as defining relations the first four relations in (\ref{def-aff1}) together with the first three relations in (\ref{def-aff2}). The quotient of $\Z\wr B_n^{\mathrm{aff}}$ over the relations $t_j^d = 1$ is the modular framed affine braid group $(\Z/d\Z) \wr B_n^{\mathrm{aff}}$. Thus, the affine and cyclotomic Yokonuma--Hecke algebras are quotients of the modular framed affine braid group algebra over $\cR_m$.
These braid groups and their connections with knots and links will be described more precisely in Section \ref{sec-inv}. 
\hfill $\triangle$}\end{rems}

We denote by ${\rm H}(m,n)$ the quotient of the algebra ${\rm Y}(d,m,n)$ over the relations $t_j=1$, $j=1,\ldots,n$. If $m<\infty$, the algebra ${\rm H}(m,n)$ is the Ariki--Koike algebra, also called the cyclotomic Hecke algebra of type $G(m,1,n)$, while if $m=\infty$, the algebra ${\rm H}(\infty,n)$ is the affine Hecke algebra of 
${\rm GL}$.  For $d=1$, the algebra ${\rm Y}(1,m,n)$ coincides with ${\rm H}(m,n)$.

Let $\pi_{\rm H}$ be the natural surjective morphism from ${\rm Y}(d,m,n)$ to its quotient ${\rm H}(m,n)$. The image of the subalgebra of ${\rm Y}(d,m,n)$ generated by $t_1,\ldots,t_n,g_1,\ldots,g_{n-1}$ under the map $\pi_{\rm H}$ is denoted by ${\rm H}_m(n)$. It is well-known that the subalgebra ${\rm H}_m(n)$ of ${\rm H}(m,n)$ is isomorphic to the finite Hecke algebra of type $A$ (over the ring $\cR_m$). In other words, the Hecke algebra ${\rm H}_m(n)$ is the quotient of the Yokonuma--Hecke algebra ${\rm Y}_m(d,n)$ by the relations $t_j=1$, $j=1,\ldots,n$. The following diagram summarises the different algebras and their connections (vertical arrows correspond to the projection $\pi_{\rm H}$; horizontal arrows correspond to the injection $\iota_{\rm Y}$):

\begin{equation}\label{diagram}
\begin{array}{ccc} {\rm Y}_m(d,n) & \hookrightarrow & {\rm Y}(d,m,n) \\[0.5em]
\downarrow&&\downarrow\\[0.5em]
{\rm H}_m(n) &  \hookrightarrow  &  {\rm H}(m,n)
\end{array}
\end{equation}

\vskip .1cm
\begin{rems}\label{rem-degeneration}
{\rm {\bf (a)} For $q^2=p^k$, where $p$ is a prime number and $k \in \mathbb{Z}_{>0}$, and $F$ a local non-archimedean field (a finite extension of $\mathbb{Q}_p$ or a field of
Laurent series $\mathbb{F}_{q^2}((t))$) of residual field $\mathbb{F}_{q^2}$, 
the affine Hecke algebra ${\rm H}(\infty,n)$ is the centraliser algebra associated to the permutation representation of ${\rm GL}_n(F)$ with respect to an Iwahori subgroup. 
For $d=p^k-1$, the affine Yokonuma--Hecke algebra ${\rm Y}(d,\infty,n)$ is isomorphic to the pro-$p$-Iwahori Hecke ring of ${\rm GL}_n(F)$, that is, 
the centraliser algebra associated to the permutation representation of ${\rm GL}_n(F)$ with respect to a pro-$p$-Iwahori subgroup.
A presentation of the pro-$p$-Iwahori Hecke ring of ${\rm GL}_n(F)$ by generators and relations, and bases for it as a module over $\mathcal{R}_\infty$ and $\Z$,
are given by Vign\'eras in \cite{Vi1}. A more general definition of a pro-$p$-Iwahori Hecke ring, using generators and relations,  is given and studied by Vign\'eras in her preprints \cite{Vi2,Vi3,Vi4}.
The affine Yokonuma--Hecke algebra ${\rm Y}(d,\infty,n)$ should be a particular case of a pro-$p$-Iwahori Hecke ring.
We are grateful to Vincent S\'echerre for pointing all this out to us.}\\
{\rm {\bf (b)}
The degenerate (or graded) affine Hecke algebra of GL can be obtained as a certain degeneration, when $q$ tends to $\pm1$, of the affine Hecke algebra ${\rm H}(\infty,n)$ \cite{Dr}. In \cite{WaWa}, an analogue of the degenerate affine Hecke algebra of GL is associated to the wreath product of an 
arbitrary finite group $G$ by the symmetric group; see also \cite{RaSh} when $G$ is a cyclic group. This algebra (when $G$ is the cyclic group of order $d$) can be seen as a degeneration, when $q$ tends to $\pm1$, of the affine Yokonuma--Hecke algebra ${\rm Y}(d,\infty,n)$; see \cite[Remark 2]{ChPo}.
\hfill $\triangle$}\end{rems}

Let $w \in \mathfrak{S}_n$, where $\mathfrak{S}_n$ is the symmetric group on $n$ letters, and let $w=s_{i_1}s_{i_2}\ldots s_{i_r}$ be a reduced expression for $w$ (where $s_i$ denotes the transposition $(i,i+1)$ for all $i=1,\ldots,n-1$).
Since the generators $g_i$ of ${\rm Y}(d,m,n)$ satisfy the same braid relations as the generators of $\mathfrak{S}_n$, Matsumoto's lemma implies that the element 
\begin{equation}\label{def-gw}g_w:=g_{i_1}g_{i_2}\ldots g_{i_r}
\end{equation} 
is well-defined, that is, it does not depend on the choice of the reduced expression of $w\in \mathfrak{S}_n$.

\vskip .2cm
We define inductively elements $X_2,\ldots,X_n$ of ${\rm Y}(d,m,n)$ by
\begin{equation}\label{rec-X}
X_{i+1}:=g_iX_ig_i\ \ \ \ \text{for $i=1,\ldots,n-1$.}
\end{equation}
In \cite[Lemma 1]{ChPo}, it is proved that,
 for any $i\in\{1,\ldots,n\}$, we have: 
\begin{equation}\label{g-X}
g_jX_i=X_ig_j\ \ \ \text{for $j=1,\ldots,n-1$ such that $j\neq i-1,i$,}
\end{equation}
Moreover, we have that the elements $t_1,\ldots,t_n,X_1,\ldots,X_n$ form a commutative family  \cite[Proposition 1]{ChPo}, that is,
\begin{equation}\label{X-X} xy=yx\ \quad\text{for any $x,y\in\{t_1,\ldots,t_n,X_1,\ldots,X_n\}$\,.}
\end{equation}
Finally, we have 
\begin{equation}\label{g-XX} g_iX_iX_{i+1}=g_iX_ig_iX_ig_i=X_{i+1}X_ig_i=X_iX_{i+1}g_i\ \ \ \ \ \ \text{for $i=1,\ldots,n-1$.}
\end{equation}

\vskip .1cm
We record here some formulas that we will need in the rest of this paper:
\begin{lem}\label{lem-form}
We have the following identities satisfied in ${\rm Y}(d,m,n)$ ($i=1,\ldots,n-1$):
\begin{equation}\label{formAK}
g_iX_i^aX_{i+1}^b=\left\{\begin{array}{ll} \!\!X_i^bX_{i+1}^ag_i-(q-q^{-1})e_i\sum\limits_{k=1}^{a-b}X_i^{a-k}X_{i+1}^{b+k} & \text{if $a\geq b$,}\\[1.4em]
\!\!X_i^bX_{i+1}^ag_i+(q-q^{-1})e_i\sum\limits_{k=0}^{b-a-1}X_i^{a+k}X_{i+1}^{b-k} & \text{if $a\leq b$,}
\end{array}\right.\ \qquad a,b\in\Z\,,
\end{equation}
\begin{equation}\label{form1}
X_1\,g_1X_1^{a}t_1^bg_1=g_1X_1^{a}t_1^bg_1\,X_1+(q-q^{-1})e_1(X_1t_1^b\,g_1X_1^{a}-X_1^{a}t_1^b\,g_1X_1),\ \qquad a,b\in\Z\,, 
\end{equation}
\begin{equation}\label{form3}
X_1\,g_1^{-1}X_1^{a}t_1^bg_1=g_1^{-1}X_1^{a}t_1^bg_1\,X_1+(q-q^{-1})e_1(X_1t_1^b\,g_1X_1^{a}-X_1^{a+1}t_1^b\,g_1),\ \qquad a,b\in\Z\,.
\end{equation}
\end{lem}
\begin{proof}
To prove (\ref{formAK}), we use that $g_i$ commutes with $X_iX_{i+1}$, see (\ref{g-XX}), together with the following relations, which are easily proved by induction on $a$,
\begin{equation}\label{noname1}
g_iX_i^a=X_{i+1}^ag_i-(q-q^{-1})e_i\sum_{k=1}^aX_i^{a-k}X_{i+1}^k,\ \qquad a\in\Z_{>0},
\end{equation}
\begin{equation}\label{noname2}
g_iX_{i+1}^a=X_i^ag_i+(q-q^{-1})e_i\sum_{k=0}^{a-1}X_i^{k}X_{i+1}^{a-k},\ \qquad a\in\Z_{>0}.
\end{equation}

To obtain (\ref{form1}) for $b=0$, 
we multiply the equality $g_1X_1g_1X_1^{a}=X_1^{a}g_1X_1g_1$ by $g_1^{-1}$ from both sides, and use that $g_1^{-1}=g_1-(q-q^{-1})e_1$,
together with the fact that $e_1$ commutes with $g_1$ and $X_1$.
To obtain (\ref{form1}) for any $b\in\Z$, we first move $t_1^b$ to the right through $g_1$, so it becomes $t_2^b$. Then we apply (\ref{form1}) for $b=0$ and we move back $t_2^b$ through $g_1$ in each term. 

Finally, to obtain (\ref{form3}), we first replace $g_1^{-1}$ by $g_1-(q-q^{-1})e_1$ in $X_1\,g_1^{-1}X_1^{a}t_1^bg_1$.
We then use  (\ref{form1}) and  we transform $g_1-(q-q^{-1})e_1$ appearing in front of $X_1^{a}t_1^bg_1X_1$ back into $g_1^{-1}$. 
\end{proof}

The algebra ${\rm Y}(d,m,n)$ admits an involutive ring homomorphism $\eta$ given, for any generator $x\in\{t_1,\ldots,t_n,$ $g_1,\ldots,g_{n-1},X_1^{\pm 1}\}$, by
\begin{equation}\label{eta}
\eta(x)=x^{-1},\ \ \  \eta(q)=q^{-1}\ \ \ \text{and \ \ \,(if $m<\infty$) }\ \ \ \eta(v_a)=v_a^{-1},\ a=1,\ldots,m.
\end{equation}
The existence of the homomorphism $\eta$ is immediate once we notice that $g_i^{-2}=1-(q-q^{-1})e_ig_i^{-1}$.

\section{Representation theory of the cyclotomic Yokonuma--Hecke algebra}\label{sec-rep}

In this section, we consider only the situation $m<\infty$, and we construct, with explicit formulas, a set of representations of the algebra $\cF_m{\rm Y}(d,m,n):=\cF_m\otimes_{\cR_m}{\rm Y}(d,m,n)$, labelled by combinatorial objects called $(d,m)$-partitions. Then we show that the representations constructed are pairwise non-isomorphic and irreducible. It will turn out in  Section \ref{sec-base} that these representations form a complete set of pairwise non-isomorphic irreducible representations for $\cF_m{\rm Y}(d,m,n)$.

\subsection{Multipartitions and multitableaux}
In this subsection we will introduce the combinatorial tools needed to describe the representations of  the cyclotomic Yokonuma--Hecke algebra.

\subsubsection{$(d,m)$-partitions}
Let $\lambda\vdash n$ be a partition of $n$, that is, $\lambda=(\lambda_1,\ldots,\lambda_k)$ is a family of  positive integers such that $\lambda_1\geq\lambda_2\geq\ldots\geq\lambda_k\geq 1$ and $|\lambda|:=\lambda_1+\ldots+\lambda_k=n$. We shall also say that $\lambda$ is a partition {\em of size} $n$. 

We identify partitions with their Young diagrams: the Young diagram of $\lambda$ is a left-justified array of $k$ rows such that
the $j$-th row contains  $\lambda_j$ {\em nodes } for all $j=1,\ldots,k$. We write $\theta=(x,y)$ for the node in row $x$ and column $y$. A node $\theta \in \lambda$ is called {\it removable} if the set of nodes obtained from $\lambda$ by removing $\theta$ is still a partition. A node $\theta' \notin \lambda$ is called {\it  addable} if the set of nodes obtained from $\lambda$ by adding $\theta'$ is still a partition.

Let $r\in\Z_{>0}\,$. An $r$-partition $\blambda$, or a Young $r$-diagram, of size $n$ is an $r$-tuple of partitions such that the total number of nodes in the associated Young diagrams is equal to $n$. That is, we have $\blambda=(\blambda^{(1)},\ldots,\blambda^{(r)})$ with $\blambda^{(1)},\ldots,\blambda^{(r)}$ usual partitions such that $|\blambda^{(1)}|+\ldots+|\blambda^{(r)}|=n$.

The combinatorial objects appearing in the representation theory of the cyclotomic Yokonuma--Hecke algebra  
$\cF_m{\rm Y}(d,m,n)$ 
will be $r$-partitions with $r=dm$. It will be convenient to consider them as $d$-tuples of $m$-tuples of partitions (\emph{i.e.} $d$-tuples of $m$-partitions). We will call such an object a \emph{$(d,m)$-partition}. The size of a $(d,m)$-partition is the size of the $dm$-partition associated. We will say that the $l$-th partition of the $k$-th $d$-tuple has \emph{position} $(k,l)$. This is an example of a $(2,2)$-partition of size $7$:
\begin{equation}\label{exdm}
\left(\left(\begin{array}{l}\Box\!\Box\\[-0.5em] \Box\end{array}
\, ,\, \begin{array}{l}\Box\\[-0.5em] \phantom{\Box}\end{array}\right)
\, ,\,\left(\varnothing\,,\,\begin{array}{l}\Box\\[-0.5em] \Box\\[-0.5em] \Box\end{array}\right)\right), 
\end{equation}
with the partition $\!\begin{array}{l}\Box\!\Box\\[-0.5em] \Box\end{array}\!$ in position $(1,1)$, the partition $\Box$ in position $(1,2)$, the empty partition in position $(2,1)$ and the partition $\!\begin{array}{l}\Box\\[-0.5em] \Box\\[-0.5em] \Box\end{array}\!$ in position $(2,2)$.

A triplet $\btheta=(\theta,k,l)$ consisting of a node $\theta$, an integer $k\in\{1,\ldots,d\}$ and an integer $l\in\{1,\ldots,m\}$ will be called a $(d,m)$-node. The integer $k$ is called the \emph{$d$-position} of $\btheta$ and the integer $l$ is called the \emph{$m$-position} of $\btheta$. We will also say that the $(d,m)$-node $\btheta$ has \emph{position $(k,l)$}. A $(d,m)$-partition $\blambda$ is thus naturally identified with a set of $(d,m)$-nodes such that the subset consisting of the $(d,m)$-nodes having position $(k,l)$ forms a usual partition, for any $k\in\{1,\ldots,d\}$ and $l\in\{1,\ldots,m\}$. For a $(d,m)$-node $\btheta$ belonging to this set, we will say that $\btheta$ is a $(d,m)$-node of $\blambda$, and write $\btheta\in\blambda$. For example, the following $(2,2)$-nodes are the $(2,2)$-nodes of the $(2,2)$-partition above:
\[\bigl((1,1),1,1\bigr),\ \bigl((1,2),1,1\bigr),\ \bigl((2,1),1,1\bigr),\ \ \,\bigl((1,1),1,2\bigr),\ \ \,\bigl((1,1),2,2\bigr),\ \bigl((2,1),2,2\bigr),\ \bigl((3,1),2,2\bigr)\,.\] 

Let ${\blambda}$ be a $(d,m)$-partition. A $(d,m)$-node $\btheta=(\theta,k,l)\in\blambda$ is called removable from $\blambda$ if the set of $(d,m)$-nodes obtained from $\blambda$ by removing $\btheta$ is still a $(d,m)$-partition, or equivalently if the node $\theta$ is removable from the partition of $\blambda$ with position $(k,l)$. 
We will write $\blambda \setminus \{\btheta\}$ for the $(d,m)$-partition obtained by removing a removable $(d,m)$-node $\btheta$ from $\blambda$.
Respectively, a $(d,m)$-node $\btheta=(\theta,k,l)\notin\blambda$ is called addable to $\blambda$ if the set of $(d,m)$-nodes obtained from $\blambda$ by adding $\btheta$ is still a $(d,m)$-partition, or equivalently if the node $\theta$ is addable to the partition of $\blambda$ with position $(k,l)$. 
We will write $\blambda \cup \{\btheta\}$ for the $(d,m)$-partition obtained by adding an addable $(d,m)$-node $\btheta$ to $\blambda$.
The set of $d$-nodes removable from $\blambda$ is denoted by ${\mathcal{E}}_-(\blambda)$ and the set of $d$-nodes addable to $\blambda$ is denoted by ${\mathcal{E}}_+(\blambda)$.

\vskip .1cm
For a $(d,m)$-node $\btheta=((x,y),k,l)$, we define $\posd(\btheta):=k$, $\posm(\btheta):=l$ and the \emph{(quantum) content} $\cc(\btheta)$ of $\btheta$ by $\cc(\btheta):=v_l\,q^{2(y-x)}$ (recall that $q,v_1,\ldots,v_m$ are the parameters appearing in the defining relations of the cyclotomic Yokonuma--Hecke algebra ${\rm Y}(d,m,n)$). Note that the quantum content of $\btheta$ contains simultaneously the information about $\posm(\btheta)$ and about the \emph{classical content} $\ccc(\btheta):=y-x$. We will refer to the array $\left(\begin{array}{l}\posd(\btheta)\\ \cc(\btheta)\end{array}\right)$ as the \emph{content array} of the $(d,m)$-node $\btheta$.

A $(d,m)$-partition $\blambda$ is fully characterised by its 
collection of content arrays 
\begin{equation}\label{content arrays}
\left\{\left(\begin{array}{l}\posd(\btheta)\\ \cc(\btheta)\end{array}\right)\,|\ \btheta\in\blambda\right\}.
\end{equation}
Indeed, if two $(d,m)$-nodes $\btheta, \btheta'\in\blambda$ satisfy $\posd(\btheta)=\posd(\btheta')$ and $\cc(\btheta)=\cc(\btheta')$, then they satisfy also $\posm(\btheta)=\posm(\btheta')$ and $\ccc(\btheta)=\ccc(\btheta')$. Therefore, they lie in the same diagonal of the same partition in $\blambda$, namely the diagonal of the partition with position $(\posd(\btheta),\posm(\btheta))$ and whose nodes have classical content equal to $\ccc(\btheta)$. Once we know the number of nodes on each diagonal of each partition of $\blambda$, the $(d,m)$-partition $\blambda$ is fixed.

\subsubsection{Standard $(d,m)$-tableaux}
Let $\blambda$ be a $(d,m)$-partition of size $n$. A {\em $(d,m)$-tableau of shape $\blambda$
} is a bijection between the set $\{1,\ldots,n\}$ and the set of $(d,m)$-nodes in $\blambda$. In other words, a $(d,m)$-tableau of shape $\blambda$ is obtained by placing the numbers $1,\ldots,n$ in the $(d,m)$-nodes of $\blambda$. 
The \emph{size} of a $(d,m)$-tableau is the size of its shape. 
 A $(d,m)$-tableau is {\em standard} if its entries increase along any row and down 
 any column of every diagram in $\blambda$.
For example, the following $(2,2)$-tableau of size $7$ is a standard $(2,2)$-tableau of shape the $(2,2)$-partition 
given in (\ref{exdm}): 
 \[\left(\left(\begin{array}{l}\fbox{2}\fbox{5}\\[-0.05em] \fbox{4}\end{array}
\, ,\, \begin{array}{l}\fbox{3}\\[-0.05em] \phantom{\fbox{2}}\end{array}\right)
\, ,\,\left(\varnothing\,,\,\begin{array}{l}\fbox{1}\\[-0.05em] \fbox{6}\\[-0.05em] \fbox{7}\end{array}\right)\right).\]

For a $(d,m)$-tableau $\cT$, we denote respectively by $\posd({\mathcal{T}}|i)$, $\posm({\mathcal{T}}|i)$ and $\cc({\mathcal{T}}|i)$  the $d$-position, the $m$-position and the quantum content  of the $(d,m)$-node with the number $i$ in it. For example, for the standard $(2,2)$-tableau above, we have:
\[\posd(\cT|2)=\posd(\cT|3)=\posd(\cT|4)=\posd(\cT|5)=1\,,\ \ \posd(\cT|1)=\posd(\cT|6)=\posd(\cT|7)=2\,,\]

\[\posm(\cT|2)=\posm(\cT|4)=\posm(\cT|5)=1\,,\ \ \posm(\cT|1)=\posm(\cT|3)=\posm(\cT|6)=\posm(\cT|7)=2\,,\]

\[\cc(\cT|2)=v_1,\ \ \cc(\cT|5)=v_1q^2,\ \ \cc(\cT|4)=\frac{v_1}{q^{2}},\ \ \cc(\cT|3)=\cc(\cT|1)=v_2,\ \ \cc(\cT|6)=\frac{v_2}{q^{2}},\ \ \cc(\cT|7)=\frac{v_2}{q^{4}}\,.\]

Recall that a $(d,m)$-partition is uniquely determined by its  
collection of content arrays (\ref{content arrays}). Thus, any standard  $(d,m)$-tableau $\cT$ of size $n$ is fully characterised by its 
sequence of content arrays
\begin{equation}\label{T-cont}
\left(\left(\begin{array}{l}\posd(\cT|1)\\ \cc(\cT|1)\end{array}\right),\ \left(\begin{array}{l}\posd(\cT|2)\\ \cc(\cT|2)\end{array}\right),\ \dots\ ,\ \left(\begin{array}{l}\posd(\cT|n)\\ \cc(\cT|n)\end{array}\right)\right)\ .
\end{equation}

It is worth noting that, for any standard $(d,m)$-tableau $\cT$ and any $i\in\{1,\ldots,n-1\}$, if we have $\posd(\cT|i)=\posd(\cT|i+1)$, then $\cc(\cT|i)\neq\cc(\cT|i+1)$ (in other words, the nodes containing $i$ and $i+1$ cannot lie in the same diagonal of the same diagram).

Finally, let $\cT$ be a standard $(d,m)$-tableau of shape $\blambda$ and of size $n$. As $\cT$ is standard, the $(d,m)$-node $\btheta$ that contains $n$ is removable from $\lambda$. We will write $\cT \setminus \scriptstyle{\left\{\fbox{\scriptsize{$n$}}\right\}}$ for the standard $(d,m)$-tableau of shape $\blambda\setminus \{\btheta\}$ obtained from $\cT$ by removing $\btheta$.

\subsection{Formulas for the representations}

For any $(d,m)$-tableau $\cT$ of size $n$ and any permutation $\sigma \in \mathfrak{S}_n$, we denote by $\cT^{\sigma}$ the $(d,m)$-tableau obtained from $\cT$ by applying the permutation $\sigma$ on the numbers contained in the $(d,n)$-nodes of $\cT$. We have, for $i=1,\ldots,n$,
\begin{equation}\label{pos-con}\posd(\cT^{\sigma}|i)=\posd\bigl(\cT|\sigma^{-1}(i)\bigr)\,,\ \ \ \posm(\cT^{\sigma}|i)=\posm\bigl(\cT|\sigma^{-1}(i)\bigr)\ \ \ \text{and}\ \ \ \cc(\cT^{\sigma}|i)=\cc\bigl(\cT|\sigma^{-1}(i)\bigr)\,.\end{equation}

\begin{rem}\label{non-standard}{\rm
Note that if the $(d,m)$-tableau $\cT$ is standard, the $(d,m)$-tableau $\cT^{\sigma}$ is not necessarily standard. If $\sigma=s_i=(i,i+1)$ and $\cT$ is a standard $(d,m)$-tableau, the $(d,m)$-tableau $\cT^{s_i}$ is not standard if and only if $\posd(\cT|i)=\posd\bigl(\cT|i+1\bigr)$ and $\cc(\cT|i)=q^{\pm2}\cc\bigl(\cT|i+1\bigr)$.
\hfill$\triangle$}\end{rem}

Let $\{\xi_1,\ldots,\xi_d\}$ be the set of all $d$-th roots of unity (ordered arbitrarily). Denote by $\mathcal{P}(d,m,n)$ the set of all $(d,m)$-partitions of size $n$, and let $\blambda \in \mathcal{P}(d,m,n)$. Let $V_{\blambda}$ be an  $\cF_m$-vector space with a basis $\{\bv_{_{\cT}}\}$ indexed by the standard $(d,m)$-tableaux of shape $\blambda$. We set $\bv_{_{\cT}}:=0$ for any non-standard $(d,m)$-tableau $\cT$ of shape $\blambda$. 

\begin{prop}\label{prop-rep} 
Let $\cT$ be a standard $(d,m)$-tableau of shape $\blambda  \in \mathcal{P}(d,m,n)$. For brevity, we set $\posd_i\!\!:=\posd(\cT|i)$ 
and $\cc_i:=\cc(\cT|i)$ for $i=1,\ldots,n$. 
The vector space $V_{\blambda}$ is a representation of $\cF_m{\rm Y}(d,m,n)$ with the action of the generators on the basis element $\bv_{_{\cT}}$ defined as follows:
\begin{equation}\label{rep-X}
X_1(\bv_{_{\cT}})=\cc_1\bv_{_{\cT}}\  ;
\end{equation}
for $j=1,\ldots,n$,
\begin{equation}\label{rep-t}
t_j(\bv_{_{\cT}})=\xi_{\pos^{\!(d)}_j}\bv_{_{\cT}}\  ;
\end{equation}
for $i=1,\ldots,n-1$, if $\posd_{i} \neq \posd_{i+1}$ then
\begin{equation}\label{rep-g1}
g_i(\bv_{_{\cT}})=\bv_{_{\cT^{s_i}}},
\end{equation}
and if $\posd_{i}=\posd_{i+1}$ then
\begin{equation}\label{rep-g2}
g_i(\bv_{_{\cT}})=\frac{\cc_{i+1}(q-q^{-1})}{\cc_{i+1}-\cc_i}\,\bv_{_{\cT}}+\frac{q\cc_{i+1}-q^{-1}\cc_i}{\cc_{i+1}-\cc_i}\,\bv_{_{\cT^{s_i}}},
\end{equation}
where $s_i$ is the transposition $(i,i+1)$.
\end{prop}
\begin{proof} We would first like to point out that the formulas for the action of the generators $t_1,\ldots,t_n$, $g_1,\ldots,g_{n-1}$ are formally exactly the same as the formulas for the representations of the Yokonuma--Hecke algebra in \cite[Proposition 5]{ChPo}. The difference only lies in the definition of the quantum content which involves here the parameters $v_1,\ldots,v_m$. 
Moreover, for any standard $(d,m)$-tableau $\cT$, the necessary and sufficient conditions for the $(d,m)$-tableau $\cT^{s_i}$ to be standard (see Remark \ref{non-standard}) are the same as for a $(d,1)$-tableau.
Therefore, the verification that the defining relations (\ref{def-aff1}) are satisfied by the images of the generators $t_1,\ldots,t_n,g_1,\ldots,g_{n-1}$ in the representations is exactly the same as the verification for the Yokonuma--Hecke algebras in \cite{ChPo}. It is a straightforward calculation, and we
do not repeat it here.

It remains to check the defining relations (\ref{def-aff1}) involving the generator $X_1$. The relations 
$$\begin{array}{rclcl}
X_1g_i & = & g_iX_1 && \mbox{for all $i=2,\ldots,n-1$,}\\[0.1em]
X_1t_j & = & t_jX_1 && \mbox{for all $j=1,\ldots,n$,}\\[0.1em]
(X_1-v_1)\cdots(X_1-v_m) & = & 0 && \mbox{if $m<\infty$}
\end{array}
$$
are obviously satisfied on $V_{\blambda}$. Finally, the relation $X_1g_1X_1g_1=g_1X_1g_1X_1$, equivalent to  $X_2X_1=X_1X_2$, follows from Lemma \ref{lem-rep} below.
\end{proof}

Recall that the elements $X_2,\ldots,X_n$ 
are defined inductively by $X_{i+1}=g_iX_ig_i$, ${i=1,\ldots,n-1}$.
\begin{lem}\label{lem-rep}
The action of the elements $X_1,\ldots,X_n$ on $V_{\blambda}$ is given, on a basis element $\bv_{_{\cT}}$ as above, by
\begin{equation}\label{rep-Xi}
X_i(\bv_{_{\cT}})=\cc_i\bv_{_{\cT}}\ \ \ \ \ \text{for $i=1,\ldots,n$.}
\end{equation}
\end{lem}
\begin{proof}
We prove Formula (\ref{rep-Xi}) by induction on $i$. For $i=1$, this is the defining action of $X_1$ given by (\ref{rep-X}). 
Let now $i\in\{1,\ldots,n-1\}$. We will show that
\begin{equation}\label{rep-Xi+1}
X_{i+1}(\bv_{_{\cT}})=  g_iX_{i}g_i(\bv_{_{\cT}}) =\cc_{i+1}\bv_{_{\cT}}. 
\end{equation}
Note that if $\posd_i=\posd_{i+1}$, we are in the situation of the Ariki--Koike algebra ${\rm H}(m,n)$ for the action of $g_i$ on $\bv_{_{\cT}}$
and $\bv_{_{\cT^{s_i}}}$; 
in this case, (\ref{rep-Xi+1})  is well-known and follows from a straightforward calculation (see, for example, \cite{ArKo}).
Thus, let $\posd_i\neq\posd_{i+1}$. Then the $(d,m)$-tableau $\cT^{s_i}$ is standard and we calculate:
\[X_{i+1}(\bv_{_{\cT}})=g_iX_ig_i(\bv_{_{\cT}})=g_iX_i(\bv_{_{\cT^{s_i}}})=
\cc(\cT^{s_i}| {i})\!\cdot g_i(\bv_{_{\cT^{s_i}}})=
\cc_{i+1}\!\cdot g_i(\bv_{_{\cT^{s_i}}})=\cc_{i+1}\!\cdot \bv_{_{\cT}}\ \,,\]
where we use the induction hypothesis in the third equality,  Equation (\ref{pos-con}) in the fourth, and Formula (\ref{rep-g1}) for the action of $g_i$ in the second and last equalities.
\end{proof}

\subsection{Distinctness and irreducibility} In the previous paragraph we constructed representations $V_{\blambda}$ of $\cF_m{\rm Y}(d,m,n)$, where $\blambda$ runs over the set $\mathcal{P}(d,m,n)$. In these subsection, we will show that these representations are distinct and irreducible.

\begin{prop}\label{prop-dis-irr}
The representations $V_{\blambda}$, where $\blambda$ runs over the set $\mathcal{P}(d,m,n)$, are irreducible pairwise non-isomorphic representations of $\cF_m{\rm Y}(d,m,n)$.
\end{prop}
\begin{proof} The proof is very similar to the proof of the analogous results for the Ariki--Koike algebras in \cite{ArKo} and for the Yokonuma--Hecke algebras in \cite{ChPo}. We give it here for completeness.

The fact that the representations $V_{\blambda}$, where $\blambda\in\mathcal{P}(d,m,n)$, are pairwise non-isomorphic follows from the described action of the elements $t_1,\ldots,t_n,X_1,\ldots,X_n$ in Proposition \ref{prop-rep} and Lemma \ref{lem-rep}, together with the already noted fact that a standard 
$(d,m)$-tableau $\cT$ is characterised by its sequence of content arrays, see (\ref{T-cont}). 

Let $\blambda\in\mathcal{P}(d,m,n)$. The irreducibility of $V_{\blambda}$ is proved by induction on the size $n$ of $\blambda$. For $n=1$, the representations are one-dimensional so there is nothing to prove. Let $n>1$ and denote by $\mathrm{A}$ the subalgebra of $\cF_m{\rm Y}(d,m,n)$ generated by $t_1,\ldots,t_{n-1},g_1,\ldots,g_{n-2},X_1$. The algebra $\mathrm{A}$ is a quotient of the algebra $\cF_m{\rm Y}(d,m,n-1)$.

Let $\bmu$ be a $(d,m)$-partition of size $n-1$ of the form $\blambda \setminus\{\btheta\}$ where $\btheta$ is a removable $(d,m)$-node of $\blambda$. As $\cF_m$-vector space, $V_{\bmu}$ is isomorphic to the subspace of $V_{\blambda}$ 
spanned by the vectors of the form $\bv_{_{\cT}}$, with $\cT$ such that $\cT \setminus \scriptstyle{\left\{\fbox{\scriptsize{$n$}}\right\}}$ is of shape $\bmu$. Through this identification, we have the following isomorphism of $\cF_m$-vector spaces:
\begin{equation}\label{rest}
V_{\blambda}\cong \bigoplus_{\btheta\in\mathcal{E}_-(\blambda)} V_{\blambda \setminus\{\btheta\}},
\end{equation}
where $\mathcal{E}_-(\blambda)$ denotes the set of removable $(d,m)$-nodes of $\blambda$. A direct inspection at the formulas (\ref{rep-X})--(\ref{rep-g2}) for the action of the generators shows that the isomorphism in (\ref{rest}) is in fact an isomorphism of $\mathrm{A}$-modules. By induction hypothesis, the representations $V_{\blambda \setminus\{\btheta\}}$ appearing in (\ref{rest}) are irreducible representations of  $\cF_m{\rm Y}(d,m,n-1)$, and hence of $\mathrm{A}$. Moreover, we already showed that they are pairwise non-isomorphic.

Now assume that $M$ is a non-trivial proper $\cF_m{\rm Y}(d,m,n)$-submodule of $V_{\blambda}$. By the decomposition (\ref{rest}) of the $\mathrm{A}$-module $V_{\blambda}$ as a direct sum of irreducible $\mathrm{A}$-modules, there must be two $(d,m)$-nodes $\btheta,\btheta'$ removable from $\blambda$ such that $V_{\blambda \setminus\{\btheta\}}\subset M$ and $V_{\blambda \setminus\{\btheta'\}}\cap M=\{0\}$. Let $\cT$ be a standard $(d,m)$-tableau of shape $\blambda$ with number $n$ in $\btheta$ and number $n-1$ in $\btheta'$. As $\btheta$ and $\btheta'$ are both removable from $\blambda$, they cannot lie in the same diagonal nor in adjacent diagonals of the same diagram in $\blambda$, and thus such a standard $(d,m)$-tableau $\cT$ exists; moreover, $\cT^{s_{n-1}}$ is also standard. By construction, $\bv_{_{\cT}}\in V_{\blambda \setminus\{\btheta\}}$ and $\bv_{_{\cT^{s_{n-1}}}}\in V_{\blambda \setminus\{\btheta'\}}$. Now, if $\posd(\btheta) \neq \posd(\btheta')$, due to (\ref{rep-g1}), we have
\[\bv_{_{\cT^{s_{n-1}}}}=g_{n-1}(\bv_{_{\cT}}).\]
If $\posd(\btheta) = \posd(\btheta')$, then,
due to (\ref{rep-g2}), we have
\[\frac{q\cc_{n}-q^{-1}\cc_{n-1}}{\cc_{n}-\cc_{n-1}}\,\bv_{_{\cT^{s_{n-1}}}}=g_{n-1}(\bv_{_{\cT}})-\frac{\cc_{n}(q-q^{-1})}{\cc_{n}-\cc_{n-1}}\,\bv_{_{\cT}}\ \]
and $q\cc_{n}-q^{-1}\cc_{n-1}\neq 0$, following Remark \ref{non-standard}.
In every case, we have that $\bv_{_{\cT^{s_{n-1}}}}$  belongs to the $\cF_m{\rm Y}(d,m,n)$-submodule $M$.
This contradicts the fact that  $V_{\blambda \setminus\{\btheta'\}}\cap M=\{0\}$. Thus, a non-trivial proper $\cF_m{\rm Y}(d,m,n)$-submodule of $V_{\blambda}$ does not exist.
\end{proof}

\section{Linear bases of ${\rm Y}(d,m,n)$}\label{sec-base}

In this section, we return to the general situation of an arbitrary $m\in\Z_{>0}\cup\{\infty\}$. The goal of this section is to construct explicitly several $\cR_m$-bases of the algebra ${\rm Y}(d,m,n)$. We first exhibit some generating sets of elements of ${\rm Y}(d,m,n)$ and then we use the representation theory developed in the preceding section to deduce that, when $m<\infty$, these sets of elements are linearly independent. Due to the uniformity (with respect to $m$) of the form of the basis elements, the linear independence for $m=\infty$ is a consequence of the result for finite $m$.
Finally, we use the results obtained in this section to conclude that the representations constructed in the previous section form a complete set of pairwise non-isomorphic irreducible representations and to obtain a semisimplicity criterion for $\cF_m{\rm Y}(d,m,n)$ when $m<\infty$.

\subsection{Generating sets}
Recall that we identify the Yokonuma--Hecke algebra ${\rm Y}_m(d,n)$ of type $A$ with the subalgebra of ${\rm Y}(d,m,n)$ generated by $t_1,\ldots,t_n,g_1,\ldots,g_{n-1}$.
It is known that  ${\rm Y}_m(d,n)$ is a free $\cR_m$-module of rank $d^nn!$ \cite{ju2}. Let $\mathcal{B}_{d,n}$ be an $\cR_m$-basis of ${\rm Y}_m(d,n)$. 

\begin{exmp}\label{exmp of basis}{\rm
Juyumaya \cite{ju2} has shown that the following set is a basis for ${\rm Y}_m(d,n)$:
\begin{center}
$\mathcal{B}_{d,n}^{\rm can}:=\left\{\,t_1^{r_1}\ldots t_n^{r_n} g_w\,\left|\,w \in \mathfrak{S}_n, \, r_1,\ldots,r_n \in \Z/d\Z\right\}\right.,$
\end{center}
where $\mathfrak{S}_n$ is the symmetric group on $n$ letters and $g_w$ is defined in (\ref{def-gw}).
}
\hfill$\triangle$ 
\end{exmp}

We denote by $\mathcal{B}^{\text{AK}}_{d,m,n}$ the following set of elements of ${\rm Y}(d,m,n)$:
\begin{equation}\label{base-AK}
X_1^{a_1}\ldots X_n^{a_n}\cdot\omega\ \quad\text{where $(a_1,\ldots,a_n)\in E_m^n$ and $\omega\in\mathcal{B}_{d,n}$.}
\end{equation}
Let us now introduce the following notation for $k=1,\ldots,n$,
\[\begin{array}{l}

W^{(k)}_{J,a,b}:=g_{J}^{-1}\ldots g_2^{-1}g_1^{-1}X_1^a\,t_1^b\,g_1g_2\ldots g_{k-1}\,, \\[1em] 

W^{(k)-}_{J,a,b}:=g_{J}\ldots g_2g_1\,X_1^a\,t_1^b\,g_1^{-1}g_2^{-1}\ldots g_{k-1}^{-1}\,,\\[1em] 

\widetilde{W}^{(k)}_{J,a,b}:=g_{J}\ldots g_2g_1\,X_1^a\,t_1^b\,g_1g_2\ldots g_{k-1}\,, \\[1em] 

\widetilde{W}^{(k)-}_{J,a,b}:=g_{J}^{-1}\ldots g_2^{-1}g_1^{-1}\,X_1^a\,t_1^b\,g_1^{-1}g_2^{-1}\ldots g_{k-1}^{-1}\,,

\end{array}\]
where $J\in\{0,\ldots,k-1\}$ 
and $a,b\in\Z$. 
We use the following standard conventions: for $\epsilon=\pm1$, $g_{J}^{\epsilon}\ldots g_2^{\epsilon}g_1^{\epsilon}:=1$ and $g_{k-J}^{\epsilon}\ldots g_{k-2}^{\epsilon}g_{k-1}^{\epsilon}:=1$  if $J=0$. 
 Then we denote, respectively, by $\mathcal{B}^{\text{Ind}}_{d,m,n}$, $\mathcal{B}^{\text{Ind}-}_{d,m,n}$, $\widetilde{\mathcal{B}}^{\text{Ind}}_{d,m,n}$ and $\widetilde{\mathcal{B}}^{\text{Ind}-}_{d,m,n}$ the following sets of elements of ${\rm Y}(d,m,n)$:
\begin{equation}\label{base-ind1}
W^{(n)}_{J_n,a_n,b_n}\ldots W^{(2)}_{J_2,a_2,b_2}W^{(1)}_{J_1,a_1,b_1}\,, \ \quad\text{$J_k\in\{0,\ldots,k-1\}$, $a_k\in E_m$ and $b_k\in\{0,\ldots,d-1\}$.}
\end{equation}
\begin{equation}\label{base-ind2}
W^{(n)-}_{J_n,a_n,b_n}\ldots W^{(2)-}_{J_2,a_2,b_2}W^{(1)-}_{J_1,a_1,b_1}\,, \ \quad\text{$J_k\in\{0,\ldots,k-1\}$, $a_k\in E_m$ and $b_k\in\{0,\ldots,d-1\}$.}
\end{equation}
\begin{equation}\label{base-ind3}
\widetilde{W}^{(n)}_{J_n,a_n,b_n}\ldots \widetilde{W}^{(2)}_{J_2,a_2,b_2}\widetilde{W}^{(1)}_{J_1,a_1,b_1}\,, \ \quad\text{$J_k\in\{0,\ldots,k-1\}$, $a_k\in E_m$ and $b_k\in\{0,\ldots,d-1\}$.}
\end{equation}
\begin{equation}\label{base-ind4}
\widetilde{W}^{(n)-}_{J_n,a_n,b_n}\ldots \widetilde{W}^{(2)-}_{J_2,a_2,b_2}\widetilde{W}^{(1)-}_{J_1,a_1,b_1}\,, \ \quad\text{$J_k\in\{0,\ldots,k-1\}$, $a_k\in E_m$ and $b_k\in\{0,\ldots,d-1\}$.}
\end{equation}

The set $\mathcal{B}^{\text{AK}}_{d,m,n}$ is the analogue of the Ariki--Koike basis of the Ariki--Koike algebra $H(m,n)$ \cite{ArKo} for $m<\infty$,
and the standard Bernstein basis of the affine Hecke algebra of ${\rm GL}$ for $m=\infty$. The four other sets are inductive sets with respect to $n$, which are analogous to the inductive bases of $H(m,n)$ studied in \cite{La2,OgPo1}. The proof of the proposition below generalises the methods used in \cite{ArKo,OgPo1}.

\begin{prop}\label{prop-gen}
Each set  $\mathcal{B}^{\text{AK}}_{d,m,n}$, $\mathcal{B}^{\text{Ind}}_{d,m,n}$, $\mathcal{B}^{\text{Ind}-}_{d,m,n}$, $\widetilde{\mathcal{B}}^{\text{Ind}}_{d,m,n}$ and $\widetilde{\mathcal{B}}^{\text{Ind}-}_{d,m,n}$ generates (linearly over $\cR_m$) the algebra ${\rm Y}(d,m,n)$.
\end{prop}
\begin{proof} For $\mathcal{B}^{\text{AK}}_{d,m,n}$, let $A$ be the $\cR_m$-span of the set of elements (\ref{base-AK}) inside ${\rm Y}(d,m,n)$. As $\mathcal{B}_{d,n}$ is an $\cR_m$-basis of ${\rm Y}_m(d,n)$, the unit element can be expressed as linear combinations with coefficients in $\cR_m$ of elements $\omega\in\mathcal{B}_{d,n}$ (in Example \ref{exmp of basis}, we have in fact that the unit element is an element of the basis). 
Therefore, the unit element of ${\rm Y}(d,m,n)$ belongs to $A$.

It is enough to show that the product (for example, from the left)  of a generator of ${\rm Y}(d,m,n)$ with an element of $\mathcal{B}^{\text{AK}}_{d,m,n}$ still belongs to $A$,
because then $A$ becomes a subalgebra of ${\rm Y}(d,m,n)$ containing 
the unit element and 
all the generators of ${\rm Y}(d,m,n)$, that is, $A={\rm Y}(d,m,n)$. 

Let $(a_1,\ldots,a_n)\in E_m^n$ and $\omega\in\mathcal{B}_{d,n}$.
First, we have
$$X_1^{\pm 1} X_1^{a_1}\ldots X_n^{a_n}\cdot\omega = X_1^{a_1 \pm 1}\ldots X_n^{a_n}\cdot \omega \in A,$$
either automatically, or with the use of 
(\ref{X1m})--(\ref{inv-X1}) if $m < \infty$. 
Now, by (\ref{X-X}), we have that
$$t_j  X_1^{a_1}\ldots X_n^{a_n}\cdot\omega = X_1^{a_1}\ldots X_n^{a_n}\cdot t_j\omega \in A  \quad \text{ for all } j=1,\ldots,n,$$
since $t_j\omega$ can be written as an  $\cR_m$-linear combination of elements of   $\mathcal{B}_{d,n}$. Finally, by (\ref{g-X}), we have that
$$g_i  X_1^{a_1}\ldots X_n^{a_n}\cdot\omega = X_1^{a_1}\ldots g_i X_i^{a_i} X_{i+1}^{a_{i+1}}   \ldots X_n^{a_n}\cdot\omega \quad \text{ for all } i=1,\ldots,n-1.$$
With the use of (\ref{formAK}) and what we have seen above, we deduce that
$$X_1^{a_1}\ldots g_i X_i^{a_i} X_{i+1}^{a_{i+1}}   \ldots X_n^{a_n}\cdot\omega = X_1^{a_1}\ldots X_i^{a_{i+1}} X_{i+1}^{a_{i}}g_i    \ldots X_n^{a_n}\cdot\omega + \text{an element of } A.$$
Applying again (\ref{g-X}) yields 
$$X_1^{a_1}\ldots X_i^{a_{i+1}} X_{i+1}^{a_{i}}g_i    \ldots X_n^{a_n}\cdot\omega = X_1^{a_1}\ldots X_i^{a_{i+1}} X_{i+1}^{a_{i}}  \ldots X_n^{a_n}\cdot g_i\omega \in A \quad \text{ for all } i=1,\ldots,n-1, $$
since $g_i \omega$ can be written as an  $\cR_m$-linear combination of elements of   $\mathcal{B}_{d,n}$.
Thus, $$g_i  X_1^{a_1}\ldots X_n^{a_n}\cdot\omega \in A \quad \text{ for all } i=1,\ldots,n-1.$$

\vskip .1cm
We proceed similarly for $\mathcal{B}^{\text{Ind}}_{d,m,n}$. The unit element 
of ${\rm Y}(d,m,n)$ belongs to $\mathcal{B}^{\text{Ind}}_{d,m,n}$. So we just have to check that the product (for example, from the left) of a generator of ${\rm Y}(d,m,n)$ with any element of $\mathcal{B}^{\text{Ind}}_{d,m,n}$ is in the $\cR_m$-span of $\mathcal{B}^{\text{Ind}}_{d,m,n}$. We prove this statement by induction on $n$. 

First note that $W^{(n)}_{J,a,b+d}=W^{(n)}_{J,a,b}$ for any $b\in\Z$ and that, if $m<\infty$, the element $W^{(n)}_{J,a,b}$ with $a\in\Z$ can be rewritten as an $\cR_m$-linear combination of elements $W^{(n)}_{J,a',b}$ with $a'\in E_m$. The latter follows immediately from (\ref{X1m})--(\ref{inv-X1}).
This remark yields in particular the statement for $n=1$.
For $n>1$, we consider the products $x\cdot W^{(n)}_{J,a,b}$, where $x\in\{t_1,\ldots,t_n,g_1,\ldots,g_{n-1},X_1^{\pm1}\}\,$, $J\in\{0,\ldots,n-1\}$ and $a,b\in\Z$. Due to the above remark, it is enough to show that these products can be written as $\cR_m$-linear combinations of terms of the form $W^{(n)}_{J',a',b'}\cdot u$, where $u\in{\rm Y}(d,m,n-1)$, $J'\in\{0,\ldots,n-1\}$ and $a',b'\in\Z$ . Following the induction hypothesis, this will yield our statement. 

The rewriting of $x\cdot W^{(n)}_{J,a,b}$ is a straightforward case-by-case analysis which yields the following explicit formulas (for $l=1,\ldots,n$ and $i=1,\ldots,n-1$):
\begin{equation}\label{mult-t}t_l\cdot W^{(n)}_{J,a,b}=\left\{\begin{array}{ll} 
W^{(n)}_{J,a,b}\cdot t_l & \text{if $l\leq J$\,,}\\  & \\
W^{(n)}_{J,a,b+1} & \text{if $l=J+1$\,,}\\ & \\ 
W^{(n)}_{J,a,b}\cdot t_{l-1} & \text{if $l>J+1$\,,}
\end{array}\right.
\end{equation}
\begin{equation}\label{mult-g}g_i\cdot W^{(n)}_{J,a,b}=\left\{\begin{array}{ll} 
W^{(n)}_{J,a,b}\cdot g_i & \text{if $i<J$\,,}\\ & \\
W^{(n)}_{J-1,a,b} & \text{if $i=J$\,,}\\  & \\
W^{(n)}_{J+1,a,b}+(q-q^{-1})\displaystyle\frac{1}{d}\sum\limits_{s=0}^{d-1}W^{(n)}_{J,a,b-s}\cdot t_{J+1}^{s} & \text{if $i=J+1$\,,}\\ & \\ 
W^{(n)}_{J,a,b}\cdot g_{i-1} & \text{if $i>J+1$\,,}
\end{array}\right.
\end{equation}
\begin{equation}\label{mult-X}X_1\cdot W^{(n)}_{J,a,b}=\left\{\begin{array}{ll} 
W^{(n)}_{0,a+1,b} & \ \ \ \text{if $J=0$\,,}\\ & \\
W^{(n)}_{J,a,b}\cdot X_1+(q-q^{-1})\displaystyle\frac{1}{d}\sum\limits_{s=0}^{d-1}\Bigl(W^{(n)}_{0,1,b-s}\cdot g_{J-1}^{-1}\ldots g_1^{-1}t_1^sX_1^a &\\
\hspace{5.2cm}-W^{(n)}_{0,a+1,b-s}\cdot g_{J-1}^{-1}\ldots g_1^{-1}t_1^s\Bigr) & \ \ \ \text{if $J>0$\,,}
\end{array}\right.
\end{equation}
where we use Formula (\ref{form3}) in Lemma \ref{lem-form}, together with the equality:
\begin{equation}\label{maybe}
g_J^{-1} \ldots g_2^{-1} \cdot W_{0,a,b}^{(n)} = W_{0,a,b}^{(n)} \cdot g_{J-1}^{-1} \ldots g_1^{-1},
\end{equation}
which follows directly from 
(\ref{mult-g}).
For finite $m$ the proof for $\mathcal{B}^{\text{Ind}}_{d,m,n}$ is finished, while for $m=\infty$ it remains to multiply $W^{(n)}_{J,a,b}$ by $X_1^{-1}$. 
By multiplying both sides of (\ref{mult-X}) by $X_1^{-1}$, 
we obtain:
\begin{equation}\label{mult-X2}X_1^{-1}\!\!\cdot W^{(n)}_{J,a,b}=\left\{\begin{array}{ll} 
W^{(n)}_{0,a-1,b} & \ \ \ \text{if $J=0$\,,}\\ &\\
W^{(n)}_{J,a,b}\cdot X_1^{-1}-(q-q^{-1})\displaystyle\frac{1}{d}\sum\limits_{s=0}^{d-1}\Bigl(W^{(n)}_{0,0,b-s}\cdot g_{J-1}^{-1}\ldots g_1^{-1}t_1^sX_1^{a-1} &\\
\hspace{5.1cm}-W^{(n)}_{0,a,b-s}\cdot g_{J-1}^{-1}\ldots g_1^{-1}t_1^sX_1^{-1}\Bigr) & \ \ \ \text{if $J>0$\,.}
\end{array}\right.
\end{equation}

We can perform similar straightforward calculations for $\widetilde{\mathcal{B}}^{\text{Ind}}_{d,m,n}$ to prove that the $\cR_m$-span of the elements in $\widetilde{\mathcal{B}}^{\text{Ind}}_{d,m,n}$ is stable by multiplication (from the left) by the generators. We only indicate that we have to use Formula (\ref{form1}) in Lemma \ref{lem-form}, instead of (\ref{form3}), for the multiplication by $X_1$. We skip the details.
Then it remains to prove that the unit element belongs to the $\cR_m$-span of the elements in $\widetilde{\mathcal{B}}^{\text{Ind}}_{d,m,n}$. For $n=1$, 
$1=\widetilde{W}^{(1)}_{0,0,0}$.
 Then we notice that, for $k=2,\dots,n$, we have $1=g_{k-1}^{-1}g_{k-2}^{-1}\dots g_1^{-1} \widetilde{W}^{(k)}_{0,0,0}$. So by induction on $n$ and the stability property, this yields the desired result. 

Finally, the generating property for $\mathcal{B}^{\text{Ind}-}_{d,m,n}$ and $\widetilde{\mathcal{B}}^{\text{Ind}-}_{d,m,n}$ follows from the results for $\mathcal{B}^{\text{Ind}}_{d,m,n}$ and $\widetilde{\mathcal{B}}^{\text{Ind}}_{d,m,n}$, applying the ring homomorphism $\eta$ of  ${\rm Y}(d,m,n)$ defined in (\ref{eta}).
\end{proof}

\begin{rem}\label{rem-bases} {\rm
Let $E'_m$ be a subset of $\Z$ such that $\{X_1^a\,|\,a\in E'_m\}$ is an $\cR_m$-basis of $\cR_m[X_1^{\pm1}]$. Denote respectively by 
$\mathcal{B}^{\text{AK}}_{d,m,n}(E'_m)$,  
$\mathcal{B}^{\text{Ind}}_{d,m,n}(E'_m)$, $\mathcal{B}^{\text{Ind}-}_{d,m,n}(E'_m)$, $\widetilde{\mathcal{B}}^{\text{Ind}}_{d,m,n}(E'_m)$ and $\widetilde{\mathcal{B}}^{\text{Ind}-}_{d,m,n}(E'_m)$ the sets of elements as in (\ref{base-AK})--(\ref{base-ind4}), 
but with the conditions $a_k\in E_m$ replaced by $a_k\in E'_m$. Then the proof of Proposition \ref{prop-gen} extends immediately to show that these sets of elements are also generating sets (over $\cR_m$) of  the algebra ${\rm Y}(d,m,n)$. If $m=\infty$, the only choice is $E'_{\infty}=E_{\infty}=\Z$. 
If $m<\infty$, the two relations obtained by applying the ring homomorphism $\eta$
to (\ref{X1m})--(\ref{inv-X1}) imply that we can take $E'_m=-E_m$. Moreover, if $m$ is odd, we can take $E'_m=\{0,\pm1,\pm2,\dots,\pm \frac{m-1}{2}\}$. Indeed, by multiplying (\ref{X1m}) by $X_1^{-(m-1)/2}$, we obtain $X_1^{(m+1)/2}$ as a linear combination of elements in $\{X_1^a\,|\,a\in E'_m\}$.
By further applying $\eta$, we obtain $X_1^{-(m+1)/2}$ as a linear combination of elements in $\{X_1^a\,|\,a\in E'_m\}$.
The set $E_m'=\{0,\pm1,\pm2,\dots,\pm \frac{m-1}{2}\}$ for $m$ odd will be used in the proof of Theorem \ref{theo-bases}.
\hfill$\triangle$
}\end{rem}

\subsection{Bases}

We will use the representation theory developed in the previous section to prove that each generating set in Proposition \ref{prop-gen} is actually a basis of ${\rm Y}(d,m,n)$. Recall that,  in Section \ref{sec-rep}, we constructed a set $\{V_{\blambda}\}_{\blambda  \in \mathcal{P}(d,m,n)}$ of pairwise non-isomorphic irreducible representations of $\cF_m{\rm Y}(d,m,n)$ for $m<\infty$. The following standard equality holds:
\begin{equation}\label{dim-rep}\sum_{\blambda  \in \mathcal{P}(d,m,n)}\bigl(\dim(V_{\blambda})\bigr)^2=(dm)^nn!\,\ \ \ \ \ \text{for $m<\infty$\,.}\end{equation}
\begin{thm}\label{theo-bases}
Each set $\mathcal{B}^{\text{AK}}_{d,m,n}$, $\mathcal{B}^{\text{Ind}}_{d,m,n}$, $\mathcal{B}^{\text{Ind}-}_{d,m,n}$, $\widetilde{\mathcal{B}}^{\text{Ind}}_{d,m,n}$ and $\widetilde{\mathcal{B}}^{\text{Ind}-}_{d,m,n}$ is an $\cR_m$-basis of ${\rm Y}(d,m,n)$. In particular, ${\rm Y}(d,m,n)$ is a free $\cR_m$-module and, if $m<\infty$, its rank is equal to $(dm)^nn!$\,.
\end{thm}
\begin{proof}
First assume that $m<\infty$. Due to (\ref{dim-rep}), we have that $\dim(\cF_m{\rm Y}(d,m,n))\geq (dm)^nn!$. By Proposition \ref{prop-gen}, each set $\mathcal{B}^{\text{AK}}_{d,m,n}$, $\mathcal{B}^{\text{Ind}}_{d,m,n}$, $\mathcal{B}^{\text{Ind}-}_{d,m,n}$, $\widetilde{\mathcal{B}}^{\text{Ind}}_{d,m,n}$ and $\widetilde{\mathcal{B}}^{\text{Ind}-}_{d,m,n}$ is a generating set of $\cF_m{\rm Y}(d,m,n)$ over $\cF_m$ and contains exactly $(dm)^nn!$ elements. Thus, the elements of each set are linearly independent over $\cF_m$, and in turn over $\cR_m$.

Now let $m=\infty$. We know by Proposition \ref{prop-gen} that each set $\mathcal{B}^{\text{AK}}_{d,\infty,n}$, $\mathcal{B}^{\text{Ind}}_{d,\infty,n}$, $\mathcal{B}^{\text{Ind}-}_{d,\infty,n}$, $\widetilde{\mathcal{B}}^{\text{Ind}}_{d,\infty,n}$ and $\widetilde{\mathcal{B}}^{\text{Ind}-}_{d,\infty,n}$ is a generating set of ${\rm Y}(d,\infty,n)$, so it remains only to prove the linear independence over $\cR_{\infty}$.
Note that $\cR_{\infty}\subset \cR_{m}$ and that, for any $m<\infty$, ${\rm Y}(d,m,n)$ is the quotient of $\cR_m\otimes_{\cR_{\infty}}{\rm Y}(d,\infty,n)$ over the last relation in (\ref{def-aff2}). Denote by $\pi^{(m)}$ the associated surjective homomorphism.

First assume that a (non-trivial) linear combination $\mathbf{L}$ over $\cR_{\infty}$ of elements of $\mathcal{B}^{\text{AK}}_{d,\infty,n}$  is equal to $0$. By multiplying $\mathbf{L}$ from the left by large enough positive powers of elements $X_i$, we can assume that only non-negative powers of the generators $X_i$ appear in it. Then taking $m_0 \in \mathbb{Z}_{>0}$ larger than any powers of the $X_i$ appearing in $\mathbf{L}$ and applying $\pi^{(m_0)}$ to $\mathbf{L}$, it implies a dependence relation $\pi^{(m_0)}(\mathbf{L})=0$ over $\cR_{\infty}\subset \cR_{m_0}$ between elements of the basis $\mathcal{B}^{\text{AK}}_{d,m_0,n}$ of ${\rm Y}(d,m_0,n)$. This contradicts the first part of the proof.

Now, for any of the sets $\mathcal{B}^{\text{Ind}}_{d,\infty,n}$, $\mathcal{B}^{\text{Ind}-}_{d,\infty,n}$, $\widetilde{\mathcal{B}}^{\text{Ind}}_{d,\infty,n}$ and $\widetilde{\mathcal{B}}^{\text{Ind}-}_{d,\infty,n}$, assume that a (non-trivial) linear combination $\mathbf{L}$ of its elements over $\cR_{\infty}$ is equal to $0$. Let $m_+\in\Z$ be the largest absolute value of the powers of $X_1$ appearing in $\mathbf{L}$ and take 
$m_0:=2m_+ + 1$ (or any larger odd integer). Then applying $\pi^{(m_0)}$ to $\mathbf{L}$, it implies a dependence relation $\pi^{(m_0)}(\mathbf{L})=0$ over $\cR_{\infty}\subset \cR_{m_0}$ between elements of one of the sets  $\mathcal{B}^{\text{Ind}}_{d,m,n}(E'_{m_0})$, $\mathcal{B}^{\text{Ind}-}_{d,m,n}(E'_{m_0})$, $\widetilde{\mathcal{B}}^{\text{Ind}}_{d,m,n}(E'_{m_0})$ and $\widetilde{\mathcal{B}}^{\text{Ind}-}_{d,m,n}(E'_{m_0})$ for 
$E'_{m_0}=\{0, \pm1,\pm2,\dots,\pm m_+ \}$ 
(see Remark \ref{rem-bases}). This is a contradiction as these sets are generating sets of ${\rm Y}(d,m_0,n)$ containing $(dm_0)^nn!$ elements, hence are bases of ${\rm Y}(d,m_0,n)$ due to the first part of the proof.
\end{proof}

\begin{rem}\label{appendices}{\rm The proof  of Theorem \ref{theo-bases} relies on the representation theory of the cyclotomic Yokonuma--Hecke algebras ${\rm Y}(d,m,n)$ (for  finite $m$) and uses the dimension formula (\ref{dim-rep}). However, one can prove directly the linear independence of the sets $\mathcal{B}^{\text{AK}}_{d,m,n}$, $\mathcal{B}^{\text{Ind}}_{d,m,n}$, $\mathcal{B}^{\text{Ind}-}_{d,m,n}$, $\widetilde{\mathcal{B}}^{\text{Ind}}_{d,m,n}$ and $\widetilde{\mathcal{B}}^{\text{Ind}-}_{d,m,n}$ for any $m$ (finite or infinite), by defining an explicit representation of ${\rm Y}(d,m,n)$  and checking that the images of the elements of these sets are linearly independent operators. For this type of arguments in the case of the Ariki--Koike algebra and the affine Hecke algebra of 
${\rm GL}$, one can see  \cite{OgPo1}.
 \hfill$\triangle$
}\end{rem}

\subsection{A semisimplicity criterion for the cyclotomic Yokonuma--Hecke algebra}

Let us consider again the cyclotomic Yokonuma--Hecke algebra ${\rm Y}(d,m,n)$. Theorem \ref{theo-bases}, in combination with (\ref{dim-rep}), implies that
\begin{equation}\label{dim-rep2}\sum_{\blambda  \in \mathcal{P}(d,m,n)}\bigl(\dim(V_{\blambda})\bigr)^2= \dim(\cF_m{\rm Y}(d,m,n)) \,\ \ \ \ \ \text{for $m<\infty$\,.}\end{equation}
The following result is a direct consequence of (\ref{dim-rep2}).

\begin{prop}\label{s-s}
For $m<\infty$, the algebra $\cF_m{\rm Y}(d,m,n)$ is semisimple and
the set $\{V_{\blambda}\}_{\blambda  \in \mathcal{P}(d,m,n)}$  is a complete set of pairwise non-isomorphic irreducible representations of $\cF_m{\rm Y}(d,m,n)$.
\end{prop}

We now use the semisimplicity criterion for the Ariki--Koike algebra ${\rm H}(m,n)$ given in \cite{Ar} to obtain a semisimplicity criterion for the cyclotomic Yokonuma--Hecke algebra ${\rm Y}(d,m,n)$. 
The criteria turn out to be the same.

\begin{prop}\label{semisimplicity criterion} Let $m < \infty$ and
let  $\vartheta: \cR_m \rightarrow \mathbb{C}$ be a $\C$-algebra homomorphism.
We consider the specialised cyclotomic Yokonuma--Hecke algebra ${\rm Y}_\vartheta:= \mathbb{C} \otimes_{ \mathbb{C}[q,q^{-1}]}  {\rm Y}(d,m,n)$, defined via $\vartheta$.
The algebra ${\rm Y}_\vartheta$ is (split) semisimple if and only if $\vartheta(P)\neq 0$, where 
$$P=\prod_{1\leq k \leq n}(1+q^2+\cdots+q^{2(k-1)}) \prod_{0 \leq s <t < m}\prod_{-n<l<n}(q^{2l}v_s-v_t)\ .$$
\end{prop}
 
\begin{proof}
Ariki's semisimplicity criterion \cite{Ar} states that the specialised Ariki--Koike algebra ${\rm H}_\vartheta:= \mathbb{C} \otimes_{ \mathbb{C}[q,q^{-1}]} {\rm H}(m,n)$, defined via $\vartheta$, is semisimple if and only if $\vartheta(P)\neq 0$. Since ${\rm H}(m,n)$ is a quotient of the algebra ${\rm Y}(d,m,n)$, we obtain that if 
${\rm Y}_\vartheta$ is semisimple, then ${\rm H}_\vartheta$ is also semisimple and $\vartheta(P)\neq 0$.
For the converse statement, we will use the following lemma.

\begin{lem}\label{lem-ss}
Let $\bmu$ be a $(d,m)$-partition of size $N-1$, and let $\btheta$ and $\btheta'$ be two distinct $(d,m)$-nodes addable to $\bmu$ such that $\pos^{(d)}(\btheta)= \pos^{(d)}(\btheta')$. The following hold: \begin{enumerate}[(a)]
\item If $\pos^{(m)}(\btheta)= \pos^{(m)}(\btheta')$, then $\cc(\btheta)/\cc(\btheta')=q^{2k}$ for some $k\in\Z$ such that $|k|\in\{1,\dots,N\}$.  \smallbreak
\item If $\pos^{(m)}(\btheta)\neq \pos^{(m)}(\btheta')$, then $\cc(\btheta)/\cc(\btheta')=q^{2l}v_s/v_t$ for some $l\in\Z$ such that $|l|\in\{0,1,\dots,N-1\}$, and some $s,t\in\{1,\dots,m\}$ such that $s\neq t$. 
\end{enumerate}
\end{lem}
\begin{proof}[Proof of Lemma \ref{lem-ss}]
Let $\cc(\btheta)=v_sq^{2x}$ and $\cc(\btheta')=v_tq^{2y}$ for some $s,t\in\{1,\dots,d\}$ and $x,y \in\Z$. 

If $s=t$, then $x\neq y$ and we have $\cc(\btheta)/\cc(\btheta')=q^{2(x-y)}$. Moreover, $|x-y|-1$ is the number of diagonals strictly between the diagonal of $\btheta$ and the diagonal of $\btheta'$. All these diagonals have to be occupied by at least one $(d,m)$-node of $\bmu$, so we must have $|x-y|-1\leq N-1$. 

If $s\neq t$, then we have $\cc(\btheta)/\cc(\btheta')=q^{2(x-y)}v_s/v_t$. Let $\epsilon_x$ and $\epsilon_y$ denote respectively the signs of $x$ and $y$. The diagonals with content $v_s,v_sq^{2\epsilon_x},\dots,v_sq^{2\epsilon_x(|x|-1)}$ have to be occupied by at least one $(d,m)$-node of $\bmu$. Similarly, the diagonals with content $v_t,v_tq^{2\epsilon_{y}},\dots,v_tq^{2\epsilon_{y}(|y|-1)}$ have to be occupied by at least one $(d,m)$-node of $\bmu$. So we obtain $|x|+|y|\leq N-1$, which yields the second assertion of the lemma.
\end{proof}

We return to the proof of Proposition \ref{semisimplicity criterion}. Assume that $\vartheta(P)\neq 0$. Set $\oq:=\vartheta(q)$. If $\oq^2=1$,
the specialised algebra ${\rm Y}_\vartheta$  is the crossed product of a finite-dimensional commutative algebra (generated by $t_1,\dots,t_n,X_1,\dots,X_n$) by the symmetric group acting by permutation. It is therefore semisimple. So we can assume, in addition to $\vartheta(P)\neq 0$, that $\oq^2\neq1$. Then we have $\oq^{2N}\neq1$ for any integer $N$ such that $|N|\leq n$.

First, in order to be able to construct representations $V^{\vartheta}_{\blambda}$ of ${\rm Y}_\vartheta$ as in Proposition \ref{prop-rep} (with the parameters specialised via $\vartheta$), we must have, for any standard $(d,m)$-tableau $\cT$ of size $n$,
\begin{equation}\label{ss-cond1}
\vartheta\bigl(\cc(\cT|i)\bigr)\neq\vartheta\bigl(\cc(\cT|i+1)\bigr) ,
\end{equation}
for any $i=1,\dots,n-1$ such that $\posd(\cT|i)=\posd(\cT|i+1)$. If the $(d,m)$-nodes with entries $i$ and $i+1$ lie in adjacent diagonals, Equation (\ref{ss-cond1}) follows from $\oq^2\neq 1$. Otherwise, as $\cT$ is standard, these $(d,m)$-nodes are both addable to the $(d,m)$-partition $\bmu_{i-1}$ of size $i-1$ that is obtained by keeping only the $(d,m)$-nodes of $\cT$ containing $1,\dots,i-1$. Then Equation (\ref{ss-cond1}) follows from Lemma \ref{lem-ss} applied to $\bmu_{i-1}$ together with the assumption 
$\vartheta(P)\neq 0$.

Second, in order to be able to repeat the proof of Proposition \ref{prop-dis-irr} concerning the irreducibility of the representations $V^{\vartheta}_{\blambda}$, we must have
\begin{equation}\label{ss-cond2}
\vartheta\bigl(\cc(\cT|n-1)\bigr)\neq\oq^{\pm2}\vartheta\bigl(\cc(\cT|n)\bigr)\ ,
\end{equation}
for any standard $(d,m)$-tableau $\cT$ of size $n$ such that 
$\posd(\cT|n-1)=\posd(\cT|n)$ and $\cT^{s_{n-1}}$ is standard. If $\cT$ is a standard $(d,m)$-tableau of size $n$ such that $\cT^{s_{n-1}}$ is standard, then the $(d,m)$-nodes with entries $n-1$ and $n$ are both addable to the $(d,m)$-partition $\bmu_{n-2}$ of size $n-2$ that  is obtained by keeping only the $(d,m)$-nodes of $\cT$ containing $1,\dots,n-2$. Then Equation (\ref{ss-cond2}) follows from Lemma \ref{lem-ss} applied to $\bmu_{n-2}$ together with the assumption $(\oq^2-1)\vartheta(P)\neq 0$.

By induction on the size and with the help of Lemma \ref{lem-ss}, we straightforwardly have that any standard
$(d,m)$-tableau $\cT$ of size $n$ is fully characterised by its sequence of specialised content arrays
\begin{equation}\label{T-cont-spec}
\left(\left(\begin{array}{l}\posd(\cT|1)\\ \vartheta\bigl(\cc(\cT|1)\bigr)\end{array}\right),\ \left(\begin{array}{l}\posd(\cT|2)\\ \vartheta\bigl(\cc(\cT|2)\bigr)\end{array}\right),\ \dots\ ,\ \left(\begin{array}{l}\posd(\cT|n)\\ \vartheta\bigl(\cc(\cT|n)\bigr)\end{array}\right)\right)\ .
\end{equation}
This implies, as in the proof of Proposition \ref{prop-dis-irr}, that the constructed irreducible representations $V^{\vartheta}_{\blambda}$  of ${\rm Y}_\vartheta$ are pairwise non-isomorphic.

Finally, the semisimplicity of the algebra ${\rm Y}_\vartheta$ under the assumption $(\oq^2-1)\vartheta(P)\neq 0$ follows from Equation (\ref{dim-rep}) together with the fact, implied by Theorem \ref{theo-bases}, that the dimension of ${\rm Y}_\vartheta$ is equal to $(dm)^nn!$.
\end{proof}

\section{Markov traces on ${\rm Y}(d,m,n)$}\label{sec-markov}

Now we are ready to define and study Markov traces on the cyclotomic and affine Yokonuma--Hecke algebras. In order to define a Markov trace $\tr : {\rm Y}(d,m,n) \rightarrow \cR_m$, we will define intermediary  linear maps $\tr_k: {\rm Y}(d,m,k) \rightarrow {\rm Y}(d,m,k-1)$ for $k \in \mathbb{Z}_{>0}$ with certain properties and then show that $\tr$ is in fact a composition of these maps.
In this section again, $m$ is arbitrary in $\Z_{>0}\cup\{\infty\}$.

\subsection{Chains of relative traces}

For each of the bases of ${\rm Y}(d,m,n)$ studied in the previous section, the basis elements involving only the generators $t_1,\ldots,t_{n-1},g_1,\ldots,g_{n-2},X_1^{\pm1}$ are in one-to-one correspondence with the elements of the corresponding basis of ${\rm Y}(d,m,n-1)$. Thus, the subalgebra of ${\rm Y}(d,m,n)$ generated by $t_1,\ldots,t_{n-1},g_1,\ldots,g_{n-2},X_1^{\pm1}$ is isomorphic to ${\rm Y}(d,m,n-1)$. This allows to consider the chain (on $n$) 
of algebras
\begin{equation}\label{chainproperty}
{\rm Y}(d,m,0):=\cR_m\subset {\rm Y}(d,m,1)\subset\cdots\subset{\rm Y}(d,m,n-1)\subset{\rm Y}(d,m,n)\subset\cdots,
\end{equation}
where the inclusion monomorphisms are given by ${\rm Y}(d,m,n-1)\ni x\mapsto x\in{\rm Y}(d,m,n)$ for any $x\in\{t_1,\ldots,t_{n-1},g_1,\ldots,g_{n-2},X_1^{\pm 1}\}$. Thus, in what follows, we will very often consider elements of ${\rm Y}(d,m,n)$ as elements of ${\rm Y}(d,m,n')$ for any $n' \geq n$.

\begin{defn}{\rm
Let $z$ and $x_{a,b}$, with $a\in E_m$ and $b\in \{0,\ldots,d-1\}$, be parameters in $\cR_{m}$.
A \emph{chain of relative traces (with parameters $z$ and $x_{a,b}$)} is a set of $\cR_m$-linear maps 
$\{\tr_k\}_{k \in \Z_{>0}}$ where 
$$\tr_k: {\rm Y}(d,m,k) \rightarrow {\rm Y}(d,m,k-1),$$
satisfying:
 \begin{equation}\label{1}
\tr_1(X_1^at_1^{b})=x_{a,b}\ \ \text{ for } a\in E_m,\ b \in \{0,\ldots,d-1\},
 \end{equation}
 and, for $k\geq  2$, $u, v \in {\rm Y}(d,m,k-1)$ and $Z \in {\rm Y}(d,m,k)$,
  \begin{equation}\label{2}
  \tr_k(uZv) = u\,\tr_k(Z)\,v,
  \end{equation}
   \begin{equation}\label{3}
  \tr_k (g_{k-1}^\varepsilon u g_{k-1}^{-\varepsilon }) = \tr_{k-1}(u)\ \ \  \text{ for }  \varepsilon = \pm 1,
  \end{equation}
    \begin{equation}\label{4}
    \tr_{k-1}\bigl(\tr_k(g_{k-1}Z)\bigr) = \tr_{k-1}\bigl(\tr_k(Zg_{k-1})\bigr), 
  \end{equation}
    \begin{equation}\label{5}
   \tr_k(g_{k-1}) = z.
  \end{equation}}
\end{defn}
\begin{rem}{\rm
We note that, by (\ref{1}), we have $\tr_1(1)=x_{0,0}$, and morever, by applying (\ref{3}) with $u=1$, we obtain $\tr_k(1)=\tr_{k-1}(1)$ for any $k\geq2$. Thus, we have $\tr_k(1)=x_{0,0}$ for any $k\geq 1$. More generally, using (\ref{2}) with $Z=v=1$, we have
\begin{equation}\label{rem-tr}
\tr_k(u)=x_{0,0}\,u\ \ \ \ \ \ \text{for any $u\in{\rm Y}(d,m,k-1)$.}
\end{equation}
We will impose $x_{0,0}=1$ later in Subsection \ref{subsec-Markov}.
\hfill$\triangle$ 
}\end{rem}

We will now prove the existence and uniqueness of relative traces. For this, we are going to use the
elements $W^{(k)}_{J,a,b}\in {\rm Y}(d,m,k)$, $k\in \Z_{>0}\,$, defined in Section \ref{sec-base} by:
\[W^{(k)}_{J,a,b}:=  
g_{J}^{-1}\ldots g_2^{-1}g_1^{-1}X_1^at_1^b\,g_1g_2\ldots g_{k-1},\]
where $J\in\{0,\ldots,k-1\}$ and $a,b\in\Z$. 
It follows from Theorem \ref{theo-bases} 
that, for any $k\in \Z_{>0}$ and any basis $\mathcal{B}_{k-1}$ of ${\rm Y}(d,m,k-1)$, the set of elements
\begin{equation}\label{base}
W^{(k)}_{J,a,b}\,w\ \ \ \ \ \ \text{with}\ J\in\{0,\ldots,k-1\},\ a\in E_m,\ b \in \{0,\ldots,d-1\}\ \ \text{and}\ \ w\in\mathcal{B}_{k-1}\,,
\end{equation}
forms a basis of ${\rm Y}(d,m,k)$. Recall that the left action of the generators of ${\rm Y}(d,m,k)$ on these elements is given by Formulas (\ref{mult-t})--(\ref{mult-X2}).

\begin{prop}\label{prop-RT}
Let $z$ and $x_{a,b}$, with $a\in E_m$ and $b\in \{0,\ldots,d-1\}$, be parameters in $\cR_{m}$.
There exists a unique chain of relative traces $\tr_k$ with parameters $z$ and $x_{a,b}$, and it is given, for any $k\geq1$, by
\begin{equation}\label{RT1}\tr_k (W^{(k)}_{J,a,b}\,w)= z\,W^{(k-1)}_{J,a,b}w \quad \text{if}\ \,0\leq J<k-1, \end{equation}
\begin{equation}\label{RT2}\tr_k (W^{(k)}_{J,a,b}\,w)= x_{a,b}\,w \quad \text{if}\ J=k-1,
\end{equation}
where $a\in E_m$, $b\in \{0,\ldots,d-1\}$ and $w\in {\rm Y}(d,m,k-1)$.
\end{prop}
The remaining of this subsection is devoted to the proof of this proposition. We define for later use $x_{a,b}\in\cR_{m}$ for any $a,b\in\Z$ by 
$$x_{a,b}:=\tr_1(X_1^at_1^{b}),\ \ \ \ \ a,b\in\Z,$$ 
where $\tr_1$ is given on ${\rm Y}(d,m,1)$ by (\ref{RT2}) with $k=1$. Note that $x_{a,b+d}=x_{a,b}$ for any $a,b\in\Z$ and that, if $m<\infty$ and $a\notin E_m$, then $x_{a,b}$ is an $\cR_{m}$-linear combination of $x_{a',b}$ with $a'\in E_m$ (this follows from 
Equations (\ref{X1m}) and (\ref{inv-X1})). 

We start with a lemma that we will use in the proof of the proposition.
\begin{lem}\label{lemm-RT1}
As a consequence of Formulas (\ref{RT1}) and (\ref{RT2}), we have, for $k\geq1$ and $Z\in{\rm Y}(d,m,k)$, 
\begin{equation}\label{eq-lemm-RT1}
\tr_k(t_kZ)=\tr_k(Zt_k) .
\end{equation} 
\end{lem}
\begin{proof} Since $\tr_k$ is a linear map,
it is enough to take $Z=W^{(k)}_{J,a,b}\,w$ to be a basis element of ${\rm Y}(d,m,k)$ as in (\ref{base}). Note that $w$ commutes with $t_k$ since $w\in {\rm Y}(d,m,k-1)$. First, if $J=k-1$, then $t_kZ=Zt_k$ and Formula (\ref{eq-lemm-RT1}) follows. So let $J<k-1$. Then we have
\[\tr_k(t_kZ)=\tr_k\bigl(W^{(k)}_{J,a,b}t_{k-1}w\bigr)=zW^{(k-1)}_{J,a,b}t_{k-1}w=zW^{(k-1)}_{J,a,b+1}w.\]
On the other hand, we have
\[\tr_k(Zt_k)=\tr_k\bigl(W^{(k)}_{J,a,b+1}w\bigr)=zW^{(k-1)}_{J,a,b+1}w.\]
\end{proof}

\begin{proof}[Proof of Proposition \ref{prop-RT}]
Assume that a chain of relative traces $\tr_k$ with parameters $z$ and $x_{a,b}$ exists. Then Equation (\ref{RT1}) is a direct consequence of (\ref{2}) (with $u=g_{J}^{-1}\ldots g_1^{-1}X_1^a t_1^b\,g_1\ldots g_{k-2}$, $Z=g_{k-1}$ and $v=w$) and (\ref{5}). Equation (\ref{RT2}) is obtained by first applying (\ref{2}) with $u=1$, $Z=W^{(k)}_{k-1,a,b}$ and $v=w$, then repeating  (\ref{3}) $k-1$ times and finally using (\ref{1}). Since the set of elements (\ref{base}) is a basis of ${\rm Y}(d,m,k)$, Equations (\ref{RT1}) and (\ref{RT2}) uniquely define the chain of relative traces $\tr_k$ if it exists.

We now assume that a set of linear maps $\tr_k$, $k\in \Z_{>0}\,$, is defined by (\ref{RT1})--(\ref{RT2}) and we will show, in order to prove the proposition, that these linear maps satisfy (\ref{1})--(\ref{5}).  
Equations (\ref{1}) and (\ref{5}) are obviously satisfied, being respectively, Equation (\ref{RT2}) for $k=1$, and Equation (\ref{RT1}) for $J=k-2$, $a=b=0$ and $w=1$. 

\vskip .2cm
\paragraph{\textbf{Proof of (\ref{2})}} It is enough to take $u$ to be any generator of $Y(d, m, k -1)$, namely $u\in\{t_1,\ldots,t_{k-1},$ $g_1,\ldots,g_{k-2},X_1^{\pm1}\}$, and $Z=W^{(k)}_{J,a,b}\,w$ a basis element of ${\rm Y}(d,m,k)$ as in (\ref{base}). 

We consider first $J<k-1$. From Formulas (\ref{mult-t})--(\ref{mult-X2}), it is immediate to see that $uZv$ is a linear combination of elements of the form $W^{(k)}_{J',a',b'}\,w'$ with $J'<k-1$, and moreover that $u\,W^{(k-1)}_{J,a,b}w\,v$ is the same linear combination with every $W^{(k)}_{J',a',b'}\,w'$  replaced by $W^{(k-1)}_{J',a',b'}\,w'$. Thus, using (\ref{RT1}), it implies that $\tr_k(uZv)=z u\,W^{(k-1)}_{J,a,b}w\,v= u\,\tr_k(Z)\,v$.

Let now $J=k-1$. If $u\in\{t_1,\ldots,t_{k-1},g_1,\ldots,g_{k-2}\}$ then Formulas (\ref{mult-t})--(\ref{mult-g})  yield  $u\,W^{(k)}_{k-1,a,b}w=W^{(k)}_{k-1,a,b}uw$. Therefore, using (\ref{RT2}), $\tr_k(uZv)=x_{a,b}\,uwv=u\,\tr_k(Z)\,v$.

If $u=X_1$, then, using the second line in (\ref{mult-X}) with $J=k-1$ and applying (\ref{RT1})--(\ref{RT2}), we obtain
\[\tr_k(X_1Zv)=x_{a,b}\,X_1wv+(q-q^{-1})\displaystyle\frac{z}{d}\sum\limits_{s=0}^{d-1}\Bigl(W^{(k-1)}_{0,1,b-s}\cdot g_{k-2}^{-1}\ldots g_1^{-1}t_1^sX_1^a - W^{(k-1)}_{0,a+1,b-s}\cdot g_{k-2}^{-1}\ldots g_1^{-1}t_1^s\Bigr)wv.\]
Since $W^{(k-1)}_{0,a',b'}\cdot g_{k-2}^{-1}\ldots g_1^{-1}=X_1^{a'}t_1^{b'}$ for any 
$a',b' \in\Z$, 
the above formula becomes simply $\tr_k(X_1Zv)=x_{a,b}\,X_1wv$, which is equal to $X_1\tr_k(Z)v$.

Finally, if we take $u=X_1^{-1}$, we have, for all $Z \in {\rm Y}(d,m,k)$ and all $v \in {\rm Y}(d,m,k-1)$,
$$\tr_k (X_1^{-1}Zv) = X_1^{-1} X_1 \tr_k (X_1^{-1}Zv)  = X_1^{-1}  \tr_k (X_1X_1^{-1}Zv)= X_1^{-1} \tr_k(Zv)=X_1^{-1} \tr_k(Z)v,$$ 
where we use the already proved (\ref{2}) for $u=X_1$ and for $u=1$ in the second and fourth equality respectively.
This concludes the verification of (\ref{2}).

\vskip .3cm
We record here some useful consequences of Formula (\ref{2}) combined with Lemma \ref{lemm-RT1}.
 \begin{lem}\label{lemm-RT2}
As a consequence of Formulas (\ref{RT1}) and (\ref{RT2}), we have, for $k\geq2$ , $u\in{\rm Y}(d,m,k-1)$ and $Z\in{\rm Y}(d,m,k)$,
\begin{equation}\label{eq1-lemm-RT2}
\tr_k(e_{k-1}ug_{k-1})=\tr_k(g_{k-1}ue_{k-1}) \ ;
\end{equation}
\begin{equation}\label{eq2-lemm-RT2}
\tr_{k-1}\bigl(\tr_k(e_{k-1}Z)\bigr)=\tr_{k-1}\bigl(\tr_k(Ze_{k-1})\bigr).
\end{equation}
\end{lem}
\begin{proof}[Proof of Lemma \ref{lemm-RT2}] Note that $t_k$ commutes with every element $u\in{\rm Y}(d,m,k-1)$. Thus, the left hand side of (\ref{eq1-lemm-RT2}) is equal to
\[\frac{1}{d}\sum_{s=0}^{d-1}\tr_k(t_{k-1}^st_k^{-s}ug_{k-1})=\frac{1}{d}\sum_{s=0}^{d-1}\tr_k(t_{k-1}^sug_{k-1}t_{k-1}^{-s})=\frac{1}{d}\sum_{s=0}^{d-1}t_{k-1}^su\tr_k(g_{k-1})t_{k-1}^{-s}\,,\]
where we use the already proved Formula (\ref{2}) in the last equality. Similarly, the right hand side of (\ref{eq1-lemm-RT2}) is equal to
\[\frac{1}{d}\sum_{s=0}^{d-1}\tr_k(g_{k-1}ut_{k-1}^st_k^{-s})=\frac{1}{d}\sum_{s=0}^{d-1}\tr_k(t_{k-1}^{-s}g_{k-1}ut_{k-1}^{s})=\frac{1}{d}\sum_{s=0}^{d-1}t_{k-1}^{-s}\tr_k(g_{k-1})ut_{k-1}^{s}\,.\]
Therefore, Formula (\ref{eq1-lemm-RT2}) follows from the fact (included in (\ref{RT1})) that $\tr_k(g_{k-1})=z\in\cR_m$.

Using Formula (\ref{2}), we have
\[\tr_{k-1}\bigl(\tr_k(e_{k-1}Z)\bigr)=\frac{1}{d}\sum_{s=0}^{d-1}\tr_{k-1}\bigl(\tr_k(t_{k-1}^st_k^{-s}Z)\bigr)=\frac{1}{d}\sum_{s=0}^{d-1}\tr_{k-1}\bigl(t_{k-1}^s\tr_k(t_k^{-s}Z)\bigr),\]
\[\tr_{k-1}\bigl(\tr_k(Ze_{k-1})\bigr)=\frac{1}{d}\sum_{s=0}^{d-1}\tr_{k-1}\bigl(\tr_k(Zt_{k-1}^st_k^{-s})\bigr)=\frac{1}{d}\sum_{s=0}^{d-1}\tr_{k-1}\bigl(\tr_k(Zt_k^{-s})t_{k-1}^s\bigr).\]
Thus, Formula (\ref{eq2-lemm-RT2}) follows from two applications of Lemma \ref{lemm-RT1}.
\end{proof}

\paragraph{\textbf{Proof of (\ref{3})}}
It is enough to prove (\ref{3}) for  $u$ an element of the basis (\ref{base}) of ${\rm Y}(d,m,k-1)$ for $k\geq2$. So let $u=W_{J,a,b}^{(k-1)} w$, where $J\in\{0,\ldots,k-2\}$, $a \in E_m$, $b\in \{0,\ldots,d-1\}$ and $w \in {\rm Y}(d,m,k-2)$.

We first note that, using $g_{k-1}^{-1}=g_{k-1}-(q-q^{-1})e_{k-1}$, we have
$$ g_{k-1}ug_{k-1}^{-1}=g_{k-1}^{-1}ug_{k-1} + (q-q^{-1})(e_{k-1}ug_{k-1} - g_{k-1}u e_{k-1}).$$ 
Together with Formula (\ref{eq1-lemm-RT2}), it implies that
\[\tr_k(g_{k-1}^{-1}ug_{k-1})=\tr_k(g_{k-1}ug_{k-1}^{-1}).\]
Therefore, it is enough to prove (\ref{3}) for $\varepsilon = -1$. If $J=k-2$, then we have
$$\tr_k(g_{k-1}^{-1}ug_{k-1})=\tr_k(g_{k-1}^{-1}W^{(k-1)}_{k-2,a,b}wg_{k-1})=\tr_k(W^{(k)}_{k-1,a,b}w) =x_{a,b}w = \tr_{k-1}(u),$$
where we use the facts that $g_{k-1}$ commutes with $w$ and $g_{k-1}^{-1}W^{(k-1)}_{k-2,a,b}g_{k-1}=W^{(k)}_{k-1,a,b}$.
If $J < k-2$, then $k>2$ and, using the fact that $g_{k-1}^{-1}g_{k-2} g_{k-1}=g_{k-2} g_{k-1}g_{k-2}^{-1}$, we obtain
$$g_{k-1}^{-1}ug_{k-1}=g_{J}^{-1}\ldots g_2^{-1}g_1^{-1} X_1^a t_1^bg_1g_2\ldots g_{k-3}\cdot g_{k-1}^{-1}g_{k-2} g_{k-1}\cdot w=W^{(k)}_{J,a,b}\,  g_{k-2}^{-1}w.$$
Thus, we have
$$\tr_k(g_{k-1}^{-1}u g_{k-1})= z W^{(k-1)}_{J,a,b}g_{k-2}^{-1}w = z W^{(k-2)}_{J,a,b}w = \tr_{k-1}(u).$$

\vskip .2cm
\paragraph{\textbf{Proof of (\ref{4})}}
It is enough to prove (\ref{4}) for $Z$ an element of the basis (\ref{base}) of ${\rm Y}(d,m,k)$ for $k\geq2$. So let $Z=W^{(k)}_{J,a,b}w$, where $J\in\{0,\ldots,k-1\}$, $a \in E_m$, $b\in \{0,\ldots,d-1\}$ and $w \in {\rm Y}(d,m,k-1)$.

We note that, due to Formula (\ref{eq2-lemm-RT2}) together with $g_{k-1}^{-1}=g_{k-1}-(q-q^{-1})e_{k-1}$, Formula (\ref{4}) is equivalent to
\begin{equation}\label{new4}
  \tr_{k-1}\bigl(\tr_k(g_{k-1}^{-1}Z)\bigr) = \tr_{k-1}\bigl(\tr_k(Zg_{k-1}^{-1})\bigr).
 \end{equation}

Assume first that $J<k-1$ and note that $Z=W^{(k-1)}_{J,a,b}g_{k-1}w$. We will prove (\ref{new4}). 
By (\ref{2}) and (\ref{3}), we have
$$\tr_k(g_{k-1}^{-1}Z)=\tr_k(g_{k-1}^{-1}W^{(k-1)}_{J,a,b}g_{k-1}w)=
\tr_k(g_{k-1}^{-1}W^{(k-1)}_{J,a,b}g_{k-1})w=
\tr_{k-1}\bigl(W^{(k-1)}_{J,a,b}\bigr)w.$$
Now, since $\tr_{k-1}\bigl(W^{(k-1)}_{J,a,b}\bigr) \in {\rm Y}(d,m,k-2)$, we have by (\ref{2}) that
\begin{equation}\label{eq-proof1}
\tr_{k-1}\bigl(\tr_k(g_{k-1}^{-1}Z)\bigr) = \tr_{k-1}\bigl(W^{(k-1)}_{J,a,b}\bigr)\tr_{k-1}(w).
\end{equation}
On the other hand, again by (\ref{2}) and (\ref{3}), we have 
$$\tr_k(Zg_{k-1}^{-1})=\tr_k\bigl(W^{(k-1)}_{J,a,b}g_{k-1}wg_{k-1}^{-1}\bigr)=W^{(k-1)}_{J,a,b}\tr_k(g_{k-1}wg_{k-1}^{-1})=W^{(k-1)}_{J,a,b}\tr_{k-1}(w).$$
Now, since $\tr_{k-1}(w) \in {\rm Y}(d,m,k-2)$, we have by (\ref{2}) that
\begin{equation}\label{eq-proof2}
\tr_{k-1}\bigl(\tr_k(Zg_{k-1}^{-1})\bigr)=\tr_{k-1}\bigl(W^{(k-1)}_{J,a,b}\bigr)\tr_{k-1}(w).
\end{equation}
Thus, Formula (\ref{new4}), if $J<k-1$, follows from (\ref{eq-proof1}) and (\ref{eq-proof2}).

Assume now that $J=k-1$. In this case, we will prove (\ref{4}) directly. We will use the following lemma (which corresponds to the case $k=2$).
\begin{lem}\label{lemm-RT3}
As a consequence of Formulas (\ref{RT1}) and (\ref{RT2}), we have,
for any $a,a',b,b'\in\Z$, 
\begin{equation}\label{eq-lemm-RT3}
\tr_1\bigl(\tr_2(g_1^{-1}  X_1^a t_1^b g_1 X_1^{a'}  t_1^{b'} g_1)\bigr)=\tr_1\bigl(\tr_2(X_1^a t_1^b g_1 X_1^{a'} t_1^{b'})\bigr)=z\,x_{a+a',b+b'}.
\end{equation}
\end{lem}
\begin{proof}[Proof of Lemma \ref{lemm-RT3}] 
The second equality follows directly from (\ref{RT1}) and (\ref{RT2}), since 
\begin{center}
$\tr_2(X_1^a t_1^b g_1 X_1^{a'} t_1^{b'})=z\,X_1^{a+a'} t_1^{b+b'}$\quad and \quad$\tr_1(X_1^{a+a'} t_1^{b+b'})=x_{a+a',b+b'}\,$.
\end{center}
For the first equality, we start with
\[X_1^ag_1 X_1^{a'}g_1= g_1 X_1^{a'}g_1X_1^a + (q-q^{-1})e_1\sum_{i=1}^a(X_1^ig_1X_1^{a+a'-i} - X_1^{a+a'-i}g_1X_1^i).\]
This formula is easily proved by induction on $a$, the basis of induction for $a=1$ being Formula (\ref{form1}) with $b=0$.
Thus, using the fact that $g_1^{-1}$ commutes with $e_1$, we have
\[\begin{array}{rcl}
g_1^{-1} X_1^a t_1^b g_1 X_1^{a'} t_1^{b'}g_1 & = &t_1^{b'}g_1^{-1} X_1^ag_1 X_1^{a'}g_1t_1^b \\[0.0em] & &\\[-1.2em]
 & = & t_1^{b'}\Bigl(X_1^{a'}g_1X_1^a + (q-q^{-1})e_1\displaystyle\sum_{i=1}^a(g_1^{-1}X_1^ig_1X_1^{a+a'-i} - g_1^{-1}X_1^{a+a'-i}g_1X_1^i)\Bigr)t_1^{b}.
 \end{array}\]
 Writing each term of the form $e_1\,g_1^{-1}X_1^cg_1X_1^{c'}$ as $\displaystyle\frac{1}{d}\sum_{s=1}^{d}g_1^{-1}X_1^ct_1^sg_1\,X_1^{c'}t_1^{-s}$, we apply $\tr_2$ and we use (\ref{2}) together with (\ref{RT1})--(\ref{RT2}) to obtain
 \[\tr_2(g_1^{-1} X_1^a t_1^b g_1 X_1^{a'} t_1^{b'}g_1)=t_1^{b'}\Bigl(z\,X_1^{a+a'}+(q-q^{-1})\displaystyle\frac{1}{d}\sum_{i=1}^a\sum_{s=1}^{d}(x_{i,s}\,X_1^{a+a'-i}t_1^{-s}-x_{a+a'-i,s}\,X_1^{i}t_1^{-s})\Bigr)t_1^{b}.\]
 Now we note that, for each $i\in\{1,\ldots,a\}$, we have
 \[\tr_1\Bigl(t_1^{b'}\sum_{s=1}^{d}(x_{i,s}\,X_1^{a+a'-i}t_1^{-s}-x_{a+a'-i,s}\,X_1^{i}t_1^{-s})t_1^{b}\Bigr)=\sum_{s=1}^{d}(x_{i,s}\,x_{a+a'-i,b+b'-s}-x_{a+a'-i,s}\,x_{i,b+b'-s}),\]
 which is equal to $0$ since 
\begin{center}
 $\displaystyle\sum_{s=1}^{d}x_{a+a'-i,s}\,x_{i,b+b'-s}=\hspace{-0.4cm}\sum_{s'=b+b'-d}^{b+b'-1}\hspace{-0.4cm}x_{a+a'-i,b+b'-s'}\,x_{i,s'}$\quad and \quad$x_{i,j+d}=x_{i,j}$  \quad for any $i,j\in\Z$. 
 \end{center}
 We conclude that
\[\tr_1\bigl(\tr_2(g_1^{-1} X_1^a t_1^b g_1 X_1^{a'} t_1^{b'}g_1)\bigr)=\tr_1\Bigl(t_1^{b'}(z\,X_1^{a+a'})t_1^{b}\Bigr)=z\,x_{a+a',b+b'} ,\]
which completes the verification of Formula (\ref{eq-lemm-RT3}).
\end{proof}

We return to the 
proof of (\ref{4}) with $Z=W^{(k)}_{k-1,a,b}\,w$, where $k\geq2$, $a\in E_m$, $b\in\{0,\dots,d-1\}$ and $w\in{\rm Y}(d,m,k-1)$. 
It is enough to take $w$ to be an element of the basis (\ref{base}) of ${\rm Y}(d,m,k-1)$. So let $Z=W^{(k)}_{k-1,a,b}W^{(k-1)}_{J',a',b'} w'$, where $J'\in\{0,\ldots,k-2\}$, 
$a' \in E_m$, $b' \in \{0,\ldots,d-1\}$ 
and $w' \in {\rm Y}(d,m,k-2)$.
As $g_{k-1} W^{(k)}_{k-1,a,b}=W^{(k)}_{k-2,a,b}$, we have
\[\begin{array}{rcl}
\tr_k(g_{k-1}Z) & = & z\, W^{(k-1)}_{k-2,a,b}W^{(k-1)}_{J',a',b'} w'\\ & &\\
 & =&z\,g_{k-2}^{-1}\ldots g_2^{-1}g_1^{-1} X_1^at_1^bg_1g_2\ldots g_{k-2} g_{J'}^{-1}\ldots g_2^{-1}g_1^{-1} X_1^{a'}t_1^{b'}g_1g_2\ldots g_{k-2}w'.
 \end{array}\]
Hence, if $J'=k-2$, we have
\begin{equation}\label{eq-proof3}
\tr_k(g_{k-1}Z)=z\,W^{(k-1)}_{k-2,a+a',b+b'}w'.
\end{equation}
If now $J'<k-2$, we use the fact that (see (\ref{mult-g}))
\begin{center}
$W^{(k-1)}_{k-2,a,b}\,g_i=g_i\,W^{(k-1)}_{k-2,a,b}$ \quad for $i<k-2$  
\end{center}
to move $g_{J'}^{-1}\ldots g_2^{-1}g_1^{-1}$ to the left, and that 
\begin{center}
$g_i\, W^{(k-1)}_{0,a',b'}=W^{(k-1)}_{0,a',b'}\,g_{i-1}$ \quad for $1<i<k-1$ 
\end{center}
to move $g_2\ldots g_{k-2}$ to the right. We obtain that, if $J'<k-2$, $\tr_k(g_{k-1}Z)$ is equal to
\begin{equation}\label{eq-proof4}
z\,g_{J'}^{-1}\ldots g_2^{-1}g_1^{-1}\cdot g_{k-2}^{-1}\ldots g_2^{-1}g_1^{-1} X_1^at_1^b\,g_1\,X_1^{a'}t_1^{b'}g_1g_2\ldots g_{k-2}\cdot g_1\ldots g_{k-3 } w'.
\end{equation}
Now we apply $\tr_{k-1}$ to (\ref{eq-proof3}) and (\ref{eq-proof4}) and we use (\ref{2}) and (\ref{3}) to find that the left-hand side of (\ref{4}) 
is equal to
\begin{equation}\label{eq-proof5}
 \tr_{k-1}\bigl(\tr_k(g_{k-1}Z)\bigr) = 
\left\{\begin{array}{ll}
z\,x_{a+a',b+b'}\,w' & \text{if $J'=k-2$\,,}\\ &\\
z\,g_{J'}^{-1}\ldots g_2^{-1}g_1^{-1}\cdot\tr_2\bigl(g_1^{-1} X_1^at_1^b\,g_1\,X_1^{a'}t_1^{b'}g_1\bigr)\cdot g_1\ldots g_{k-3 }w' & \text{if $J'<k-2$\,.}
\end{array}\right.
\end{equation}

For the right-hand side of (\ref{4}), note first that $g_{k-1}$ commutes with $w'$ and thus, 
\[Zg_{k-1}=g_{k-1}^{-1}\ldots g_2^{-1}g_1^{-1} X_1^at_1^bg_1g_2\ldots g_{k-1} g_{J'}^{-1}\ldots g_2^{-1}g_1^{-1} X_1^{a'}t_1^{b'}g_1g_2\ldots g_{k-1}\,w'=W^{(k)}_{k-1,a,b}W^{(k)}_{J',a',b'} w'.\] 
Then we use  the fact that (see (\ref{mult-g}))
\begin{center}
 $W^{(k)}_{k-1,a,b}\,g_i=g_i\,W^{(k)}_{k-1,a,b}$ \quad for $i<k-1$ 
 \end{center}
 to move $g_{J'}^{-1}\ldots g_2^{-1}g_1^{-1}$ to the left, and that 
 \begin{center}
 $g_i \,W^{(k)}_{0,a',b'}=W^{(k)}_{0,a',b'}\,g_{i-1}$\quad  for $1<i<k$
 \end{center}
 to move $g_2\ldots g_{k-1}$ to the right. We obtain
\[Zg_{k-1}=g_{J'}^{-1}\ldots g_2^{-1}g_1^{-1}\cdot g_{k-1}^{-1}\ldots g_2^{-1}g_1^{-1} X_1^at_1^b\,g_1\, X_1^{a'}t_1^{b'}g_1g_2\ldots g_{k-1}\cdot g_1g_2\ldots g_{k-2}\,w'.\] 
Now we apply $\tr_k$ and use first (\ref{2}) and then (\ref{3}) $k-2$ times. This yields:
\[\tr_k(Zg_{k-1})=g_{J'}^{-1}\ldots g_2^{-1}g_1^{-1}\cdot\tr_2(g_1^{-1} X_1^at_1^b\,g_1\, X_1^{a'}t_1^{b'}g_1)\cdot g_1\ldots g_{k-2}\,w'.\] 
Finally, we apply $\tr_{k-1}$ and we use (\ref{2}) and (\ref{3}) (note that $\tr_2(g_1^{-1} X_1^at_1^b\,g_1\, X_1^{a'}t_1^{b'}g_1)\in{\rm Y}(d,m,1)$). We obtain that the right hand side of (\ref{4}) is equal to
\begin{equation}\label{eq-proof6}
\tr_{k-1}\bigl(\tr_k(Zg_{k-1})\bigr)=
\left\{\begin{array}{ll}
\tr_1\bigl(\tr_2(g_1^{-1} X_1^at_1^b\,g_1\, X_1^{a'}t_1^{b'}g_1)\bigr)\,w' & \text{if $J'=k-2$\,,}\\ & \\
z\,g_{J'}^{-1}\ldots g_2^{-1}g_1^{-1}\cdot\tr_2\bigl(g_1^{-1} X_1^at_1^b\,g_1\,X_1^{a'}t_1^{b'}g_1\bigr)\cdot g_1\ldots g_{k-3 }w' & \text{if $J'<k-2$\,.}
\end{array}\right.
\end{equation}
The comparison of (\ref{eq-proof5}) and (\ref{eq-proof6}) using 
Lemma \ref{lemm-RT3} concludes the verification of (\ref{4}), and in turn the proof of Proposition \ref{prop-RT}.
\end{proof}

\subsection{Markov traces on ${\rm Y}(d,m,n)$}\label{subsec-Markov}

Let $n>0$. We define inductively elements $\tX_1,\tX_2,\ldots,\tX_n$ of ${\rm Y}(d,m,n)$ by 
\[\tX_1:=X_1\ \ \ \text{and}\ \ \ \tX_{i+1}:=g_i^{-1}\tX_ig_i\ \ \text{for $i=1,\ldots,n-1$.}\]
Note that $\tX_i^a =W^{(i)}_{i-1,a,0}$ for any $i\in\{1,\ldots,n\}$ and 
$a\in\Z$. 
Hence, we  have $t_i^b\tX_i^a=W^{(i)}_{i-1,a,b} = \tX_i^a t_i^b$ for any  $b\in\Z$.

We recall that the algebra ${\rm Y}(d,m,k)$ for $k\leq n$ is identified with the subalgebra of ${\rm Y}(d,m,n)$ generated by $g_1,\ldots,g_{k-1},t_1,\ldots,t_k,X_1^{\pm1}$.

\begin{defn}\label{def-Markov}
{\rm Let $z$ and $x_{a,b}$, with $a\in E_m$ and $b\in \{0,\ldots,d-1\}$, be parameters in $\cR_{m}$ such that $x_{0,0}=1$.
A \emph{Markov trace with parameters $z$ and $x_{a,b}$} on the algebra ${\rm Y}(d,m,n)$ is an $\cR_m$-linear map $\tr\,:\,{\rm Y}(d,m,n)\to\cR_m$ satisfying the following conditions:
\begin{eqnarray}
\label{Markov1}&& \tr(YZ)=\tr(ZY)\ \qquad \text{for any $Y,Z\in{\rm Y}(d,m,n)$\,,}\\[0.5em]
\label{Markov2}& & \tr(g_{k-1}\,u)=z\,\tr(u) \ \qquad \text{for any $k\in\{2,\ldots,n\}$ and $u\in{\rm Y}(d,m,k-1)$\,,}\\[0.5em]
\label{Markov3}& & \tr(\tX_k^a t_k^b\,u)=x_{a,b}\,\tr(u)\ \qquad \text{for any $k\in\{1,\ldots,n\}$ and $u\in{\rm Y}(d,m,k-1)$\,,}
\end{eqnarray}
where $a\in E_m$ and $b\in \{0,\ldots,d-1\}$.}
\end{defn}

Let $\{\tr_k\}_{k \in \Z_{>0}}$ be the unique chain of relative traces with parameters $z$ and $x_{a,b}$, with $x_{0,0}=1$, given by Proposition \ref{prop-RT}. We define a linear map $\tau$ from ${\rm Y}(d,m,n)$ to $\cR_m$ by
$$\tau:= \tr_1 \circ \tr_2 \circ \cdots \circ\tr_{n-1} \circ \tr_n\,.$$
We note that, from (\ref{rem-tr}) together with the fact that $x_{0,0}=1$, we have, for any $k\leq n$ and $u\in{\rm Y}(d,m,k)$,
\begin{equation}\label{tau-chaine}
\tau(u)=\tr_1 \circ \tr_2 \circ \cdots \circ\tr_{n-1} \circ \tr_n(u)=\tr_1 \circ \tr_2 \circ \cdots  \circ \tr_k(u)\,.
\end{equation}
Then it follows immediately from (\ref{RT1}) and (\ref{RT2}) that the action of $\tau$ on the elements of the basis (\ref{base}) of ${\rm Y}(d,m,n)$ is given by the following initial condition and recursive formula:
\begin{equation}\label{form-tau}
\tau(1)=1\ \ \ \text{and}\ \ \ \tau\bigl(W^{(k)}_{J,a,b}\,w\bigr)= \left\{\begin{array}{ll}
z\,\tau\bigl(W^{(k-1)}_{J,a,b}\,w\bigr) & \text{if $0 \leq J<k-1$,}\\ & \\
x_{a,b}\,\tau(w) & \text{if $J=k-1$,}
\end{array}\right.
\end{equation}
where $k=1,\ldots,n$, $J\in\{0,\ldots,k-1\}$, $a \in E_m$, $b\in\{0,\ldots,d-1\}$ and $w\in{\rm Y}(d,m,k-1)$.

 \begin{prop}\label{prop-tau}
 The map $\tau$ is the unique Markov trace with parameters $z$ and $x_{a,b}$ on the algebra ${\rm Y}(d,m,n)$.
 \end{prop}
\begin{proof}
First assume that $\tr$ is a Markov trace with parameters $z$ and $x_{a,b}$ on the algebra ${\rm Y}(d,m,n)$. Then $\tr(1)=x_{0,0}=1$. Let $k\in\{1,\ldots,n\}$, $a \in E_m$, $b\in\{0,\ldots,d-1\}$ and $w\in{\rm Y}(d,m,k-1)$. The condition (\ref{Markov3}) is
\[\tr\bigl(W^{(k)}_{k-1,a,b}\,w\bigr)=x_{a,b}\,\tr(w).\]
Let $J\in\{0,\ldots,k-2\}$. Then we have $W^{(k)}_{J,a,b}=W^{(k-1)}_{J,a,b}g_{k-1}$ and so
 \[\tr\bigl(W^{(k)}_{J,a,b}\,w\bigr)=\tr(W^{(k-1)}_{J,a,b}g_{k-1}\,w)=\tr(g_{k-1}\,w\,W^{(k-1)}_{J,a,b})=z\,\tr(w\,W^{(k-1)}_{J,a,b})=z\,\tr(W^{(k-1)}_{J,a,b}\,w),\]
where we have used successively (\ref{Markov1}), (\ref{Markov2}) and (\ref{Markov1}) again. So, if it exists, the Markov trace $\tr$ coincides with the linear map $\tau$ and thus is uniquely defined.
 
It remains to show that the linear map $\tau$ satisfies Conditions (\ref{Markov1})--(\ref{Markov3}). Conditions (\ref{Markov2}) and (\ref{Markov3}) are contained in (\ref{form-tau}), so we only need to prove that $\tau$ is a trace function.

We will proceed by induction on $n$. The algebra ${\rm Y}(d,m,1)$ is commutative so if $Y,Z\in{\rm Y}(d,m,1)$, there is nothing to prove. Now let $1<k\leq n$ and assume that $\tau(YZ)=\tau(ZY)$ for any $Y,Z\in{\rm Y}(d,m,k-1)$. We will show that $\tau(YZ)=\tau(ZY)$ for any $Y,Z\in{\rm Y}(d,m,k)$.

It is enough to take $Y$ to be a generator of the algebra ${\rm Y}(d,m,k)$ and $Z=W^{(k)}_{J,a,b} w$ be an element of the basis (\ref{base}) of ${\rm Y}(d,m,k)$. If $Y\in{\rm Y}(d,m,k-1)$, then by (\ref{2}), (\ref{tau-chaine}) and the induction hypothesis, we have
$$ \tau(Y Z) = \tau(Y \tr_k(Z) ) = \tau(\tr_k(Z) Y) = \tau(\tr_k(Z Y))=\tau(Z Y).$$
So it remains to take $Y\in\{g_{k-1},t_k\}$. By (\ref{4}) if $Y=g_{k-1}$, and by Lemma \ref{lemm-RT1} if $Y=t_k$, we have
$$\tau(Y Z) =\tau(Z Y),$$
which concludes the verification of Condition (\ref{Markov1}) for $\tau$.
\end{proof}

\begin{rem}{\rm
The integer $n$ is absent from the notation $\tr$ for the Markov trace on ${\rm Y}(d,m,n)$ of Definition \ref{def-Markov}. This is justified by the following fact. Denote, just for this remark, the Markov trace on ${\rm Y}(d,m,n)$ by $\tr^{(n)}$. An element $u\in{\rm Y}(d,m,n)$ can be seen, by the chain property (\ref{chainproperty}) of ${\rm Y}(d,m,n)$, as an element of ${\rm Y}(d,m,n')$ for any $n'\geq n$. Then, due to Proposition \ref{prop-tau} and Formula (\ref{tau-chaine}), we have
\[\tr^{(n')}(u)=\tr^{(n)}(u)\ \ \ \ \ \text{for any $n'\geq n$.}\]
Thus the Markov trace $\tr$ of Definition \ref{def-Markov} can actually be interpreted as a Markov trace on the whole chain, on $n$, of algebras 
${\rm Y}(d,m,n)$. \hfill $\triangle$}\end{rem}

\begin{rem}{\rm For $d=1$, the Markov trace $\tr$ on the Ariki--Koike algebra ${\rm H}(m,n)$ was introduced 
\begin{itemize}
\item by Ocneanu/Jones \cite{Jo} for $m=1$, \smallbreak
\item by  Lambropoulou and Geck \cite{La1, Gela} for $m=2$, and \smallbreak 
\item by Lambropoulou \cite{La2} for $m \geq 3$. 
\end{itemize} 
For $m=1$, $\tr$ is the Markov trace on the Yokonuma--Hecke algebra ${\rm Y}_m(d,n)$ defined by Juyumaya \cite{ju2}.
\hfill $\triangle$}
\end{rem}

\begin{rem}{\rm
As we will see in the next section, only Conditions (\ref{Markov1}) and (\ref{Markov2}) are necessary for obtaining invariants for framed knots and links. The additional Condition (\ref{Markov3}) allowed us to describe explicitly the Markov trace, and in turn prove its existence and uniqueness. It is an open question to describe all linear maps from ${\rm Y}(d,m,n)$ to $\cR_m$ satisfying only (\ref{Markov1}) and (\ref{Markov2}), already for the (non-framed) affine case $d=1$ and for the (non-affine) framed case $m=1$.

For example, let us take $m=1$ and $n=2$, that is, we consider the Yokonuma--Hecke algebra ${\rm Y}(d,1,2)$. It is easy to check that all linear maps from ${\rm Y}(d,1,2)$ to $\cR_m$ satisfying (\ref{Markov1}) and (\ref{Markov2}) are given on the basis elements $W^{(2)}_{J,0,b}\,t_1^{b'}$, where $J \in \{0,1\}$ and $b,b' \in \{0,\ldots,d-1\}$, by:
\[\tr(t_1^{b}g_1t_1^{b'})=x_{b+b'\,\text{mod}(d)}\ \ \ \ \ \text{and}\ \ \ \ \ \tr(g_1^{-1}t_1^{b}g_1t_1^{b'})=y_{b,b'},\]
where the parameters $x_a,\, y_{b,b'}\in\cR_m$, with $a,b,b'\in\{0,\ldots,d-1\}$, satisfy 
\[x_a=z\,y_{0,a} \ \ \ \ \ \text{and}\ \ \ \ \  y_{b,b'}=y_{b',b}.\] 
Such a linear map satisfies (\ref{Markov3}) if and only if we have in addition $y_{b,b'}=y_{0,b}\,y_{0,b'}$. 
\hfill $\triangle$}\end{rem}

\section{Invariants for framed knots and links in the solid torus}\label{sec-inv}
As stated in the begining of this paper, the affine and cyclotomic Yokonuma--Hecke algebras can be seen as quotients of the modular framed affine braid group algebra over $\cR_m$. 
In this section, we will see that each framed link in the solid torus can be represented by an element of the modular framed affine braid group. We will then use the Markov trace on ${\rm Y}(d,m,n)$ constructed in the previous section to define invariants for framed knots and links in the solid torus. Our approach will be a generalisation of the approach used by Lambropoulou and Geck \cite{La1, Gela, La2} for (non-framed) knots and links in the solid torus 
 (case $d=1$), as well as the approach used by Juyumaya and Lambropoulou \cite{jula2} for (usual) framed knots and links (case $m=1$).

\subsection{Modular framed affine braids}
Let $n \in \Z_{>0}$. We denote by $B_n^{\mathrm{aff}}$ the affine braid group with a presentation given by: 
\begin{itemize}
\item generators: $\sigma_0, \sigma_1, \ldots, \sigma_{n-1}$, \smallbreak
\item and relations: \begin{equation}\label{affinebraid}
\begin{array}{rclcl}
\sigma_0\sigma_1\sigma_0\sigma_1 & = & \sigma_1\sigma_0\sigma_1\sigma_0 && \\[0.1em]
\sigma_i\sigma_j & = & \sigma_j\sigma_i && \mbox{for all $i,j=0,1,\ldots,n-1$ such that $\vert i-j\vert > 1$,}\\[0.1em]
\sigma_i\sigma_{i+1}\sigma_i & = & \sigma_{i+1}\sigma_i\sigma_{i+1} && \mbox{for  all $i=1,\ldots,n-2$.}\\[0.1em]
\end{array}\end{equation}
\end{itemize}

The usual braid group $B_n$ on $n$ strands is isomorphic to the subgroup of $B_n^{\mathrm{aff}}$ generated by $ \sigma_1, \ldots, \sigma_{n-1}$, and also to the quotient of $B_n^{\mathrm{aff}}$ over the relation $\sigma_0=1$. We keep the notation $\sigma_1, \ldots, \sigma_{n-1}$ for the images of these elements in the quotient, and we denote by $\pi_0$ the surjective homomorphism from $B_n^{\mathrm{aff}}$ to $B_n$ given by
\[\sigma_0\mapsto 1\ \ \ \ \ \text{and}\ \ \ \ \ B_n^{\mathrm{aff}}\ni\sigma_i\mapsto\sigma_i\in B_n\ \ \ \text{for any $i=1,\ldots,n-1$.}\]
Let $\alpha = \sigma_{i_1}^{b_{i_1}} \sigma_{i_2}^{b_{i_2}}\ldots \sigma_{i_r}^{b_{i_r}} \in B_n$, with
$i_1,i_2,\ldots,i_r \in \{1,\ldots,n-1\}$ and $b_{i_1}, b_{i_2},\ldots, b_{i_r} \in \Z$.
We will denote by $\epsilon(\alpha)$ the sum of the exponents of $\alpha$, that is,
\begin{equation}\label{sumofexp}
\epsilon(\alpha):= b_{i_1} + b_{i_2} + \cdots +b_{i_r}\,.
\end{equation}
We extend the definition of the map $\epsilon$ to the affine braid group $B_n^{\mathrm{aff}}$ as follows:
\begin{equation}\label{sumofexp2}
\epsilon(\beta):= \epsilon\bigl(\pi_0(\beta)\bigr)\ \ \ \ \ \text{for any $\beta\in B_n^{\mathrm{aff}}$\,.}
\end{equation}
The defining relations (\ref{affinebraid}) are homogeneous in the generators $\sigma_0, \sigma_1, \ldots, \sigma_{n-1}$, so 
$\epsilon(\beta)$ does not depend on the chosen word for $\beta$ in terms of the generators.
There is also a diagrammatic interpretation of $\beta$ in terms of $n$ braid strands (as in the classical braid group case) plus one extra fixed strand that extends to infinity (the ``pole''). The generators $\sigma_1,\ldots,\sigma_{n-1}$ appearing in $\beta$ correspond to the classical braid group movements, while every appearance of the generator $\sigma_0$ corresponds to a circling of the first strand around the pole (see, for example, \cite{La1, Gela, La2}). 
Moreover, if $\beta = \sigma_{i_1}^{b_{i_1}} \sigma_{i_2}^{b_{i_2}}\ldots \sigma_{i_r}^{b_{i_r}}$, with $i_1,i_2,\ldots,i_r \in \{0,1,\ldots,n-1\}$ and $b_{i_1}, b_{i_2},\ldots, b_{i_r} \in \Z$, we set
\begin{equation}\label{epsilon'}
\epsilon'(\beta):=b_{i_1} + b_{i_2} + \cdots +b_{i_r}.
\end{equation}

Further, we denote by $\pi$ be the natural surjective homomorphism from $B_n^{\mathrm{aff}}$ to the symmetric group $\mathfrak{S}_n$ on $n$ letters, given by 
\[\sigma_0\mapsto1\ \ \ \ \ \text{and}\ \ \ \ \ \sigma_i \mapsto s_i\ \ \ \text{for any $i=1,\ldots,n-1$,}\]
where $s_i$ is the transposition $(i,i+1)$. We set  $\overline{\beta}:=\pi(\beta)$ for any $\beta\in B_n^{\mathrm{aff}}$.

Let now $d \in \Z_{> 0}$ and consider the modular framed (or  $(\Z/d\Z)$-framed) affine braid group
$(\Z/d\Z)\wr B_n^{\mathrm{aff}}$. The group $(\Z/d\Z)\wr B_n^{\mathrm{aff}}$ has a presentation given by 
\begin{itemize}
\item generators: $\sigma_0, \sigma_1, \ldots, \sigma_{n-1}, t_1, t_2, \ldots, t_n$, \smallbreak
\item and relations (\ref{affinebraid}) together with:
\begin{equation}\label{framedaffinebraid}
\begin{array}{rclcl}
t_j^d   & =  &  1 && \mbox{for all $j=1,\ldots,n$,}\\[0.1em]
t_it_j & =  & t_jt_i &&  \mbox{for all $i,j=1,\ldots,n$,}\\[0.1em]
t_j\sigma_0 & = & \sigma_0t_j && \mbox{for all $j=1,\ldots,n$,}\\[0.1em]
t_j\sigma_i & = & \sigma_i t_{s_i(j)} && \mbox{for all  $j=1,\ldots,n$ and $i=1,\ldots,n-1$,}\\[0.1em]
\end{array}
\end{equation}
where   $s_i$ is the transposition $(i,i+1)$.
\end{itemize}
The group $(\Z/d\Z)\wr B_n^{\mathrm{aff}}$ is the semi-direct product of $(\Z/d\Z)^n$ with the group $B_n^{\mathrm{aff}}$, with the action of $B_n^{\mathrm{aff}}$ on $(\Z/d\Z)^n$ given by the composition of $\pi$ and the natural permutation action of $\mathfrak{S}_n$ on $(\Z/d\Z)^n$. We have that
\begin{equation}\label{conjug}
t_1^{a_1}\ldots t_n^{a_n}\,\beta=\beta\,t_1^{a_{\overline{\beta}(1)}}\ldots t_{n}^{a_{\overline{\beta}(n)}}\ \ \ \ \text{for any $a_1,\ldots,a_n\in\{0,1,\ldots,d-1\}$ and $\beta\in B_n^{\mathrm{aff}}$, }
\end{equation}
and any $\alpha\in(\Z/d\Z)\wr B_n^{\mathrm{aff}}$ can be written uniquely in the form
\[\alpha = t_1^{a_1}\ldots t_n^{a_n} \beta\ \ \ \ \text{where $a_1,\ldots,a_n\in\{0,1,\ldots,d-1\}$ and $\beta\in B_n^{\mathrm{aff}}$.}\]
We will refer to $t_1^{a_1}\ldots t_n^{a_n}$ as the ``framing part'' of $\alpha$ and  to $\beta$ as the  ``braiding part'' of $\alpha$. We set
\begin{equation}\label{exponent}
\epsilon(\alpha):=\epsilon(\beta),
\end{equation}
where $\epsilon(\beta)$ is defined in (\ref{sumofexp2}), and
\begin{equation}\label{exponent'}
\epsilon'(\alpha):=\epsilon'(\beta),
\end{equation}
where $\epsilon'(\beta)$ is defined in (\ref{epsilon'}).  
There is also a diagrammatic interpretation of $\alpha$ as follows: The diagram of $\alpha$ is  the diagram of  $\beta$ with the integer $a_i$ attached to the $i$-th strand, for $i=1,\ldots,n$. We will call $a_i$ 
the \emph{framing} of the $i$-th braid strand. 
Then, by construction, multiplication in $(\Z/d\Z)\wr B_n^{\mathrm{aff}}$ corresponds to concatenation of framed braid diagrams, that is, usual concatenation of braid diagrams together with addition (in $\Z/d\Z$) of the framings on each strand (see, for example, \cite{jula1}
for the non-affine framed situation).

\begin{rem}{\rm
We use the same notation $t_j$ for the generators of $(\Z/d\Z)^n$ inside $(\Z/d\Z)\wr B_n^{\mathrm{aff}}$ and for generators of ${\rm Y}(d,m,n)$; 
this should not lead to any confusion,  since the subalgebra of ${\rm Y}(d,m,n)$ generated by $t_1,\ldots,t_n$ is isomorphic to the group algebra of $(\Z/d\Z)^n$.
\hfill $\triangle$}\end{rem}

\begin{rem}{\rm
Note that we can interpret any affine braid as a  $(\Z/d\Z)$-framed affine braid with all framings equal to $0$.\hfill $\triangle$}
\end{rem}

\subsection{Modular framed solid torus links} 
From now on, we will simply say \emph{solid torus links} for the links in the solid torus. A \emph{framed solid torus link} is a solid torus link where each connected component 
has an integer attached to it. We will call this integer the \emph{framing} of the connected component.  A \emph{modular framed (or  $(\Z/d\Z)$-framed) solid torus link} is a framed solid torus link where all framings belong to 
$\Z/d\Z$.

The closure of an affine braid can be interpreted as a solid torus link \cite{La1}. 
For a $(\Z/d\Z)$-framed affine braid $\alpha\in(\Z/d\Z)\wr B_n^{\mathrm{aff}}$, we will denote by $\widehat{\alpha}$ its closure, which is the $(\Z/d\Z)$-framed solid torus link defined as follows:
We consider the closure $\widehat{\beta}$ of the braiding part $\beta$ of $\alpha$. Then $\widehat{\alpha}$ is obtained 
by attaching to each connected component of $\widehat{\beta}$ the sum 
(in $\Z/d\Z$)
of the framings of the strands forming this component after closure. Again, we can interpret a solid torus link as a  $(\Z/d\Z)$-framed solid torus link with all framings equal to $0$.

In \cite[Theorem 1]{La1}, the analogue of Alexander's theorem is proved for solid torus links; namely, any solid torus link can be obtained as the closure of an affine braid. Moreover, it is obvious that by adding a suitable framing on the affine braid, one can obtain any possible framing in the solid torus link. So the analogue of Alexander's theorem is also true in our setting.
\begin{thm}\label{Sofia1}
Any $(\Z/d\Z)$-framed solid torus link can be obtained as the closure of a $(\Z/d\Z)$-framed affine braid.
\end{thm}

\begin{rem} {\rm 
For any $n,n'\in\Z>0$ such that $n'>n$, we denote by $\iota_{n,n'}$ the group homomorphism from $(\Z/d\Z)\wr B_n^{\mathrm{aff}}$ to $(\Z/d\Z)\wr B_{n'}^{\mathrm{aff}}$ given by
$$
\iota_{n,n'}\ :\ \ (\Z/d\Z)\wr B_n^{\mathrm{aff}}\ni x\mapsto x\in (\Z/d\Z)\wr B_{n'}^{\mathrm{aff}}\ \ \ \ \ \ \ \ \text{for any $x\in\{\sigma_0, \sigma_1, \ldots, \sigma_{n-1}, t_1, t_2, \ldots, t_n\}$\,.}
$$
The homomorphism $\iota_{n,n'}$ is in fact an isomorphism between $(\Z/d\Z)\wr B_n^{\mathrm{aff}}$ and the subgroup of $(\Z/d\Z)\wr B_{n'}^{\mathrm{aff}}$ generated by $\sigma_0, \sigma_1, \ldots, \sigma_{n-1}, t_1, t_2, \ldots, t_n$. 
This allows us to consider the chain, on $n$, of  groups 
\begin{equation}\label{secondchain}
\{1\} \subset (\Z/d\Z)\wr B_1^{\mathrm{aff}} \subset\cdots\subset (\Z/d\Z)\wr B_{n-1}^{\mathrm{aff}}  \subset (\Z/d\Z)\wr B_n^{\mathrm{aff}} \subset\cdots.
\end{equation}
However, the closures  of $\alpha\in(\Z/d\Z)\wr B_n^{\mathrm{aff}}$ and $\iota_{n,n'}(\alpha)$ are different $(\Z/d\Z)$-framed solid torus links. This is why, in what follows, 
whenever we say that $\alpha\in(\Z/d\Z)\wr B_n^{\mathrm{aff}}$, we mean that $\alpha$ is expressed diagrammatically exactly on $n$ braid strands (plus the pole). 
We thus consider the whole union $\bigcup_{n \geq 1}(\Z/d\Z)\wr B_n^{\mathrm{aff}}$,
and we use the chain (\ref{secondchain}) to define multiplication between its elements.
\hfill $\triangle$}
\end{rem}

\begin{defn}{\rm Two $(\Z/d\Z)$-framed affine braids $\alpha, \alpha'\in\bigcup_{n \geq 1}(\Z/d\Z)\wr B_n^{\mathrm{aff}}$ are \emph{equivalent} if and only if there exists a finite sequence 
of $(\Z/d\Z)$-framed affine braids $\alpha_0,\alpha_1,\ldots,\alpha_r \in \bigcup_{n \geq 1}(\Z/d\Z)\wr B_n^{\mathrm{aff}}$ with $\alpha = \alpha_0$ and $\alpha' = \alpha_r$ such that, for all $i=1,\ldots,r$, one of the following holds:
\begin{enumerate}[(i)]
\item there exist $n \geq 1$ and $\gamma_i \in (\Z/d\Z)\wr B_n^{\mathrm{aff}}$ such that $\alpha_{i-1}, \alpha_i \in (\Z/d\Z)\wr B_n^{\mathrm{aff}}$ and $\alpha_{i} = \gamma_i\alpha_{i-1}\gamma_i^{-1}$\,;
\vskip .1cm
\item there exists  $n \geq 1$ such that $\alpha_{i-1} \in (\Z/d\Z)\wr B_n^{\mathrm{aff}}$, $\alpha_{i} \in (\Z/d\Z)\wr B_{n+1}^{\mathrm{aff}}$ and $\alpha_i =\alpha_{i-1}\,\sigma_n^{\pm 1}$\,;
\vskip .1cm
\item there exists  $n \geq 1$ such that $\alpha_{i-1} \in (\Z/d\Z)\wr B_{n+1}^{\mathrm{aff}}$, $\alpha_{i} \in  (\Z/d\Z)\wr B_n^{\mathrm{aff}}$ and $\alpha_{i-1} = \alpha_{i}\,\sigma_n^{\pm 1}$\,.
\end{enumerate}
We will write $\alpha \sim \alpha'$ for two equivalent $(\Z/d\Z)$-framed affine braids.}
\end{defn}

Two $(\Z/d\Z)$-framed solid torus links are isotopic if the underlying solid torus links are isotopic and the framing is conserved. The following theorem has been shown in \cite[Theorem 3]{La1} if we assume that $\alpha, \alpha'\in\bigcup_{n \geq 1}B_n^{\mathrm{aff}}$ (case $d=1$). It has been also proved in \cite[Lemma 1]{KoSm} for (non-affine) $(\Z/d\Z)$-framed braids (case $m=1$). Using \cite[Theorem 3]{La1}, the proof of \cite[Lemma 1]{KoSm} generalises straightforwardly to our setting. We give all details here for completeness.

\begin{thm}\label{Sofia2}
Let $\alpha, \alpha' \in \bigcup_{n \geq 1}(\Z/d\Z)\wr B_n^{\mathrm{aff}}$. The $(\Z/d\Z)$-framed solid torus links $\widehat{\alpha}$ and $\widehat{\alpha'}$ are isotopic if and only if $\alpha \sim \alpha'$.
\end{thm}
\begin{proof}
Let $\alpha = t_1^{a_1}\ldots t_n^{a_n} \beta\in(\Z/d\Z)\wr B_n^{\mathrm{aff}}$ and $\alpha' = t_1^{a_1'}\ldots t_m^{a_m'} \beta'\in(\Z/d\Z)\wr B_m^{\mathrm{aff}}$, where $\beta\in  B_n^{\mathrm{aff}}$, $\beta'\in  B_m^{\mathrm{aff}}$ and $a_1,\ldots,a_n,a'_1,\ldots,a'_m\in\{0,\ldots,d-1\}$.

Assume that $\alpha \sim \alpha'$. Then we also have $\beta \sim \beta'$ because conjugating in $(\Z/d\Z)\wr B_n^{\mathrm{aff}}$ by elements $t_j$ does not change the braiding part. Using the known result for affine braids in $B_n^{\mathrm{aff}}$, we deduce that $ \widehat{\beta}$ and $\widehat{\beta'}$ are isotopic solid torus links.
We then observe that moves (i)--(iii) do not change the framing of each link component after the framed braids are closed. It is obvious for moves (ii) and (iii). For the move (i), let $\gamma=t_1^{c_1}\ldots t_n^{c_n} \delta$ with $\delta\in B_n^{\mathrm{aff}}$ and write
\[\gamma\alpha\gamma^{-1}=
t_1^{c_1}\ldots t_n^{c_n} \delta\cdot t_1^{a'_1}\ldots t_n^{a'_n}\beta\cdot\delta^{-1}t_1^{-c_1}\ldots t_n^{-c_n}\,.\]
Thus after the closure, the framings coming from $\gamma$ and $\gamma^{-1}$ cancel each other out. We conclude that $\widehat{\alpha}$ and $\widehat{\alpha'}$ are isotopic $(\Z/d\Z)$-framed solid torus links.

 Now assume that  $ \widehat{\alpha}$ and $\widehat{\alpha'}$ are isotopic $(\Z/d\Z)$-framed solid torus links. Then $ \widehat{\beta}$ and $\widehat{\beta'}$ are isotopic solid torus links. Using again the known result for affine braids, we deduce that $\beta \sim \beta'$ as affine braids. Without loss of generality, we may assume that $\beta'$ is obtained from $\beta$ with just one of the moves (i) or (ii), where $\gamma_i$ in move (i) is restricted to be in $B_n^{\mathrm{aff}}$. 

If $m=n+1$ and $\beta'=\beta\,\sigma_n^{\pm1}$, then $\alpha'=t_1^{a_1'}\ldots t_{n}^{a_{n}'}t_{n+1}^{a_{n+1}'} \beta\,\sigma_n^{\pm1}$. Note that $\beta$ commutes with $t_{n+1}$ and therefore, $\alpha'=t_1^{a_1'}\ldots t_{n}^{a_{n}'}\beta\,\sigma_n^{\pm1}t_{n}^{a_{n+1}'}$. By first conjugating by $t_{n}^{a_{n+1}'}$ and then applying move (iii), we obtain that
\[\alpha'\,\sim\, t_1^{a_1'}\ldots t_{n}^{a_{n}'+a_{n+1}'}\,\beta\,\sigma_n^{\pm1}\,\sim\, t_1^{a_1'}\ldots t_{n}^{a_{n}'+a_{n+1}'}\,\beta.\]
If $m=n$ and $\beta'$ = $\gamma\beta \gamma^{-1}$ for some $\gamma \in B_n^{\mathrm{aff}}$ then, by (\ref{conjug}) and move (i), we have
\[\alpha'=\gamma\cdot t_1^{a_{\overline{\gamma}(1)}'}\ldots t_{n}^{a_{\overline{\gamma}(n)}'}\,\beta\cdot \gamma^{-1}\,\sim\, t_1^{a_{\overline{\gamma}(1)}'}\ldots t_{n}^{a_{\overline{\gamma}(n)}'}\,\beta.\]
In both cases, we have an element $\alpha''\in(\Z/d\Z)\wr B_n^{\mathrm{aff}}$ such that $\alpha' \sim \alpha''$, whence $\widehat{\alpha'}$ and $\widehat{\alpha''}$ are isotopic, and the braiding part of $\alpha''$ is the same as the braiding part of $\alpha$. So it suffices to show that if two $(\Z/d\Z)$-framed affine braids with the same braiding part have isotopic closure, then they are equivalent. We will show that in fact they are conjugate by an element with trivial braiding part.

We can assume now that $m=n$ and $\beta'=\beta$. Let $\overline{\beta}=\tau_1\tau_2\ldots\tau_k$ be the decomposition of $\overline{\beta}$ into disjoint cycles with $C_1,\ldots,C_k\subset\{1,\ldots,n\}$ the supports of $\tau_1,\ldots,\tau_k$. Since $\widehat{\alpha}$ and $\widehat{\alpha'}$ are isotopic we have 
\begin{equation}\label{samesum}
\sum_{i \in C_j}a_i = \sum_{i \in C_j}a_i' \quad \text{for all $j=1,\ldots,k$}.
\end{equation}
Conjugating  $ \alpha = t_1^{a_1}\ldots t_n^{a_n} \beta$ by an element $t_1^{r_1}\ldots t_n^{r_n}$, we obtain
$$t_1^{a_1+r_1-r_{\overline{\beta}^{-1}(1)}}t_2^{a_2+r_2-r_{\overline{\beta}^{-1}(2)}}\ldots t_n^{a_n+r_n-r_{\overline{\beta}^{-1}(n)}}\,\beta.$$
We want to show that, for some values of $r_1,\ldots,r_n$, this is equal to $\alpha'$. Thus we need a solution to the following system of equations for $r_1,\ldots,r_n$:
\begin{equation}\label{system}
r_i - r_{\tau_j^{-1}(i)} = a_i' -a_i \quad\  \text{for } i \in \tau_j,\ \  j=1,\ldots,k.
\end{equation}

Let $j \in \{1,\ldots,k\}$. Suppose $\tau_j$ is a cycle of order $m_j$ and write $\tau_j=(p_{j,1},\ldots,p_{j,m_j})$, that is, choose an element $p_{j,1}\in C_j$ and define $p_{j,l+1}:=\tau_j(p_{j,l})$ for $l=1,\ldots,m_j-1$. Set
\begin{equation}\label{sol}
r_{p_{j,l}} := \sum_{\mu=1}^{l} (a_{p_{j,\mu}}' -  a_{p_{j,\mu}})\ \ \ \ \text{for $l=1,\ldots,m_j$.}
\end{equation}
Then, for any $l\in\{1,\ldots,m_j\}$, we have
\[r_{p_{j,l}} - r_{\tau_j^{-1}(p_{j,l})} =\left\{\begin{array}{ll} 
r_{p_{j,l}} - r_{p_{j,l-1}}=a_{p_{j,l}}' - a_{p_{j,l}} & \text{if $l>1$,}\\ & \\
r_{p_{j,1}} - r_{p_{j,m_j}}=a_{p_{j,1}}' - a_{p_{j,1}} - r_{p_{j,m_j}} & \text{if $l=1$.}
\end{array}\right.\]
Since $r_{p_{j,m_j}}=\sum_{i \in C_j}(a_i'-a_i)=0$ by (\ref{samesum}), we conclude that (\ref{sol}) is a solution to the system of equations (\ref{system}).
\end{proof}

\subsection{Invariants for $(\Z/d\Z)$-framed solid torus links}

Let us consider the group algebra $\cR_m [(\Z/d\Z)\wr B_n^{\mathrm{aff}}]$ of the modular framed affine braid group over $\cR_m$.
As we mentioned earlier, the affine and cyclotomic Yokonuma--Hecke algebras are quotients of this algebra.
Thus, there is a natural surjective algebra homomorphim $\delta_n : \cR_m [(\Z/d\Z)\wr B_n^{\mathrm{aff}}] \rightarrow {\rm Y}(d,m,n)$ given by
\begin{equation}\label{delta}
\sigma_0\mapsto X_1\,,\ \ \ \ \ \sigma_i\mapsto g_i\,,\ i=1,\ldots,n-1,\ \ \ \ \ \text{and}\ \ \ \ \ t_j\mapsto t_j\,,\ j=1,\ldots,n.
\end{equation}

Let us also consider the Markov trace $\tr$ of Definition \ref{def-Markov} 
(given by Proposition \ref{prop-tau}) 
with parameters $z$ and $x_{a,b}$, with $a\in E_m$ and $b\in \{0,\ldots,d-1\}$, such that $x_{0,0}=1$. We also assume that $z\neq 0$. From now on, we set 
$$E := \tr(e_i)=\frac{1}{d} \sum_{s=0}^{d-1} x_{0,-s}x_{0,s} \quad \text{for all }\, i \in \Z_{>0}\,,$$
where the last equality follows easily from (\ref{Markov3}).

We extend the ring of definition by setting $\displaystyle\widetilde{\cR}_m:=\cR_m[z^{-1},\sqrt{\frac{z - (q-q^{-1})E}{z}}^{\pm1}]$. We define
$$\omega:= \frac{z - (q-q^{-1})E}{z}\in\widetilde{\cR}_m\ \ \ \ \ \text{and}\ \ \ \ \ D:= \frac{1}{z \sqrt{\omega}} \in\widetilde{\cR}_m.$$

For any $\alpha \in (\Z/d\Z) \wr B_n^{\mathrm{aff}}$, we set
\begin{equation}
\Gamma_m( {\alpha} ) := D^{n-1} \sqrt{\omega}^{\epsilon(\alpha)} (\tr \circ \delta_n) (\alpha)\in\widetilde{\cR}_m ,
\end{equation}
where $\epsilon(\alpha)$ is defined in (\ref{exponent}) and $\delta_n$ is the natural surjection from 
$\cR_m [(\Z/d\Z)\wr B_n^{\mathrm{aff}}]$ to  ${\rm Y}(d,m,n)$
 defined in (\ref{delta}). 

\begin{prop}\label{stableunderer} Let $m\in\Z_{>0}\cup\{\infty\}$.
Assume that, for any $n\geq 1$,
\begin{equation}\label{E-cond}
\tr(ue_n) = \tr(u)\tr(e_n) \quad \text{ for all } u \in {\rm Y}(d,m,n).
\end{equation}
Let $\alpha,\,\alpha' \in\bigcup_{n \geq 1}(\Z/d\Z)\wr B_n^{\mathrm{aff}}$. If $\alpha \sim \alpha'$, then $\Gamma_m( \alpha)=\Gamma_m(\alpha')$.
\end{prop}

\begin{proof}
It is enough to show that, for all $\alpha, \beta \in (\Z/d\Z) \wr B_n^{\mathrm{aff}}$, we have 
$$ \Gamma_m(\alpha \beta ) = \Gamma_m(\beta \alpha)\ \ \ \ \ \text{and}\ \ \ \ \ \Gamma_m(\alpha) =  \Gamma_m\bigl(\alpha\,\sigma_n^{\pm 1}\bigr).$$
The first equality follows immediately from (\ref{Markov1}).

Now note that, as we consider the chain (on $n$) of algebras ${\rm Y}(d,m,n)$, we have 
\[\delta_{n+1}\bigl(\alpha\,\sigma_n^{\epsilon}\bigr)=\delta_n(\alpha)\,g_n^{\epsilon}\ \quad \text{ for } \epsilon=\pm 1.\] 
Then, due to (\ref{Markov2}), we have
$$ \Gamma_m\bigl(\alpha\,\sigma_n\bigr) = D^{n} \sqrt{\omega}^{\epsilon(\alpha)+1} \tr\bigl( \delta_n (\alpha) g_n\bigr) = z\, D^{n} \sqrt{\omega}^{\epsilon(\alpha)+1} \tr\bigl( \delta_n (\alpha)\bigr)
=z\, D\,\sqrt{\omega}\, \Gamma_m(\alpha).$$
Since $D=1/z\sqrt{\omega} $, we get $ \Gamma_m\bigl(\alpha\,\sigma_n\bigr)= \Gamma_m(\alpha)$. Further, we have
$$\Gamma_m\bigl(\alpha\,\sigma_n^{-1}\bigr) = D^{n} \sqrt{\omega}^{\epsilon(\alpha)-1} \tr\bigl( \delta_n (\alpha) g_n^{-1}\bigr) =  D^{n} \sqrt{\omega}^{\epsilon(\alpha)-1} \Bigl(\tr\bigl( \delta_n (\alpha)g_n\bigr) -(q-q^{-1})\tr\bigl( \delta_n(\alpha)e_n\bigr)\Bigr).$$
Using  (\ref{Markov2}) and the assumption (\ref{E-cond}), we deduce that
$$\Gamma_m\bigl(\alpha\,\sigma_n^{-1}\bigr) = \bigl(z -(q-q^{-1})E\bigr)\,D^{n} \sqrt{\omega}^{\epsilon(\alpha)-1}\tr\bigl( \delta_n (\alpha)\bigr) = \frac{z - (q-q^{-1})E}{z} \,\omega^{-1}\,\Gamma_m\bigl(\alpha\,\sigma_n\bigr)\,.$$
Since $\omega=(z - (q-q^{-1})E)/{z}$, we get  $\Gamma_m\bigl(\alpha\,\sigma_n^{-1}\bigr)=\Gamma_m\bigl(\alpha\,\sigma_n\bigr)\,.$
\end{proof}

Let $\mathcal{L}_d^{\rm tor}$ denote the set of  $(\Z/d\Z)$-framed  solid torus links. Following Theorem \ref{Sofia1}, 
$$\mathcal{L}_d^{\rm tor} = \bigcup_{n \geq 1} \left\{ \widehat{\alpha}\,|\, \alpha \in (\Z/d\Z)\wr B_n^{\mathrm{aff}}\right\}.$$
Combining Proposition \ref{stableunderer} with Theorem \ref{Sofia2} yields the following result, which is the objective of Section \ref{sec-inv}.

\begin{thm}\label{invariant} Let $m\in\Z_{>0}\cup\{\infty\}$.
Assume that, for any $n\geq 1$, $(\ref{E-cond})$ holds.
Then the map
$$\begin{array}{cccc}
\widehat{\Gamma}_m : & \mathcal{L}_d^{\rm tor} & \rightarrow & \widetilde{\cR}_m\\
& \widehat{\alpha} &  \mapsto & \Gamma_m( \alpha )
\end{array}$$
is an isotopy invariant of oriented $(\Z/d\Z)$-framed solid torus links, that is, if $\widehat{\alpha}=\widehat{\alpha'}$, for some $\alpha, \alpha' \in  \bigcup_{n \geq 1}(\Z/d\Z)\wr B_n^{\mathrm{aff}}$, then $\widehat{\Gamma}_m(\widehat{\alpha})=\widehat{\Gamma}_m(\widehat{\alpha'})$.
\end{thm}

 Condition (\ref{E-cond}) will be from now on referred to as the \emph{affine E-condition}. Note that it is a sufficient condition for  $\widehat{\Gamma}_m$ to be an isotopy invariant, but we do not know whether it is necessary.
 
\begin{rems} \label{rem-invariant}{\rm 
{\bf (a)} The affine E-condition is the analogue of the E-condition imposed by Juyumaya and Lambropoulou \cite{jula2}  on classical Yokonuma--Hecke algebras in order to define invariants for (non-affine) framed knots and links. In fact, if we  restrict $\widehat{\Gamma}_m$ to the (non-affine) framed links (case $m=1$), the invariants we constructed in Theorem \ref{invariant} are the same as the ones constructed in \cite{jula2}.\\
{\bf (b)} If, for any $\alpha \in (\Z/d\Z) \wr B_n^{\mathrm{aff}}$, we set
\begin{equation}
\Gamma_m'( {\alpha} ) := D^{n-1} \sqrt{\omega}^{\epsilon'(\alpha)} (\tr \circ \delta_n) (\alpha)\in\widetilde{\cR}_m ,
\end{equation}
where $\epsilon'(\alpha)$ is defined in (\ref{exponent'}), then we can repeat the proof of Proposition \ref{stableunderer} to obtain that $\Gamma_m'$ is stable on the equivalence classes of  $\bigcup_{n \geq 1}(\Z/d\Z)\wr B_n^{\mathrm{aff}}$. Thus, similarly to Theorem  \ref{invariant}, we deduce that the map
$$\begin{array}{cccc}
\widehat{\Gamma}_m' : & \mathcal{L}_d^{\rm tor} & \rightarrow & \widetilde{\cR}_m\\
& \widehat{\alpha} &  \mapsto & \Gamma_m'( \alpha )
\end{array}$$
is also an isotopy invariant of oriented $(\Z/d\Z)$-framed solid torus links. \hfill $\triangle$}
\end{rems}

If now we restrict $\widehat{\Gamma}_m$ to the solid torus links with all framings equal to $0$, then $\widehat{\Gamma}_m$ becomes an invariant of (non-framed) links in the solid torus.
For $m=1$, it was recently shown in \cite{CJKL} that the Juyumaya--Lambropoulou invariants for classical links coming from the Yokonuma--Hecke algebras \cite{jula2,jula3} are not equivalent to the HOMFLYPT polynomial, which is the invariant coming from the usual Hecke algebras (case $m=d=1$). More specifically, there are pairs of classical links that are distinguished by the Juyumaya--Lambropoulou invariants, and not by the HOMFLYPT polynomial. Since the set of classical links can be embedded in the set of classical links in the solid torus, we obtain that the invariants  $\widehat{\Gamma}_m$ restricted to $\mathcal{L}_1^{\rm tor}$ distinguish pairs of links that are not distinguished by the HOMFLYPT-type invariants  
obtained from the affine and cyclotomic Hecke algebras (case $d=1$) in \cite{La1, Gela, La2}. Namely, if $\widehat{\Delta}_m$ denotes the invariant of classical links in the solid torus obtained as the restriction of $\widehat{\Gamma}_m$ to $\mathcal{L}_1^{\rm tor}$, we have the following:

\begin{prop}\label{delta_inv}
Let $m\in\Z_{>0}\cup\{\infty\}$. The invariants $\widehat{\Delta}_m$ are not generically equivalent to the HOMFLYPT-type invariants of links in the solid torus.
\end{prop}

\begin{rem}{\rm
The term ``generically'' in Proposition \ref{delta_inv} is used in order to exclude some extreme cases where the invariants $\widehat{\Delta}_m$ coincide with 
the HOMFLYPT-type invariants of links in the solid torus, such as when $d=1$ or when, following the analogous result in the non-affine case \cite{chla}, $E=1$ or $q=\pm 1$. 
 \hfill $\triangle$}
\end{rem}

\begin{rem}{\rm In \cite{jula4}, Juyumaya and Lambropoulou used the Yokonuma--Hecke algebras to construct invariants for singular knots and links. We expect that 
their results can be generalised to the case of ${\rm Y}(d,m,n)$ for $m \neq 1$ in order to obtain invariants for singular knots and links in the solid torus. \hfill $\triangle$}
\end{rem}

\subsection{The affine E-system} 

In this subsection, we will study further the affine E-condition and show that it imposes some restrictions on the values of the parameters $x_{a,b}$ of the Markov trace $\tr$.

\begin{prop} The affine E-condition holds, that is, we have
$$\tr(ue_n) = \tr(u)\tr(e_n) \quad\ \text{for any $n\geq 1$ and all $u \in {\rm Y}(d,m,n)$},$$
if and only if
\begin{equation}\label{E-syst}
 \frac{1}{d} \sum_{s=0}^{d-1}x_{0,-s}x_{a,b+s}  = x_{a,b}\,E  \quad \text{ for all } a \in E_m,\ \ b \in \{0,\ldots,d-1\}.
\end{equation}
\end{prop}

\begin{proof} Let us take first $n=1$.
We have, using (\ref{Markov3}), 
$$ \tr(X_1^at_1^be_1) = \frac{1}{d} \sum_{s=0}^{d-1} x_{0,-s} \tr(X_1^at_1^bt_1^{s})=  \frac{1}{d} \sum_{s=0}^{d-1}x_{0,-s}x_{a,b+s}.$$
As $\tr(X_1^at_1^b)=x_{a,b}$, we conclude that if the affine E-condition holds, then (\ref{E-syst}) must hold. 

Now assume that (\ref{E-syst}) is true. We will prove that
$$\tr(ue_n) = \tr(u)\tr(e_n) = \tr(u)E \quad \text{ for all } u \in {\rm Y}(d,m,n)$$
by induction on $n$, and by taking $u$ to be an arbitrary element of the basis (\ref{base}) of $ {\rm Y}(d,m,n)$. We have already proved it for $n=1$.

Let us take $n>1$, and let $u = W^{(n)}_{J,a,b}\,w$, where $J \leq n-1$, $a \in E_m$, $b \in \{0,\ldots,d-1\}$ and $w \in {\rm Y}(d,m,n-1)$. 
First assume that $J<n-1$. Then 
$$\tr(u)= z\,\tr( W^{(n-1)}_{J,a,b}\,w)$$ and 
\[\begin{array}{ll}
\tr(ue_n)  & = \displaystyle\frac{1}{d}\sum_{s=0}^{d-1}x_{0,-s}\tr(W^{(n)}_{J,a,b}t_n^s\,w) = \frac{1}{d}\sum_{s=0}^{d-1}x_{0,-s}\tr(W^{(n)}_{J,a,b+s}\,w)= \frac{z}{d}\sum_{s=0}^{d-1}x_{0,-s}\tr(W^{(n-1)}_{J,a,b+s}\,w)\\[1.2em]
& =\displaystyle\frac{z}{d}\sum_{s=0}^{d-1}x_{0,-s}\tr(W^{(n-1)}_{J,a,b}t_{n-1}^s\,w)= \frac{z}{d}\sum_{s=0}^{d-1}\tr(t_n^{-s}W^{(n-1)}_{J,a,b}t_{n-1}^s\,w)= z\, \tr(W^{(n-1)}_{J,a,b}e_{n-1}\,w)\\[1.2em]
& = z\, \tr(w\,W^{(n-1)}_{J,a,b}e_{n-1}),
\end{array}\]
where, besides the properties of the Markov trace (\ref{Markov1})--(\ref{Markov3}), we used that  $W^{(k)}_{J,a,b}t_k=W^{(k)}_{J,a,b+1}$ for any $k\geq1$ and $t_nW^{(n-1)}_{J,a,b}=W^{(n-1)}_{J,a,b}t_n$. Using the induction hypopthesis, we conclude that
$$\tr(ue_n)=z\, \tr(w\,W^{(n-1)}_{J,a,b})\,E=z\, \tr(W^{(n-1)}_{J,a,b}\,w)\,E=\tr(u)\,E.$$

Now assume that $J=n-1$. Then
$$\tr(u) = x_{a,b} \,\tr(w)$$
and, with a similar calculation to the one above,
$$\tr(ue_n) =  \frac{1}{d}\sum_{s=0}^{d-1}x_{0,-s}\tr(W^{(n)}_{J,a,b}t_n^s\,w)=\frac{1}{d}\sum_{s=0}^{d-1}x_{0,-s}\tr(W^{(n)}_{J,a,b+s}\,w)= \frac{1}{d}\sum_{s=0}^{d-1}x_{0,-s}x_{a,b+s}\tr(w).$$
Finally, we use the assumption (\ref{E-syst}) to conclude that $\tr(ue_n)= x_{a,b}\,E\,w=\tr(u)\,E\,$.
\end{proof}

We have just proved that the affine E-condition holds if and only if the parameters $x_{a,b}$ of the Markov trace $\tr$ are solutions of the system of equations (\ref{E-syst}). We will call this system of equations \emph{the affine E-system}. In the next subsection we will classify its solutions.

\begin{rem}{\rm Fix $m\in\Z_{>0}\cup\{\infty\}$.
Then for each solution of the affine E-system, we obtain a different isotopy invariant $\widehat{\Gamma}_m$.
It is an open question whether these isotopy invariants are equivalent.
\hfill $\triangle$}
\end{rem}

\subsection{Solutions of the affine E-system}  
We first recall the classification, obtained in the Appendix of \cite{jula2} by G\'erardin\footnote{G\'erardin works over $\C$, but his proof works also over $\cF_m$.}, of the solutions of the part of the system (\ref{E-syst}) corresponding to $a=0$.
For $a=0$, the system (\ref{E-syst}) becomes a system  of equations with unknowns $x_{0,1},\ldots,x_{0,d-1}$ (recall that $x_{0,0}=1$), known simply as \emph{E-system}.
The solutions of this system are parametrised by the non-empty subsets $S$ of $\{0,\ldots,d-1\}$.
Define $C_{i,j}:= \zeta_d^{ij}$, for $0 \leq i,j \leq d-1$, where $\zeta_d = {\rm exp}(2\pi \sqrt{-1}/d)$.
Note that $C = (C_{i,j})_{0 \leq i,j \leq d-1}$ can be seen as the character table matrix of the cyclic group $\mathbb{Z}/ d\mathbb{Z}$. Then the solution of the $E$-system parametrised by the subset $S$ is given by
$$ x_{0,j} = \frac{1}{|S|} \sum_{i \in S} C_{i,j} =  \frac{1}{|S|} \sum_{i \in S} \zeta_d^{ij} \quad \text{ for } j=0,1,\ldots,d-1.$$

We fix a subset $S\subset\{0,\ldots,d-1\}$ and consider the solution $X_S = \{x_{0,0}, x_{0,1},\ldots,x_{0,d-1}\}$ of the $E$-system parametrised by $S$. Note that $E = \tr(e_1) = \frac{1}{|S|}$.

Now let $a\neq0$. 
The equations of  (\ref{E-syst}) are exactly the same regardless the value of $a$, as long as $a \neq 0$. We have the following linear system of equations (recall that $x_{0,-s}=x_{0,d-s}$ for all $s \in \{0,\ldots,d-1\}$):
$$\left(\begin{array}{lllll}
x_{0,0} & x_{0,d-1} & x_{0,d-2} & \ldots & x_{0,1} \\
x_{0,1} & x_{0,0} & x_{0,d-1} & \ldots & x_{0,2}\\
x_{0,2} & x_{0,1} & x_{0,0} & \ldots & x_{0,3}\\
\vdots & \vdots & \vdots & \ddots &\vdots\\
x_{0,d-1} & x_{0,d-2} & x_{0,d-3} & \ldots & x_{0,0}
\end{array}\right)
\left(\begin{array}{c}
x_{a,0}\\
x_{a,1}\\
x_{a,2}\\
\vdots\\
x_{a,d-1}
\end{array}\right) =
\frac{d}{|S|}
\left(\begin{array}{c}
x_{a,0}\\
x_{a,1}\\
x_{a,2}\\
\vdots\\
x_{a,d-1}
\end{array}\right).
$$
This is equivalent to the system
\begin{equation}\label{linsyst}
\left(  \sum_{i \in S}  \left(\begin{array}{lllll}
C_{i,0}  & C_{i,d-1} & C_{i,d-2}  & \ldots &C_{i,1} \\
C_{i,1}  & C_{i,0}  & C_{i,d-1}  & \ldots & C_{i,2} \\
C_{i,2}  & C_{i,1}  & C_{i,0} & \ldots & C_{i,3} \\
\vdots & \vdots & \vdots & \ddots &\vdots\\
C_{i,d-1}  & C_{i,d-2}  &C_{i,d-3}  & \ldots &  C_{i,0} 
\end{array}\right)\right)
\left(\begin{array}{c}
x_{a,0}\\
x_{a,1}\\
x_{a,2}\\
\vdots\\
x_{a,d-1}
\end{array}\right) =
d
\left(\begin{array}{c}
x_{a,0}\\
x_{a,1}\\
x_{a,2}\\
\vdots\\
x_{a,d-1}
\end{array}\right).
\end{equation}

Fix $i \in \{0,1,\ldots,d-1\}$.
We denote by $A_i$ the matrix
$$ \left(\begin{array}{lllll}
C_{i,0}  & C_{i,d-1} & C_{i,d-2}  & \ldots &C_{i,1} \\
C_{i,1}  & C_{i,0}  & C_{i,d-1}  & \ldots & C_{i,2} \\
C_{i,2}  & C_{i,1}  & C_{i,0} & \ldots & C_{i,3} \\
\vdots & \vdots & \vdots & \ddots &\vdots\\
C_{i,d-1}  & C_{i,d-2}  &C_{i,d-3}  & \ldots &  C_{i,0} 
\end{array}\right)\ .$$
If we take
$$x_{a,j} := C_{i,j} = \zeta_d^{ij} \quad \text{ for  } j=0,1,\ldots,d-1,$$
then
$$A_i\left(\begin{array}{c}
x_{a,0}\\
x_{a,1}\\
x_{a,2}\\
\vdots\\
x_{a,d-1}
\end{array}\right) =
d
\left(\begin{array}{c}
x_{a,0}\\
x_{a,1}\\
x_{a,2}\\
\vdots\\
x_{a,d-1}
\end{array}\right),
$$
while if we take
$$x_{a,j} := C_{i',j} =\zeta_d^{i'j} \quad \text{ for } i' \neq i, \,\,j=0,1,\ldots,d-1,$$
then
$$ 
A_i\left(\begin{array}{c}
x_{a,0}\\
x_{a,1}\\
x_{a,2}\\
\vdots\\
x_{a,d-1}
\end{array}\right) =
\left(\begin{array}{c}
0\\
0\\
0\\
\vdots\\
0
\end{array}\right).$$
Since the matrix $C$ is invertible (as the character table matrix of a finite group), its rows form a linear basis of $\cF_m^d$,
which can be written as $\{(C_{i,0}, C_{i,1},\ldots, C_{i,d-1} )\ |\ i=0,1,\dots,d-1\}$.  
If we denote by $V_{A_i}(d)$ (respectively  $V_{A_i}(0)$)  the eigenspace of $A_i$ with respect to the eigenvalue $d$ (respectively $0$), we have
$$V_{A_i}(d) = {\rm Span}_{\cF_m} (\{(C_{i,0}, C_{i,1},\ldots, C_{i,d-1} )\})$$ 
and
$$V_{A_i}(0) = {\rm Span}_{\cF_m} (\{(C_{i',0}, C_{i',1},\ldots, C_{i',d-1} )\,|\,{i' \neq i}\}).$$
Thus, in particular, we have
$${\rm dim}_{{\cF_m}}V_{A_i}(d) = 1 \quad \text{ and } \quad {\rm dim}_{{\cF_m}}V_{A_i}(0) =d-1,$$
and
$$ \cF_m^d = V_{A_i}(d) \oplus V_{A_i}(0).$$

Now, set $A_S:=\sum_{i\in S}A_i$. The solutions of the linear system (\ref{linsyst}) are the elements of the eigenspace $V_{A_S}(d)$   of $A_S$ with respect to the eigenvalue $d$.
Following the above discussion, it is straightforward to see that
$$V_{A_S}(d) = {\rm Span}_{\cF_m} (\{(C_{i,0}, C_{i,1},\ldots, C_{i,d-1} )\,|\,{i  \in S}\})$$ 
and
$$V_{A_S}(0) = {\rm Span}_{\cF_m} (\{(C_{i',0}, C_{i',1},\ldots, C_{i',d-1} )\,|\,{i' \notin S}\}).$$
Thus, in particular, we have
$${\rm dim}_{{\cF_m}}V_{A_S}(d) = |S| \quad \text{ and } \quad {\rm dim}_{{\cF_m}}V_{A_i}(0) =d-|S|,$$
and
$$ \cF_m^d = V_{A_S}(d) \oplus V_{A_S}(0).$$

To summarise:

\begin{prop}
We have 
$$\tr(ue_n) = \tr(u)\tr(e_n) \quad\ \text{for any $n\geq 1$ and all $u \in {\rm Y}(d,m,n)$\,,}$$
if and only if there exists a non-empty subset $S$ of $\{0,\ldots,d-1\}$ such that
$$ x_{0,j} =\frac{1}{|S|} \sum_{i \in S} C_{i,j} =  \frac{1}{|S|} \sum_{i \in S} \zeta_d^{ij} 
\quad \text{ for } j=0,1,\ldots,d-1,$$
and, for $a\neq 0$,
$$(x_{a,0}, x_{a,1},\ldots,x_{a,d-1}) \in  {\rm Span}_{\cR_m} (\{(C_{i,0}, C_{i,1},\ldots, C_{i,d-1} )\,|\,{i  \in S}\}).$$ 
\end{prop}

\begin{rem} {\rm 
Note that when $S =\{0,\ldots,d-1\}$, we have
$$x_{0,j} =0 \quad \text{ for } j=1,\ldots,d-1\,,$$ 
and
$$V_{A_S}(d) = {\rm Span}_{\cF_m} (\{(C_{i,0}, C_{i,1},\ldots, C_{i,d-1} )\,|\,{i  \in S}\}) = \cF_m^d.$$ 
Thus, for each $a \neq 0$, $(x_{a,0}, x_{a,1},\ldots,x_{a,d-1})$ is an arbitrary vector of $\cR_m^d$.

On the other hand, if we take a solution of the affine $E$-system corresponding to a singleton subset $S=\{i\}\subset\{0,\ldots,d-1\}$, then we have
$$x_{0,1} = \zeta_d^i\,,$$ $$x_{0,j} = x_{0,1}^j \quad \text{ for } j=0,1,\ldots,d-1\,,$$ 
and, for $a\neq 0$,
$$(x_{a,0}, x_{a,1},\ldots,x_{a,d-1}) = \lambda_a \,(x_{0,0}, x_{0,1},\ldots,x_{0,d-1}) \quad \text{ for some } \lambda_a \in {\cR_m}\,.$$ 
}
\end{rem}

\section{The Markov trace with zero parameters}\label{sec-sym}

In this section, we only consider again the cyclotomic Yokonuma--Hecke algebra ${\rm Y}(d,m,n)$, that is, we take $m < \infty$.
We will study the Markov trace on ${\rm Y}(d,m,n)$ with all parameters equal to $0$ and show that it generalises both the canonical symmetrising trace defined by Bremke and Malle \cite{BM,MaMa,GIM} 
on the Ariki--Koike algebra (case $d=1$) and the canonical symmetrising trace defined in \cite{ChPo} on the Yokonuma--Hecke algebra  (case $m=1$).
We will then determine the weights of this Markov trace by expressing them in terms of Schur elements for Ariki--Koike algebras.

\subsection{Markov trace on ${\rm Y}(d,m,n)$ with zero parameters} From now on we denote by $\btau$ the unique Markov trace on ${\rm Y}(d,m,n)$ with parameters $z=0$ and $x_{a,b}=\delta_{a,0}\delta_{b,0}$, for any $a,b\in E_m$, given by Proposition \ref{prop-tau}.

Recall the basis $\mathcal{B}^{\text{Ind}}_{d,m,n}$ of  ${\rm Y}(d,m,n)$ studied in Section \ref{sec-base}. The set of elements $\mathcal{B}^{\text{Ind}}_{d,m,n}$ is defined recursively by $\mathcal{B}^{\text{Ind}}_{d,m,0}:=\{1\}$ and, for $k=1,\dots,n$, by
\begin{equation}\label{Bind1}
\mathcal{B}^{\text{Ind}}_{d,m,k}:=\{W^{(k)}_{J,a,b}\,w\ \ |\ \ J\in\{0,\ldots,k-1\}\,,\ a\in E_m\,,\ b\in\{0,\dots,d-1\}\,,\ w\in \mathcal{B}^{\text{Ind}}_{d,m,k-1}\}\ ,
\end{equation}
where
$$W^{(k)}_{J,a,b}=g_{J}^{-1}\ldots g_2^{-1}g_1^{-1}X_1^at_1^b\,g_1g_2\ldots g_{k-1}.$$ By (\ref{form-tau}) and Proposition \ref{prop-tau}, we have that the Markov trace $\btau$ is given on the basis $\mathcal{B}^{\text{Ind}}_{d,m,n}$ by the following initial condition and recursive formula:
\begin{equation}\label{btau-Bind1}
\btau(1)=1\ \ \ \ \ \ \ \text{and}\ \ \ \ \ \ \ \btau\bigl(W^{(k)}_{J,a,b}\,w\bigr)= \delta_{J,k-1}\delta_{a,0}\delta_{b,0}\,\btau(w)\ \ \ \text{for $k=1,\ldots,n$\,,}
\end{equation}
where $J\in\{0,\ldots,k-1\}$, $a \in E_m$, $b\in\{0,\ldots,d-1\}$ and $w\in\mathcal{B}^{\text{Ind}}_{d,m,k-1}$.

Let $\mathcal{B}_{d,n}$ be an $\cR_m$-basis of the Yokonuma--Hecke algebra ${\rm Y}_m(d,n)$ and recall that, by Theorem \ref{theo-bases}, the set
$$\mathcal{B}^{\text{AK}}_{d,m,n} = \{X_1^{a_1}\ldots X_n^{a_n}\cdot\omega\,|\,a_1,\ldots,a_n\in E_m,\,\omega\in\mathcal{B}_{d,n}\}$$
is also an $\cR_m$-basis of the cyclotomic Yokonuma--Hecke algebra ${\rm}Y(d,m,n)$. Using for $\mathcal{B}_{d,n}$ the canonical basis of the Yokonuma--Hecke algebra ${\rm Y}_m(d,n)$ given by Example \ref{exmp of basis},
we obtain that the set
\begin{equation}\label{symbasis}
\mathcal{B}_n:=\mathcal{B}^{\text{AK,can}}_{d,m,n}= \{X_1^{a_1}\ldots X_n^{a_n}\,t_1^{b_1}\ldots t_n^{b_n}\,g_w\ |\ a_1,\ldots,a_n\in E_m\,,b_1,\dots,b_n\in \{0,\dots,d-1\}\,,\ w\in\mathfrak{S}_n\}
\end{equation}
is an $\cR_m$-basis of ${\rm}Y(d,m,n)$.

\begin{prop}\label{prop-btau}
The Markov trace $\btau$ is given on the basis $\mathcal{B}_n$ by:
\begin{equation}\label{form-btau2}
\btau(X_1^{a_1}\ldots X_n^{a_n}\,t_1^{b_1}\ldots t_n^{b_n}\, g_w)=\left\{\begin{array}{ll}1, & \text{if }\,\,w=1\,\,\text{ and }\,\, a_1=\dots=a_n=b_1=\dots=b_n=0\,;\\[0.2em]
0, & \text{otherwise.}\end{array}\right.
\end{equation}
\end{prop}
\begin{proof}
We will prove (\ref{form-btau2}) by induction on $n$.
For $n=1$, Formula (\ref{form-btau2}) is the same as Formula (\ref{btau-Bind1}).

Let $n>1$. 
Following a standard fact about reduced expressions in the symmetric group $\mathfrak{S}_n$,
any $g_w$, with $w\in\mathfrak{S}_n$, can be written uniquely as $g_{J+1}\dots g_{n-1}\cdot g_{w'}$ for some $J\in\{0,\dots,n-1\}$ and $w'\in\mathfrak{S}_{n-1}$
(for $J=n-1$, $g_w=g_{w'}$). 
By the centrality of $\btau$ and (\ref{X-X}), 
the left-hand side of Formula (\ref{form-btau2}) 
is equal to
\[\btau(X_n^{a_n}\,t_n^{b_n}\,g_{J+1}\dots g_{n-1}\cdot g_{w'}X_1^{a_1}\ldots X_{n-1}^{a_{n-1}}t_1^{b_1}\ldots t_{n-1}^{b_{n-1}}),\]
where $g_w=g_{J+1}\dots g_{n-1}\cdot g_{w'}$. By the induction hypothesis and the centrality of $\btau$, we have that
\[\btau(g_{w'}X_1^{a_1}\ldots X_{n-1}^{a_{n-1}}t_1^{b_1}\ldots t_{n-1}^{b_{n-1}})=\left\{\begin{array}{ll}1, & \text{if }\,\,w'=1\,\,\text{ and }\,\, a_1=\dots=a_{n-1}=b_1=\dots=b_{n-1}=0\,;\\[0.2em]
0, & \text{otherwise.}\end{array}\right.\]
So it is enough to prove that, for any $a\in E_m$, $b\in\{0,\dots,d-1\}$ and $u \in {\rm}Y(d,m,n-1)$, we have
\begin{equation}\label{form-btau3}
\btau(X_n^{a}\,t_n^{b}\,g_{J+1}\dots g_{n-1}\cdot u)=\delta_{a,0}\,\delta_{b,0}\,\delta_{J,n-1}\btau(u)\ .
\end{equation}

$ $
\paragraph{\underline{Case $a=0$}} If $J=n-1$, then Formula (\ref{form-btau3}) is a particular case of (\ref{btau-Bind1}) (recall that $t_n^b=W^{(n)}_{n-1,0,b}$). If $J<n-1$, then we have, by (\ref{btau-Bind1}),
\[\btau(t_n^{b}\,g_{J+1}\dots g_{n-1}\cdot u)=\btau(g_{J+1}\dots g_{n-1}\cdot t_{n-1}^{b}u)=\btau(W^{(n)}_{J,0,0}\cdot t_{n-1}^{b}u)=0\ .\]

$ $
\paragraph{\underline{Case $a>0$ and $J=n-1$}}
For any $s\in\{1,\dots,m-1\}$, we define
\[K_s:=\text{Span}_{\cR_m}\{W^{(n)}_{J',a',b'}u'\ \,|\,\ J'\in\{0,\dots,n-1\},\ a'\in\{1,\dots,s\},\ b'\in\{0,\dots,d-1\},\ u'\in {\rm}Y(d,m,n-1)\}\ .\]
Due to the condition $a'\in\{1,\dots,s\}$, we have from (\ref{btau-Bind1}) that $\btau(x)=0$ for any $x\in K_s$, for all $s=1,\dots,m-1$. 
Thus, if we prove that 
\begin{equation}\label{K_a}
 X_n^{a}t_n^{b}\,u \in K_a \ ,
 \end{equation}
then we obtain $\btau(X_n^{a}t_n^{b}\,u)=0$, as desired. 

Now, note that, due to  Formulas (\ref{mult-t})--(\ref{mult-X}), we have, for all $s=1,\ldots, m-2$, 
\begin{center}
$x\cdot K_s\subset K_s$ \,\, for any $x\in\{t_1,\dots,t_n,g_1,\dots,g_{n-1}\}$ \,\, and \,\, $X_1\cdot K_s\subset K_{s+1}$\ ,
\end{center}
whence we deduce that
$X_n\cdot K_s\subset K_{s+1}.$
In particular, we have
$X_n^{a-1} \cdot K_1\subset K_{a}.$
So, in order to prove (\ref{K_a}), it is enough to show that $X_nt_n^{b}\,u\in K_1$. 

The assertion $X_nt_n^{b}\,u\in K_1$ follows from the following formula (for $L=n-1$):
\begin{equation}\label{change-basis}
g_{L}\ldots g_2g_1X_1t_1^b\,g_1g_2\ldots g_{n-1}=W^{(n)}_{L,1,b}+\sum_{i=0}^{L-1}\alpha_i\,W^{(n)}_{i,1,b_i}\,u_i\ \ \ \ \ \ \ \text{for any $L\in\{0,\dots,n-1\}$\,,}
\end{equation}
where $\alpha_i\in\cR_m$, $b_i\in\Z$ and $u_i\in{\rm Y}(d,m,n-1)$, for any $i\in\{0,\dots,L-1\}$. Formula (\ref{change-basis}) is trivially satisfied if $L=0$. Assume that $L>0$. Then, using that $g_L=g_L^{-1}+(q-q^{-1})e_L$, we write
\[g_{L}\ldots g_1X_1t_1^b\,g_1\ldots g_{L-1}=g_{L}^{-1}g_{L-1}\ldots g_1X_1t_1^b\,g_1\ldots g_{n-1}+(q-q^{-1})\frac{1}{d}\sum_{s=0}^{d-1}g_{L-1}\ldots g_1X_1t_1^{b-s}g_1\ldots g_{n-1}\,t_L^s\ .\]
Formula (\ref{change-basis}) follows by induction on $L$, using that 
\begin{center}
$g_L^{-1}W^{(n)}_{i,1,b_i}=W^{(n)}_{i,1,b_i}g_{L-1}^{-1}$  \quad for $i=0,\dots,L-2$.
\end{center}

$ $
\paragraph{\underline{Case $a>0$ and $J<n-1$}} We use induction on $a$, the case $a=0$ being already checked. By (\ref{g-X}) and (\ref{noname1}), 
we have respectively that $X_{n}$ commutes with $g_{J+1}\dots g_{n-2}$ and that 
$X_ng_{n-1}=g_{n-1}X_{n-1}+(q-q^{-1})e_{n-1}X_n$.
So we obtain :
$$\begin{array}{rcl}
X_n^{a}\,t_n^{b}\,g_{J+1}\dots g_{n-1}\cdot u &= &X_n^{a-1}t_{n}^bg_{J+1}\dots g_{n-2} \,(g_{n-1}X_{n-1}+(q-q^{-1})e_{n-1}X_n)\,u \\ & &\\
 & = &\displaystyle X_n^{a-1}t_{n}^bg_{J+1}\dots g_{n-1} X_{n-1}u+(q-q^{-1})\frac{1}{d} \sum_{s=0}^{d-1}X_n^a  t_{n}^{b-s}  g_{J+1}\dots g_{n-2}t_{n-1}^s u
\end{array}$$
By the induction hypothesis, we have $\btau(X_n^{a-1}t_{n}^bg_{J+1}\dots g_{n-1}X_{n-1}u)=0$, since $X_{n-1}u\in{{\rm Y}(d,m,n-1)}$. Moreover, using the already proved case 
``$a >0$ and $J=n-1$'', we have $\btau(X_n^a  t_{n}^{b-s}  g_{J+1}\dots g_{n-2}t_{n-1}^s u)=0$ for all $s=0,\ldots,d-1$, since $g_{J+1}\dots g_{n-2}t_{n-1}^su\in{\rm Y}(d,m,n-1)$.
\end{proof}

\begin{rem}\label{rem-btau}{\rm
Note that the unit element of ${\rm Y}(d,m,n)$ belongs to both  $\mathcal{B}^{\text{Ind}}_{d,m,n}$ and  $\mathcal{B}_n$. We have
\[\btau(1)=1\ \ \ \ \ \text{and}\ \ \ \ \ \ \btau(b)=0\ \ \ \ \text{for any $b\in\mathcal{B}^{\text{Ind}}_{d,m,n}\backslash\{1\}$ (respectively for any $b\in\mathcal{B}_n\backslash\{1\}$)\,.}\]
This follows immediately from (\ref{btau-Bind1}) (respectively from (\ref{form-btau2})).\hfill$\triangle$
}\end{rem}

\subsection{Schur elements for $\cF_m{\rm Y}(d,m,n)$}
Recall that $\mathcal{P}(d,m,n)$ denotes the set of all $(d,m)$-partitions of size $n$. 
By Proposition \ref{s-s}, the algebra ${\rm Y}(d,m,n)$ is semisimple over $\cF_m$ and the set $\{V_{\blambda}\}_{\blambda  \in \mathcal{P}(d,m,n)}$  is a complete set of pairwise non-isomorphic irreducible representations of $\cF_m{\rm Y}(d,m,n)$. The algebra $\cF_m{\rm Y}(d,m,n)$ is also split, following the formulas for the representations $V_{\blambda}$ given by Proposition \ref{prop-rep}. We denote by $\chi_{\blambda}$ the character of the irreducible representation $V_{\blambda}$.

We extend $\btau$ linearly over $\cF_m$ to $\cF_m{\rm Y}(d,m,n)$. 
The map $\btau$ is a symmetrising trace on $\cF_m{\rm Y}(d,m,n)$, that is, $\btau$ satisfies the following two conditions: 
\begin{enumerate}[(i)]
\item $\btau(YZ) = \btau(ZY)$ for all $Y,Z \in \cF_m{\rm Y}(d,m,n)$, and \smallbreak
\item the bilinear form $\cF_m{\rm Y}(d,m,n) \times \cF_m{\rm Y}(d,m,n) \rightarrow \cF_m,\,(X,Y) \mapsto \btau(XY)$ is non-degenerate. \smallbreak
\end{enumerate}
Conditon (i) is satisfied since $\btau$ is a Markov trace, while Condition (ii)  is true because, for $q=1$ and $v_l = {\rm exp}(2\pi l\sqrt{-1} /m) $, 
the trace $\btau$ specialises to the canonical  symmetrising trace on the group algebra of 
 ${(\Z /d \Z\times\Z /m \Z) \wr \mathfrak{S}_n}$
over $\cF_m$, and is thus non-degenerate.
We then have that $\btau$ is written uniquely as a linear combination of the irreducible characters of $\cF_m{\rm Y}(d,m,n)$ with non-zero coefficients (``weights'').
We have
$$\btau=\sum_{\blambda \in \mathcal{P}(d,m,n)} \frac{1}{s_{\blambda}} \, \chi_{\blambda}\,,$$
where $s_{\blambda} \in \cF_m$ is called the \emph{Schur element of $V_{\blambda}$ with respect to $\btau$}. 

\begin{rem}\label{cansymtr}
{\rm  The map $\btau$ is known to be a symmetrising trace on ${\rm Y}(d,m,n)$ (defined over $\cR_m$) in cases $d=1$ \cite{MaMa}
and $m=1$ \cite{ChPo}.
In these cases, $\btau$ is called the  ``canonical'' symmetrising trace on ${\rm Y}(d,m,n)$. 
\hfill $\triangle$}
\end{rem}

\subsubsection{Schur elements and primitive idempotents}
For any $\blambda \in \mathcal{P}(d,m,n)$, denote by $d_{\blambda}$ the dimension of the representation $V_{\blambda}$. 
We fix the basis $\{\bv_{_{\cT}}\}$ of $V_{\blambda}$ used in Proposition \ref{prop-rep} and use it to identify $\text{End}_{\cF_m}(V_{\blambda})$ with the matrix algebra ${\rm Mat}_{d_{\blambda}}(\cF_m)$ over  $\cF_m$. Since $\cF_m{\rm Y}(d,m,n)$ is split semisimple, it follows from the Artin--Wedderburn theorem that there exists an isomorphism
\begin{equation}\label{WA}I: \cF_m{\rm Y}(d,m,n) \rightarrow \prod_{\blambda \in   \mathcal{P}(d,m,n)} {\rm Mat}_{d_{\blambda}}(\cF_m).\end{equation}
We write $I_{\blambda}$ for the projection of $I$ onto the $\blambda$-factor, that is,
$$I_{\blambda}: \cF_m{\rm Y}(d,m,n) \twoheadrightarrow {\rm Mat}_{d_{\blambda}}(\cF_m).$$
Let $\cT$ be a standard $(d,m)$-tableau of shape $\blambda$. Since $I$ is an isomorphism, there exists a unique element
 $E_{_{\cT}}$  of $\cF_m{\rm Y}(d,m,n)$ that satisfies:
$$I_{\bmu}(E_{_{\cT}})=
\left\{ \begin{array}{ll}
0 & \text{ if } \blambda \neq \bmu \,;\\
P_{\bv_{_{\cT}}}& \text{ if } \blambda = \bmu\,,
\end{array}\right.$$
where 
$P_{\bv_{_{\cT}}}$ stands for the projection onto $\cF_m\bv_{_{\cT}}$, that is, $P_{\bv_{_{\cT}}}$ is 
the diagonal $d_{\blambda} \times d_{\blambda}$ matrix with coefficient $1$ in the column labelled by $\bv_{_{\cT}}$, and $0$ everywhere else.
Then we have
\begin{equation}\label{schur}
\btau(E_{_{\cT}}) = \frac{1}{s_{\blambda}}.
\end{equation}
We will use the above formula in order to calculate the Schur elements for $\cF_m{\rm Y}(d,m,n)$.  

\vskip 0.2cm
The set $\{P_{\bv_{_{\cT}}}\}$, where $\bv_{_{\cT}}$ runs over the basis vectors of $V_{\blambda}$, is a complete set of pairwise orthogonal primitive idempotents of ${\rm Mat}_{d_{\blambda}}(\cF_m)$. 
Thus, the element  $E_{_{\cT}}$ is a primitive idempotent of $\cF_m{\rm Y}(d,m,n)$ and
the set $\{ E_{_{\cT}}\}$, where $\cT$ runs over the set of standard $(d,m)$-tableaux of size $n$,  is a complete set of pairwise orthogonal primitive idempotents of $\cF_m{\rm Y}(d,m,n)$.

\subsubsection{Formulas for the idempotents $ E_{_{\cT}}$}
Let $\gad=\{\xi_1,\dots,\xi_d\}$ be the set of all $d$-th roots of unity (ordered arbitrarily).  The elements $t_1,\dots,t_n,X_1,\dots,X_n$ are represented by diagonal matrices in the basis $\{\bv_{_{\cT}}\}$ of $V_{\blambda}$ indexed by the standard $(d,m)$-tableaux of shape $\blambda$ (see Formulas (\ref{rep-t}) and (\ref{rep-Xi})). Moreover, the eigenvalues of the set $\{t_1,\dots,t_n,X_1,\dots,X_n\}$ allow to distinguish between all basis vectors $\bv_{_{\cT}}$ of all representations $V_{\blambda}$, with $\blambda \in \mathcal{P}(d,m,n)$.  This follows immediately from the fact that any standard $(d,m)$-tableau is uniquely determined by its sequence of content arrays, see (\ref{T-cont}).
Thus, for $\cT$ a standard $(d,m)$-tableau of size $n$, we can express the primitive idempotent $E_{_{\cT}}$ of $\cF_m{\rm Y}(d,m,n)$ in terms of the elements $t_1,\dots,t_n, X_1,\dots,X_n$ as follows:

For $i=1,\dots,n$, let $\btheta_i$ be the  $(d,m)$-node of $\cT$ with the number $i$ in it. In order to symplify notation, we set $\pos_i:=\pos^{(d)}(\btheta_i)$ and $\cc_i:=\cc(\btheta_i)$ for $i=1,\dots,n$. As the $(d,m)$-tableau $\cT$ is standard, the  $(d,m)$-node $\btheta_n$ is removable. Let ${\mathcal{U}}$ be the standard  $(d,m)$-tableau obtained from $\cT$ by removing the  $(d,m)$-node $\btheta_n$ and let $\bmu_{n-1}$ be the shape of ${\mathcal{U}}$.
The inductive formula for $E_{_{\cT}}$ in terms of the elements $t_1,\dots,t_n, X_1,\dots,X_n$ reads:
\begin{equation}\label{idem-JM}
E_{_{\cT}}=E_{_{{\mathcal{U}}}}\prod_{ \begin{array}{l}\scriptstyle{\btheta\in 
{\mathcal{E}}_+(\bmu_{n-1})}\\\scriptstyle{\cc(\btheta)\neq \cc_n}\end{array}}\hspace{-0.2cm} \frac{X_n-\cc(\btheta)}{\cc_n-\cc(\btheta)}\prod_{ \begin{array}{l}\scriptstyle{\btheta\in 
{\mathcal{E}}_+(\bmu_{n-1})}\\\scriptstyle{\pos^{(d)}(\btheta)\neq \pos_n}\end{array}  }\hspace{-0.2cm}\frac{t_n-\xi_{\pos^{(d)}(\btheta)}}{\xi_{\pos_n}-\xi_{\pos^{(d)}(\btheta)}}\ ,
\end{equation}
with $E_{_{\cT_0}}=1$ for the unique $(d,m)$-tableau $\cT_0$ of size $0$. Note that, due to the commutativity of the elements $t_1,\dots,t_n, X_1,\dots,X_n$, all terms in the above formula commute with each other.
Further, we have
\begin{equation}\label{diag-mat}
t_i E_{_{\cT}} = E_{_{\cT}} t_i = \xi_{\pos_i}E_{_{\cT}} \quad \text{and} \quad X_i E_{_{\cT}} = E_{_{\cT}} X_i =  \cc_i E_{_{\cT}} \ , \quad \text{ for all } i=1,\ldots,n.
\end{equation}

In Formula (\ref{idem-JM}), we consider the idempotent $E_{_{{\mathcal{U}}}}$ of $\cF_m{\rm Y}(d,m,n-1)$ as an element of $\cF_m{\rm Y}(d,m,n)$ thanks to the chain property (\ref{chainproperty}) of the algebras ${\rm Y}(d,m,n)$. In fact,
seeing $E_{_{{\mathcal{U}}}}$ as an element of $\cF_m{\rm Y}(d,m,n)$, we have
\begin{equation}\label{EU-ET}E_{_{{\mathcal{U}}}} = \sum_{ \begin{array}{c}\scriptstyle{\btheta \in 
{\mathcal{E}}_+(\bmu_{n-1})}\end{array}}\hspace{-0.2cm} 
E_{_{\mathcal{U} \cup \{\btheta\}}}\,\,\,,\end{equation}
where, for any $\btheta \in {\mathcal{E}}_+(\bmu)$, $\mathcal{U} \cup \{\btheta\}$ is the standard  $(d,m)$-tableau obtained from  $\mathcal{U}$ by adding the  $(d,m)$-node $\btheta$ with the number $n$ in it. 
We have
$$E_{_{{\mathcal{U}}}} \prod_{ \begin{array}{l}\scriptstyle{\btheta\in 
{\mathcal{E}}_+(\bmu_{n-1})}\\\scriptstyle{\pos^{(d)}(\btheta)\neq \pos_n}\end{array}  }\hspace{-0.2cm}\frac{t_n-\xi_{\pos^{(d)}(\btheta)}}{\xi_{\pos_n}-\xi_{\pos^{(d)}(\btheta)}}
 =  \sum_{ \begin{array}{l}\scriptstyle{\btheta\in 
{\mathcal{E}}_+(\bmu_{n-1})}\\\scriptstyle{\pos(\btheta)=\pos_n}\end{array}  }
E_{_{\mathcal{U}\cup\{\btheta\}}}\ .$$
As moreover $\{\xi_{\pos^{(d)}(\btheta)}\ |\ \btheta\in 
{\mathcal{E}}_+(\bmu)\}=\gad$ (for any $(d,m)$-partition $\bmu$), we deduce that (\ref{idem-JM}) is equivalent to
\begin{equation}\label{idem-JM2}
E_{_{\cT}}=E_{_{{\mathcal{U}}}}\prod_{ \begin{array}{l}\scriptstyle{\btheta\in 
{\mathcal{E}}_+(\bmu_{n-1})}\\\scriptstyle{\cc(\btheta)\neq \cc_n}\\\scriptstyle{\pos^{(d)}(\btheta)=\pos_n}\end{array}}\hspace{-0.2cm} \frac{X_n-\cc(\btheta)}{\cc_n-\cc(\btheta)}\prod_{\begin{array}{l}\scriptstyle{\xi\in 
\gad}\\\scriptstyle{\xi\neq \xi_{\pos_n}}\end{array}  }\hspace{-0.2cm}\frac{t_n-\xi}{\xi_{\pos_n}-\xi}\ .
\end{equation}

Let $\bmu_{i-1}$ be the shape of the standard  $(d,m)$-tableau obtained from $\cT$ by removing the  $(d,m)$-nodes $\btheta_i,\btheta_{i+1},\ldots,\btheta_n$. Repeating the above process for the idempotent $E_{_{{\mathcal{U}}}}$ and so on until we reach the $(d,m)$-tableau $\cT_0$ of size $0$, we obtain
\begin{equation}\label{idem-JM3}
E_{_{\cT}}=\prod_{i=1}^n \, \Biggl( \prod_{ \begin{array}{l}\scriptstyle{\btheta\in 
{\mathcal{E}}_+(\bmu_{i-1})}\\\scriptstyle{\cc(\btheta)\neq \cc_i}\\\scriptstyle{\pos^{(d)}(\btheta)=\pos_i}\end{array}}\hspace{-0.2cm} \frac{X_i-\cc(\btheta)}{\cc_i-\cc(\btheta)}\prod_{\begin{array}{l}\scriptstyle{\xi\in 
\gad}\\\scriptstyle{\xi\neq \xi_{\pos_i}}\end{array}  }\hspace{-0.2cm}\frac{t_i-\xi}{\xi_{\pos_i}-\xi}
\ \Biggr)\ .
\end{equation}
Set
\begin{equation}\label{def-Etp}E_{_{{\mathcal{T}}}}^{\pos}:=\prod_{i=1}^n\!\!\prod_{\begin{array}{l}\scriptstyle{\xi\in 
\gad}\\\scriptstyle{\xi \neq \xi_{\pos_i}}\end{array}  }\hspace{-0.2cm}\frac{t_i-\xi}{\xi_{\pos_i}-\xi}
\ \ \ \ \ \ \ \text{and}\ \ \ \ \ \ \ 
E_{_{{\mathcal{T}}}}^{\cc,i}:=\!\!\!\!\prod_{ \begin{array}{l}\scriptstyle{\btheta\in 
{\mathcal{E}}_+(\bmu_{i-1})}\\\scriptstyle{\cc(\btheta)\neq \cc_i}\\\scriptstyle{\pos^{(d)}(\btheta)=\pos_i}\end{array}}\hspace{-0.2cm} \frac{X_i-\cc(\btheta)}{\cc_i-\cc(\btheta)}\ \ \ \ \ \text{for all $i=1,\dots,n$.}
\end{equation}
Then  Equation (\ref{idem-JM3}) reads:
\begin{equation}\label{form-Et} E_{_{\cT}}=E_{_{\cT}}^{\pos}\,E_{_{\cT}}^{\cc,1}E_{_{\cT}}^{\cc,2}\dots E_{_{\cT}}^{\cc,n}\,.
\end{equation}
The idempotent $E_{_{\cT}}^{\pos}$ determines the $d$-position of each  $(d,m)$-node in $\cT$, while $E_{_{\cT}}^{\cc,i}$ determines the content of the $(d,m)$-node $\btheta_i$, for $i=1,\dots,n$. 

By definition of $E_{_{\cT}}^{\pos}$,  we have that 
\begin{equation}\label{comm-Epos}
t_iE_{_{\cT}}^{\pos}=E_{_{\cT}}^{\pos}t_i=\xi_{\pos_i}E_{_{\cT}}^{\pos}\ \ \ \text{for $i=1,\ldots,n$\,,}
\end{equation}
 and hence,
\begin{equation}\label{multi-Epos}
  e_{i}E_{_{\cT}}^{\pos}=E_{_{\cT}}^{\pos}e_{i}=\left\{ \begin{array}{ll} E_{_{\cT}}^{\pos}, & \text{ if }\,\, \pos_i=\pos_{i+1}\\[0.2em]
 0, & \text{ if }\,\, \pos_i\neq\pos_{i+1}\\ \end{array}\right..
 \end{equation}
 Finally, it is easy to check that 
 \begin{equation}\label{d^n}
 \btau(E_{_{\cT}}^{\pos})=   \prod_{i=1}^n \prod_{ \begin{array}{l}\scriptstyle{\xi\in 
\gad}\\\scriptstyle{\xi\neq \xi_{\pos_i}}\end{array}  }\hspace{-0.2cm}\frac{-\xi}{\xi_{\pos_i}-\xi}\ =
   \prod_{i=1}^n \frac{1}{d} 
  =\frac{1}{d^n}, 
 \end{equation}
since $\prod_{\xi \in \gad\backslash\{1\}}(1-\xi)=d$.

\subsubsection{Calculation of the Schur elements} 

Before we determine the Schur elements for $\cF_m{\rm Y}(d,m,n)$ with respect to $\btau$, we introduce some notation.
 
Recall that the Ariki--Koike algebra ${\rm H}(m,n)$ is the quotient of ${\rm Y}(d,m,n)$ over the relations $t_j=1$, $j=1,\dots,n$ (see Section \ref{sec-def}). The associated surjective homomorphism is denoted by $\pi_{\rm H}$. Recall moreover that ${\rm H}(m,n)$ coincides with ${\rm Y}(1,m,n)$. In order to avoid confusion, from now on, we will denote  by $\og_1,\dots,\og_{N-1},\oX_1$ the generators of the algebra ${\rm H}(m,N)$   (the images of $g_1,\dots,g_{N-1},X_1$ under $\pi_{\rm H}$), for any $N\in\Z_{\geq 0}$. We will also denote by $\og_w$ the image of the element $g_w$ under $\pi_{\rm H}$ for any $w\in\mathfrak{S}_N$.
Finally, we will denote by $\btau^{(0)}_N$ the canonical symmetrising trace on ${\rm H}(m,n)$ (see Remark \ref{cansymtr}). This form was  first constructed  in \cite{BM}.
It was subsequently proved that $\btau^{(0)}_N$ is given by Formula (\ref{form-btau2}) for $d=1$  \cite[Proposition 2.2]{MaMa}  and that it is the unique Markov trace on  ${\rm H}(m,N)$ with all parameters equal to $0$ \cite[Lemma 4.3]{GIM}. Hence, $\btau^{(0)}_N$ coincides with the Markov trace $\btau$ for $d=1$.

Let $\blambda \in \mathcal{P}(d,m,n)$. The $(d,m)$-partition $\blambda=(\blambda^{(1)},\ldots,\blambda^{(dm)})$ can be viewed as a $d$-tuple of $m$-partitions.
For $a=1,\ldots,d$, we denote by $\blambda[a]$ the $a$-th $m$-partition in $\blambda$, that is,
$$\blambda[a]:=(\blambda^{((a-1)m+1)},\ldots,\blambda^{(am)}).$$
 Let $n_a$ denote the size of $\blambda[a]$. Let $V_{\blambda[a]}$ be the irreducible representation of the Ariki--Koike algebra $\cF_m{\rm H}(m,n_a)$ labelled by $\blambda[a]$
and let $s_{\blambda[a]}$ be the Schur element of $V_{\blambda[a]}$ with respect to $\btau^{(0)}_{n_a}$. If $\blambda[a]$ is empty, then $s_{\blambda[a]}:=1$. 

The Schur elements $s_{\blambda[a]}$ (with respect to $\btau^{(0)}_{n_a}$) for the Ariki--Koike algebras have been independently obtained by Geck--Iancu--Malle \cite{GIM} and Mathas \cite{Ma}. For their simplest existing formula, the reader should refer to \cite{ChJa}.

\begin{prop}\label{Schur elements}
Let $\blambda \in \mathcal{P}(d,m,n)$.
We have
 \begin{equation}\label{schur-multipartition}
s_{\blambda}= d^n\,s_{\blambda{[1]}} s_{\blambda{[2]}} \cdots s_{\blambda{[d]}}.
\end{equation}
\end{prop}
\begin{proof}
Let $\cT$ be a standard  $(d,m)$-tableau of shape $\blambda$. For all $j=1,\ldots,n$, we set, for brevity, $\pos_j:=\pos^{(d)}(\cT|j)$. In order to facilitate the computation of $\btau(E_{_{\cT}})$, we will assume that 
\begin{equation}\label{cond-pos}\pos_1=\cdots=\pos_{n_1}=1\,,\ \ \ \pos_{n_1+1}=\cdots=\pos_{n_1+n_2}=2\,,\ \ \  
\ldots\ldots\,,\ \ \pos_{n_1+\cdots+n_{d-1}+1}=\cdots=\pos_{n}=d\ .
\end{equation}
Set $r_a:=n_1 + \cdots + n_{a-1}$ for any $a \in \{1,\ldots,d\}$. Recall the definition (\ref{def-Etp}) of the idempotent $E^{\pos}_{_{\cT}}$ and of the elements $E^{\cc,i}_{_{\cT}}$, $i=1,\dots,n$. We set
\[E^{(a)}_{_{\cT}}:=E^{\cc,r_a+1}_{_{\cT}}\dots E^{\cc,r_a+n_a}_{_{\cT}}\ \ \ \ \ \text{for any $a\in\{1,\dots,d\}$,}\]
where, by convention, $E^{(a)}_{_{\cT}}:=1$ if $n_a=0$. Now, Formula (\ref{form-Et}) reads:
\begin{equation}\label{form-Et2} E_{_{\cT}}=E_{_{\cT}}^{\pos}\,E_{_{\cT}}^{(1)}E_{_{\cT}}^{(2)}\dots E_{_{\cT}}^{(d)}\,.
\end{equation}

For any $i,j\in\{1,\dots,n\}$ with $i\leq j$, we denote by $\mathcal{A}^{(i,j)}$ the subalgebra of $\cF_m{\rm Y}(d,m,n)$ generated by $X_i$ and $g_i,g_{i+1},\dots,g_{j-1}$. Note that, due to the assumption (\ref{cond-pos}), the idempotent $E_{_{\cT}}^{\pos}$ commutes with any element of the subalgebra $\mathcal{A}^{(r_a+1,r_a+n_a)}$ for any $a\in\{1,\dots,d\}$. Therefore, the set of elements $E_{_{\cT}}^{\pos}\mathcal{A}^{(r_a+1,r_a+n_a)}:=\{E_{_{\cT}}^{\pos}\,x_a\ |\ x_a\in\mathcal{A}^{(r_a+1,r_a+n_a)}\}$ forms a unital algebra with unit element $E_{_{\cT}}^{\pos}$. Equation (\ref{btau1}) in the following lemma will imply as a particular case the proposition.

\begin{lem}\label{lem-btau}
Let $a\in\{1,\dots,d\}$. There exists an $\cF_m$-algebra isomorphism $\Phi_a$ from the Ariki--Koike algebra $\cF_m{\rm H}(m,n_a)$ to the algebra $E_{_{\cT}}^{\pos}\mathcal{A}^{(r_a+1,r_a+n_a)}$ given by:
\begin{equation}\label{def-phia}
\oX_1\mapsto E_{_{\cT}}^{\pos}X_{r_a+1}\ \ \ \ \ \ \text{and}\ \ \ \ \ \ \og_i\mapsto E_{_{\cT}}^{\pos} g_{r_{a}+i}\ \ \ \text{for all $i=1,\dots,n_a-1$\,.}
\end{equation}
Moreover, for any $x_a\in \mathcal{A}^{(r_a+1,r_a+n_a)}$, we have
\begin{equation}\label{btau1}\btau(E_{_{\cT}}^{\pos}\,E_{_{\cT}}^{(1)}E_{_{\cT}}^{(2)}\dots E_{_{\cT}}^{(a-1)}\,x_a)=\btau(E_{_{\cT}}^{\pos}\,E_{_{\cT}}^{(1)}E_{_{\cT}}^{(2)}\dots E_{_{\cT}}^{(a-1)})\,\btau^{(0)}_{n_a}\Bigl(\Phi_a^{-1}(E_{_{\cT}}^{\pos}x_a)\Bigr)\ .
\end{equation}
\end{lem}

\begin{proof}[Proof of Lemma \ref{lem-btau}]
Let $a\in\{1,\dots,d\}$. First, Formula (\ref{multi-Epos}) together with the assumption (\ref{cond-pos}) implies that
\begin{equation}\label{hecke-quadr}E_{_{\cT}}^{\pos}g_{r_a+i}^2  = E_{_{\cT}}^{\pos}\left(1+(q-q^{-1}) g_{r_a+i} \right)\ \ \ \ \text{for any $i\in\{1,\dots,n_a-1\}$\,.}
\end{equation}
Moreover, by construction, we have
\[E_{_{\cT}}^{\pos}=\sum_{\cT'}E_{_{\cT'}}\ ,\]
where the sum is over all standard $(d,m)$-tableaux $\cT'$ of size $n$ such that $\pos^{(d)}(\cT'|i)=\pos_i$ for all $i=1,\dots,n$. In particular, again due to the assumption (\ref{cond-pos}), we have $\pos^{(d)}(\cT'|i)\neq \pos^{(d)}(\cT'|r_a+1)$ for any $i\leq r_a$, 
and so $\cc(\cT'|r_a+1) = v_{\pos^{(m)}(\cT'|r_a+1)}$.
It follows (see (\ref{diag-mat}))
that $E_{_{\cT'}}X_{r_a+1}\in\{v_1E_{_{\cT'}},\dots,v_mE_{_{\cT'}}\}$ for such standard $(d,m)$-tableaux $\cT'$. Thus, the following relation is satisfied:
\begin{equation}\label{eq-Xma} E_{_{\cT}}^{\pos}(X_{r_a+1}-v_1)\cdots(X_{r_a+1}-v_m)  = 0\ .
\end{equation}
This ends the verification of the fact that the map $\Phi_a$, given by (\ref{def-phia}), induces an homomorphism of $\cF_m$-algebras from $\cF_m{\rm H}(m,n_a)$ to $E_{_{\cT}}^{\pos}\mathcal{A}^{(r_a+1,r_a+n_a)}$. 

From the linear independence of the set of elements $\mathcal{B}_n$ (see (\ref{symbasis}) for the definition of $\mathcal{B}_n$), it is immediate that the following elements are linearly independent elements of $E_{_{\cT}}^{\pos}\mathcal{A}^{(r_a+1,r_a+n_a)}$:
\[E_{_{\cT}}^{\pos}X_{r_a+1}^{b_1}\dots X_{r_a+n_a}^{b_{n_a}}g_{w_a}\ ,\]
where $b_1,\dots,b_{n_a}\in E_m$ and $g_{w_a}$ is in the subalgebra generated by $g_{r_a+1},\dots, g_{r_a+n_a-1}$. Thus, the dimension of $E_{_{\cT}}^{\pos}\mathcal{A}^{(r_a+1,r_a+n_a)}$ is at least $m^{n_a}\cdot n_a!\,$, which is the dimension of $\cF_m{\rm H}(m,n_a)$. As, moreover, the homomorphism $\Phi_a$ is obviously surjective, we conclude that the dimension of $E_{_{\cT}}^{\pos}\mathcal{A}^{(r_a+1,r_a+n_a)}$ is equal to $m^{n_a}\cdot n_a!$ and in turn that $\Phi_a$ is an isomorphism.

Now, let $x_a\in \mathcal{A}^{(r_a+1,r_a+n_a)}$. We use the basis $\mathcal{B}_{n_a}$ of $\cF_m{\rm H}(m,n_a)$, given by (\ref{symbasis})  for $d=1$ (and $n=n_a$), and we write
\begin{equation}\label{Phi-Etp-xa}\Phi_a^{-1}(E_{_{\cT}}^{\pos}x_a)=\sum \alpha_{b_1,\dots,b_{n_a},w_a}\oX_1^{b_1}\dots \oX_{n_a}^{b_{n_a}}\og_{w_a}\ ,
\end{equation}
where $ \alpha_{b_1,\dots,b_n,w_a}\in \cF_m$ and the sum is over $b_1,\dots,b_{n_a}\in E_m$ and $w_a\in\mathfrak{S}_{n_a}$. Using (\ref{def-phia}), we obtain
\begin{equation}\label{Etp-xa}E_{_{\cT}}^{\pos}x_a=\sum \alpha_{b_1,\dots,b_{n_a},w_a}E_{_{\cT}}^{\pos}X_{r_a+1}^{b_1}\dots X_{r_a+n_a}^{b_{n_a}}\Phi_a(\og_{w_a})\ ,
\end{equation}
where $\Phi_a(\og_{w_a})$ is a product of $E_{_{\cT}}^{\pos}$ and a word $y_{w_a}$ in the generators $g_{r_a+1},\dots, g_{r_a+n_a-1}$.  
Since $E_{_{\cT}}^{\pos}$ is an idempotent and commutes with $E_{_{\cT}}^{(1)},E_{_{\cT}}^{(2)},\dots ,E_{_{\cT}}^{(a-1)}$, we have
\begin{equation}\label{Et-xa}
E_{_{\cT}}^{\pos}\,E_{_{\cT}}^{(1)}E_{_{\cT}}^{(2)}\dots E_{_{\cT}}^{(a-1)}\,x_a=E_{_{\cT}}^{\pos}\,E_{_{\cT}}^{(1)}E_{_{\cT}}^{(2)}\dots E_{_{\cT}}^{(a-1)}\,\Bigl(\sum \alpha_{b_1,\dots,b_{n_a},w_a}X_{r_a+1}^{b_1}\dots X_{r_a+n_a}^{b_{n_a}}y_{w_a}\Bigr)\ .
\end{equation}
By Proposition \ref{prop-btau}, in order to calculate $\btau(E_{_{\cT}}^{\pos}\,E_{_{\cT}}^{(1)}E_{_{\cT}}^{(2)}\dots E_{_{\cT}}^{(a-1)}\,x_a)$, we have to write the right hand side of (\ref{Et-xa}) as a linear combination of elements of the basis $\mathcal{B}_n$, and pick out the coefficient of the unit element (see Remark \ref{rem-btau}). 
However, the element $E_{_{\cT}}^{(1)}E_{_{\cT}}^{(2)}\dots E_{_{\cT}}^{(a-1)}$ belongs to the subalgebra of $\cF_m{\rm Y}(d,m,n)$ generated by $X_1$ and $g_1,\dots,g_{r_a-1}$  (so, in  particular, it commutes with $X_{r_a+1},\dots,X_{r_a+n_a}$).  Hence, we must have
$$\begin{array}{rcl}
\btau(E_{_{\cT}}^{\pos}\,E_{_{\cT}}^{(1)}E_{_{\cT}}^{(2)}\dots E_{_{\cT}}^{(a-1)}\,x_a)&=&\btau(E_{_{\cT}}^{\pos}\,E_{_{\cT}}^{(1)}E_{_{\cT}}^{(2)}\dots E_{_{\cT}}^{(a-1)})\, 
\btau\Bigl(\sum \alpha_{b_1,\dots,b_{n_a},w_a}X_{r_a+1}^{b_1}\dots X_{r_a+n_a}^{b_{n_a}}y_{w_a}\Bigr) \\ & &\\
&=&\btau(E_{_{\cT}}^{\pos}\,E_{_{\cT}}^{(1)}E_{_{\cT}}^{(2)}\dots E_{_{\cT}}^{(a-1)})\, \alpha_{0,\dots,0,1} \ ,
\end{array}$$
where $\alpha_{0,\dots,0,1}$ is the coefficient of $E_{_{\cT}}^{\pos}$ 
in the decomposition (\ref{Etp-xa}) of $E_{_{\cT}}^{\pos}x_a$. As $\alpha_{0,\dots,0,1}$ is also the coefficient of the unit element of $\cF_m{\rm H}(m,n_a)$ in the decomposition (\ref{Phi-Etp-xa}) of $\Phi^{-1}_a(E_{_{\cT}}^{\pos}x_a)$, Formula (\ref{btau1}) follows.
\end{proof}

We return to the proof  of Proposition \ref{Schur elements}. Let $a\in\{1,\dots,d\}$. Note that $E_{_{\cT}}^{(a)}\in\mathcal{A}^{(r_a+1,r_a+n_a)}$. Let $\cT_a$ be the standard $m$-tableau of shape $\blambda[a]$ such that $\cc(\cT_a|j)=\cc(\cT|r_a+j)$ for all $j=1,\ldots,n_a$. From Formula (\ref{def-phia}) for the isomorphism $\Phi_a$, the element $\Phi^{-1}_a(E_{_{\cT}}^{\pos}E_{_{\cT}}^{(a)})$ is equal to the primitive idempotent of $\cF_m{\rm H}(m,n_a)$ corresponding to $\cT_a$ (this idempotent is described by (\ref{idem-JM}) for $d=1$ and $n=n_a$). We deduce that
\begin{equation}\label{sa}
\btau^{(0)}_{n_a}\Bigl(\Phi_a^{-1}(E_{_{\cT}}^{\pos}E_{_{\cT}}^{(a)})\Bigr)=\frac{1}{s_{\blambda{[a]}}}\ \ \ \ \text{for any $a\in\{1,\dots,d\}$\,.} 
\end{equation}
Following Equation (\ref{btau1}), we have 
\begin{equation}\label{recursive for btau}
\btau(E_{_{\cT}}^{\pos}\,E_{_{\cT}}^{(1)}E_{_{\cT}}^{(2)}\dots E_{_{\cT}}^{(a-1)}E_{_{\cT}}^{(a)})=\btau(E_{_{\cT}}^{\pos}\,E_{_{\cT}}^{(1)}E_{_{\cT}}^{(2)}\dots E_{_{\cT}}^{(a-1)})\,\cdot \, \frac{1}{s_{\blambda{[a]}}} \ \ \ \ \text{for any $a\in\{1,\dots,d\}$}.
\end{equation}
Using  (\ref{d^n}) and repeatedly (\ref{recursive for btau}) yields
\begin{equation}
\btau(E_{_{\cT}}^{\pos}\,E_{_{\cT}}^{(1)}E_{_{\cT}}^{(2)}\dots E_{_{\cT}}^{(a)})= \frac{1}{d^n\,s_{\blambda{[1]}} s_{\blambda{[2]}} \cdots s_{\blambda{[a]}}}
\ \ \ \ \text{for any $a\in\{1,\dots,d\}$}.
\end{equation}
For $a=d$, the above formula is the desired result.
\end{proof}


\begin{thebibliography}{00}

\bibitem[Ar]{Ar} S.~Ariki, {\em On the semi-simplicity of the Hecke algebra of $(\mathbb{Z}/r\mathbb{Z})\wr \mathfrak{S}_n$}, J.~Algebra {\bf 169} (1994) 216--225.

\bibitem[ArKo]{ArKo} S.~Ariki, K.~Koike, {\em A Hecke algebra of $(\mathbb{Z}/r\mathbb{Z})\wr S_n$ and construction of its irreducible representations}, Adv.~Math.~\textbf{106} (1994) 216--243.

\bibitem[BrMa]{BM} K.~Bremke, G.~Malle, {\em Reduced words and a length function for $G(e, 1, n)$}, Indag.~Math.~{\bf 8} (1997), 453--469.

\bibitem[BMR]{BMR} M.~Brou{\'e}, G.~Malle, R.~Rouquier, \emph{Complex reflection groups, braid groups, Hecke algebras}, J.~Reine Angew.~Math.~{\bf 500} (1998), 127--190.

\bibitem[ChJa]{ChJa} M.~Chlouveraki, N.~Jacon, {\em Schur elements for the Ariki--Koike algebra and applications}, J.~Algebraic Combin.~{\bf 35}, 2 (2012) 291--311.

\bibitem[CJKL]{CJKL} M.~Chlouveraki, J.~Juyumaya, K.~Karvounis, S.~Lambropoulou, \emph{Identifying the invariants for classical knots and links from the Yokonuma--Hecke algebras}, arXiv:1505.06666.

\bibitem[ChLa]{chla} M.~Chlouveraki, S.~Lambropoulou, {\em The Yokonuma--Hecke algebras and the HOMFLYPT polynomial}, J.~Knot Theory Ramifications {\bf 22},  No. 14 (2013), 1350080.

\bibitem[ChPo]{ChPo} M.~Chlouveraki, L.~Poulain d'Andecy, {\em Representation theory of the Yokonuma--Hecke algebra}, Adv.~Math.~{\bf 259} (2014), 134--172.

\bibitem[Dr]{Dr} V.~G.~Drinfeld, \emph{Degenerate affine Hecke algebras and Yangians}, Funct.~Anal.~Appl.~{\bf 20} (1986),  58--60.

\bibitem[GIM]{GIM} M.~Geck, L.~Iancu, G.~Malle, {\em Weights of Markov traces and generic degrees}, Indag.~Math.~{\bf 11} (2000), 379--397.

\bibitem[GeLa]{Gela} M.~Geck, S.~Lambropoulou, {\em Markov traces and knot invariants related to Iwahori--Hecke algebras of type $B$},   	J.~Reine Angew.~Math.~{\bf 482} (1997), 191--213.

\bibitem[Ho]{Ho} P.~Hoefsmit, \emph{Representations of Hecke algebras of finite groups with BN-pairs of classical type}, Ph.D. Thesis, University of British Columbia, 1974.


\bibitem[IsKi]{IsKi} A.~Isaev, A.~Kirillov, \emph{Bethe subalgebras in Hecke algebra and Gaudin models}, Lett. Math. Phys. {\bf 104} (2014), no. 2, 179--193.

\bibitem[IsOg]{IsOg} A.~Isaev,  O.~Ogievetsky, \emph{On Baxterized solutions of reflection equation and integrable chain models}, Nuclear Physics B  {\bf 760} (2007) 167--183.

\bibitem[Jo]{Jo} V.~F.~R.~Jones, {\em Hecke algebra representations of braid groups and link polynomials}, Annals of Math.~{\bf 126} (1987),  no. 2, 335--388.

\bibitem[Ju1]{ju} J.~Juyumaya, {\em Sur les nouveaux g\'en\'erateurs de l'alg\`ebre de Hecke H(G,U,1)}, J.~Algebra {\bf 204} (1998) 49--68.

\bibitem[Ju2]{ju2} J.~Juyumaya, {\em Markov trace on the Yokonuma--Hecke algebra}, J.~Knot Theory Ramifications {\bf 13} (2004) 25--39.
   
\bibitem[JuKa]{juka} J.~Juyumaya, S. Kannan, {\em Braid relations in the Yokonuma--Hecke algebra}, J.~Algebra {\bf 239} (2001) 272--297.   

\bibitem[JuLa1]{jula1} J.~Juyumaya, S.~Lambropoulou, {\em $p$-adic framed braids}, Topology Appl.~{\bf 154} (2007) 1804--1826.

\bibitem[JuLa2]{jula2} J.~Juyumaya, S.~Lambropoulou, {\em $p$-adic  framed braids II}, Adv.~Math.~{\bf 234} (2013) 149--191.

\bibitem[JuLa3]{jula3} J.~Juyumaya, S.~Lambropoulou, {\em An adelic extension of the Jones polynomial}, M. Banagl, D. Vogel (eds.) The mathematics of knots, Contributions
in the Mathematical and Computational Sciences, Vol.~1, Springer.

\bibitem[JuLa4]{jula4} J.~Juyumaya, S.~Lambropoulou, {\em An invariant for singular knots}, J.~Knot Theory Ramifications {\bf 18} (2009), no. 6, 825--840.

\bibitem[JuLa5]{jula5} J.~Juyumaya, S.~Lambropoulou, {\em On the framization of knot algebras}, to appear in New Ideas in Low-dimensional Topology, L. Kauffman, V. Manturov (eds), Series of Knots and Everything, WS.

\bibitem[JuLa6]{jula6} J.~Juyumaya, S.~Lambropoulou, {\em Modular framization of the BMW algebra}, arXiv:1007.0092.
      
\bibitem[KoSm]{KoSm} K.~H.~Ko, L.~Smolinsky, {\em The framed braid group and $3$-manifolds}, Proc.~Amer.~Math.~Soc.~{\bf 115}, No. 2 (1992), 541--551.  

\bibitem[La1]{La1} S.~Lambropoulou, {\em Solid torus links and Hecke algebras of B-type}, Proceedings of the
Conference on Quantum Topology, D. N. Yetter ed., pp. 225--245, World Scientific Press, (1994).

\bibitem[La2]{La2} S.~Lambropoulou, {\em Knot theory related to generalized and cyclotomic Hecke algebras of type B}, J.~Knot Theory Ramifications {\bf 8}(5) (1999) 621--658.

\bibitem[MaMa]{MaMa} G.~Malle, A.~Mathas, \emph{Symmetric cyclotomic Hecke algebras}, J.~Algebra  {\bf 205} (1998) 275--293.

\bibitem[Ma]{Ma}  A.~Mathas,  {\em Matrix units and generic degrees for the Ariki-Koike algebras}, J.~Algebra  {\bf 281} (2004), 695--730.

\bibitem[MTV]{MTV} E.~Mukhin, V.~Tarasov, A.~Varchenko, \emph{Bethe subalgebras of the group algebra of the symmetric group}, Transform.~Groups {\bf 18} (2013) 767--801.


\bibitem[OgPo]{OgPo1} O.~Ogievetsky, L.~Poulain d'Andecy, {\em Induced representations and traces for chains of affine and cyclotomic Hecke algebras}, 
J.~Geom.~Phys.~{\bf 87} (2015), 354--372 .

\bibitem[RaSh]{RaSh} A.~Ram, A.~Shepler, \emph{Classification of graded Hecke algebras for complex reflection groups}, Comment.~Math.~Helv.~\textbf{78} (2003) 308--334.

\bibitem[Th1]{thi} N.~Thiem, {\em Unipotent Hecke algebras: the structure, representation theory, and combinatorics}, Ph.D. Thesis, University of Wisconsin (2004).

\bibitem[Th2]{thi2} N.~Thiem, {\em Unipotent Hecke algebras of \,${\rm GL}_n(\mathbb{F}_q)$}, J.~Algebra {\bf 284} (2005) 559--577. 

\bibitem[Th3]{thi3} N.~Thiem, {\em A skein-like multiplication algorithm for unipotent Hecke algebras}, Trans.~Amer.~Math.~Soc.~{\bf 359}(4) (2007) 1685--1724.

\bibitem[Vi1]{Vi1} M.-F.~Vign\'eras, {\em Pro-$p$-Iwahori Hecke ring and supersingular $\bar{F}_p$-representations}, Math.~Ann.~{\bf 331} (2005) 523--556.

\bibitem[Vi2]{Vi2} M.-F.~Vign\'eras, {\em The pro-$p$-Iwahori Hecke algebra of a reductive $p$-adic group I}, to appear in Compositio mathematica (2015). 

\bibitem[Vi3]{Vi3} M.-F.~Vign\'eras, {\em The pro-$p$-Iwahori Hecke algebra of a reductive $p$-adic group II}, M\"{u}nster J.~Math.~{\bf 7} (2014), 363--379.

\bibitem[Vi4]{Vi4} M.-F.~Vign\'eras, {\em The pro-$p$-Iwahori Hecke algebra of a reductive $p$-adic group III}, to appear  in Journal of the Institute of Mathematics of Jussieu (2015).

\bibitem[WaWa]{WaWa} J.~Wan, W.~Wang, \emph{Modular representations and branching rules for wreath Hecke algebras}, Internat.~Math.~Res.~Notices (2008) Art. ID rnn:128-31.

\bibitem[Yo]{yo} T.~Yokonuma, {\em Sur la structure des anneaux de Hecke d'un groupe de Chevalley fini}, C.~R.~Acad.~Sci.~Paris Ser.~I Math.~{\bf 264}  (1967)  344--347.


\end{thebibliography}
\end{document}